\documentclass[11pt,reqno]{amsart}
\usepackage{graphicx, verbatim,lineno,titletoc}
\usepackage{amssymb,mathrsfs}
\usepackage{adjustbox}
\usepackage{calc}
\usepackage{array}
\usepackage{tikz}
\usepackage[vcentermath, enableskew]{youngtab}
\usetikzlibrary{shapes.multipart,patterns,arrows}
\usepackage{mathtools}
\mathtoolsset{showonlyrefs}

\usepackage[font=small]{caption}

\usepackage[T1]{fontenc} 
\usepackage[english]{babel} 
\usepackage{appendix}

\usepackage[all]{xy}
\usepackage{enumerate}
\usepackage[english]{babel}
\usepackage{multirow}
\usepackage{float}
\usepackage{enumitem}
\usepackage{accents}
\usepackage{bm} %bold math symbols
\usepackage[numbers]{natbib}
\usepackage{hyperref}
\hypersetup{colorlinks}

\usepackage{hyperref}
\hypersetup{colorlinks,linkcolor={red},citecolor={olive},urlcolor={red}}
\usepackage[top=1.5in, left=1.5in, right=1.5in, bottom=1.3in]{geometry}

\DeclareMathOperator*{\argmax}{arg\,max}

\usepackage[font=small]{caption}

\usepackage[T1]{fontenc} 
\usepackage[english]{babel} 
\usepackage{appendix}

\usepackage[all]{xy}
\usepackage{enumerate}
\usepackage[english]{babel}
\usepackage{multirow}
\usepackage{float}
\usepackage{enumitem}
\usepackage{accents}

\usepackage[numbers]{natbib}
\usepackage{hyperref}
\hypersetup{colorlinks}

\newcolumntype{M}[1]{>{\centering\arraybackslash}m{#1}}
\newcolumntype{N}{@{}m{0pt}@{}}

\def\liminf{\mathop{\rm lim\,inf}\limits}
\def\limsup{\mathop{\rm lim\,sup}\limits}
\def\one{\mathbf{1}}
\def\N{\mathbb{N}}
\def\Z{\mathbb{Z}}
\def\R{\mathbb{R}}

\def\E{\mathbb{E}}
\def\e{\mathbf{e}}
\def\P{\mathbb{P}}

\def\L{\mathcal{L}}

\def\eps{\varepsilon}

\def\L{\textup{L}}
\def\NA{\textup{NA}}
\newcommand{\wt}{{\mathrm{wt}}}

\newcommand{\ZZ}{\mathbb{Z}}
\newcommand{\Lnew}{L^{\textrm{new}}_{j}}

\def\x{\mathbf{x}}
\def\y{\mathbf{y}}
\def\w{\mathbf{w}}

\def\ee{\mathbf{e}}
\def\bmu{\boldsymbol{\mu}}

\newcommand{\p}{\mathbf{p}}

\DeclareMathOperator{\Var}{\textnormal{Var}}

\usepackage{theoremref}

\newtheorem{theorem}{Theorem}
\numberwithin{theorem}{section}
\newtheorem{lemma}[theorem]{Lemma}
\newtheorem{prop}[theorem]{Proposition}
\newtheorem{corollary}[theorem]{Corollary}

\theoremstyle{definition}
\newtheorem{definition}[theorem]{Definition}
\newtheorem{example}[theorem]{Example}
\newtheorem{remark}[theorem]{Remark}

\allowdisplaybreaks

\begin{document}
	
	\title[Scaling limit of soliton lengths in a multicolor box-ball system]{Scaling limit of soliton lengths \\ in a multicolor box-ball system}

	\author{Joel Lewis}
	\address{Joel Lewis, Department of Mathematics, The George Washington University, Washington, DC 20052.}
	\email{\texttt{jblewis@gwu.edu}}
	
	\author{Hanbaek Lyu}
	\address{Hanbaek Lyu, Department of Mathematics, University of Wisconsin - Madison, WI 53709}
	\email{\texttt{hlyu@math.wisc.edu}}
	
	\author{Pavlo Pylyavskyy}
	\address{Pavlo Pylyavskyy, Department of Mathematics, University of Minnesota, Minneapolis, MN 55455.}
	\email{\texttt{ppylyavs@umn.edu}}
	
	\author{Arnab Sen}
	\address{Arnab Sen, Department of Mathematics, University of Minnesota, Minneapolis, MN 55455.}
	\email{\texttt{arnab@umn.edu}}
	%\date{\today}
	
	%\subjclass[2010]{37K40, 60F05}

	\begin{abstract}
		The box-ball systems are integrable cellular automata whose long-time behavior is characterized by soliton solutions, with rich connections to other integrable systems such as the Korteweg-de Vries equation. In this paper, we consider a multicolor box-ball system with two types of random initial configurations and obtain sharp scaling limits of the soliton lengths as the system size tends to infinity. We obtain a sharp scaling limit of soliton lengths that turns out to be different from the single color case as established in \cite{levine2020phase}. A large part of our analysis is devoted to studying the associated carrier process, which is a multi-dimensional Markov chain on the orthant, whose excursions and running maxima are closely related to soliton lengths. We establish the sharp scaling of its ruin probabilities, Skorokhod decomposition, strong law of large numbers, and weak diffusive scaling limit to a semimartingale reflecting Brownian motion with explicit parameters. We also establish and utilize complementary descriptions of the soliton lengths and numbers in terms of the modified Greene-Kleitman invariants for the box-ball systems and associated circular exclusion processes. 
	\end{abstract}
	
	%In the permutation model, the system is initialized by a uniformly chosen random permutation of colors $\{1,\cdots, n\}$. We obtain sharp asymptotics of order $\sqrt{n}$ of all top soliton lengths, as well as sharp linear scaling for soliton numbers. In the independence model, the colors of each site in the interval $[1,n]$ are independently drawn from a fixed distribution $\mathbf{p}=(p_{0},p_{1},\cdots,p_{\kappa})$ on $\{0,1,\cdots,\kappa\}$. We show that the maximum ball density $p^{*}=\max(p_{1},\cdots,p_{\kappa})$ compared to the empty box density $p_{0}$ dictates the general phase transition structure for the longest soliton length: Order $\log (n)$ for $p^{*}<p_{0}$, $\sqrt{n}$ for $p^{*}=p_{0}$, and $n$ for $p^{*}>p_{0}$. For the subsequent soliton lengths, we show that they are of order $\log n$ for $p^{*}<p_{0}$  and  $\sqrt{n}$ for $p^{*}=p_{0}$. We uncover that there are two regimes inside the supercritical phase: If $p^{*}=p_{i}$ for a unique $i$, then the subsequent soliton lengths are of order $\log n$ as in the $\kappa=1$ case; otherwise, they are of order $\sqrt{n}$ instead of $\log n$.  This `non-simple supercritical phase' is unique to the multicolor $\kappa\ge 2$ case. 

	\maketitle

	\let\cleardoublepage\clearpage
	\tableofcontents

	\section{Introduction}
	\label{Introduction}
	
	\subsection{The $\kappa$-color BBS}
	The box-ball systems (BBS) are integrable cellular automata in 1+1 dimension whose long-time behavior is characterized by soliton solutions. The \textit{$\kappa$-color BBS} is a cellular automaton on the half-integer lattice $\mathbb{N}$, which we think of as an array of boxes that can fit at most one ball of any of the $\kappa$ colors. At each discrete time $t\ge 0$, the system configuration is given by a coloring $\xi^{(t)}:\mathbb{N}\rightarrow \mathbb{Z}_{\kappa+1}:=\Z/(\kappa+1) \Z=\{0,1,\cdots, \kappa\}$ with finite support, that is, such that $\xi^{(t)}_{x}=0$ for all but finitely many sites $x$. When $\xi_{x}^{(t)}=i$, we say the site $x$ is \textit{empty} at time $t$ if $i=0$ and \textit{occupied with a ball of color $i$} at time $t$ if $1\le i \le \kappa$. To define the time evolution rule, for each $1\le a \le \kappa$, let $K_{a}$ be the operator on the subset $(\mathbb{Z}_{\kappa+1})^{\mathbb{N}}$ of all $(\kappa+1)$-colorings on $\mathbb{N}$ with finite support %\jbl{It is not true that $K$ is defined on this domain, it is only defined on the part of this domain of finite support.} 
	defined as follows:
	\begin{description}
		\item{(i)} Label the balls of color $a$ from left as $a_{1},a_{2},\cdots,a_{m}$. 
		\item{(ii)} Starting from $k=1$ to $m$, successively move ball $a_{k}$ to the leftmost empty site to its right. 
	\end{description}  
	Then the time evolution $(X_{t})_{t\ge 0}$ of the basic $\kappa$-color BBS is given by 
	\begin{equation}\label{eq:BBS_def_nonlocal}
		\xi^{(t+1)} =  K_{1} \circ K_{2}\circ  \cdots \circ K_{\kappa}(\xi^{(t)}) \quad \forall t\ge 0.
	\end{equation}
	A typical 5-color BBS trajectory is shown below.
	\iffalse
	\begin{align*}
		t=0: && 0 0 3 1 2 0 5 1 3 0 0 4 1 1 2 5 2 0 0 3 2 1 1 0 0 0 0 0 0 0 0 0 0 0 0 0 0 0 0 0 0 0 0 0 0 0 0 0\cdots \\
		t=1: && 0 0 0 0 1 3 2 0 1 5 3 0 0 0 1 4 1 5 2 2 0 0 0 3 2 1 1 0 0 0 0 0 0 0 0 0 0 0 0 0 0 0 0 0 0 0 0 0\cdots\\
		t=2: && 0 0 0 0 0 1 0 3 0 2 1 5 3 0 0 1 0 4 1 0 5 2 2 0 0 0 0 3 2 1 1 0 0 0 0 0 0 0 0 0 0 0 0 0 0 0 0 0\cdots\\
		t=3: && 0 0 0 0 0 0 1 0 3 0 0 2 1 5 3 0 1 0 0 4 1 0 0 5 2 2 0 0 0 0 0 3 2 1 1 0 0 0 0 0 0 0 0 0 0 0 0 0\cdots\\
		t=4: && 0 0 0 0 0 0 0 1 0 3 0 0 0 2 1 5 0 3 1 0 0 4 1 0 0 0 5 2 2 0 0 0 0 0 0 3 2 1 1 0 0 0 0 0 0 0 0 0\cdots\\
		t=5: && 0 0 0 0 0 0 0 0 1 0 3 0 0 0 0 2 5 1 0 3 1 0 0 4 1 0 0 0 0 5 2 2 0 0 0 0 0 0 0 3 2 1 1 0 0 0 0 0\cdots\\
		t=6: && 0 0 0 0 0 0 0 0 0 1 0 3 0 0 0 0 2 0 5 1 0 3 1 0 0 4 1 0 0 0 0 0 5 2 2 0 0 0 0 0 0 0 0   3 2 1 1 0 \cdots
	\end{align*}
	\fi
	\begin{align*}
		t=0: &&	\underline{3 2 1} 0 0 0 0 5 1 3 0 0 4 1 1 2 5 2 0 0 00000000  0000 000 0000 000 00 0000 \cdots \\
		t=1: &&	0 0 0 \underline{3 2 1} 0 0 0 1 5 3 0 0 0 1 4 1 {\color{red}5 2 2} 0000 000 0000 000 00000 0000 0000 \cdots \\
		t=2: &&	0 0 0 0 0 0 \underline{3 2 1} 0 1 0 5 3 0 0 1 0 4 1 0 {\color{red}5 2 2} 0000 000 0000 0000  0000 00000 \cdots \\
		t=3: &&	0 0 0 0 0 0 0 0 0 {\color{magenta}3} 0 2 1 1 5 3 0 1 0 0 4 1 0 0 {\color{red}5 2 2} 0000 000 0000  0000 000000 \cdots \\
		t=4: &&	0 0 0 0 0 0 0 0 0 0 {\color{magenta}3} 0 0 0 2 1 5 0 3 1 1 0 4 1 0 0 0 {\color{red}5 2 2}
		0000 000  000000 0  0000\cdots \\
		t=5: &&	0 0 0 0 0 0 0 0 0 0 0 {\color{magenta}3} 0 0 0 0 {\color{purple}2} {\color{orange}5 1} 0 0 {\color{green}3 1} 0 {\color{blue}4 1 1} 0 0 0 {\color{red}5 2 2}  0000  0000 000000 00 \cdots \\
		t = 6: && 0 0 0 0 0 0 0 0 0 0 0 0 {\color{magenta}3} 0 0 0 0 {\color{purple}2} 0 {\color{orange}5 1} 0 0 {\color{green}3 1} 0 0 {\color{blue}4 1 1} 0 0 0 {\color{red}5 2 2} 0 0 0 0 0 0 0 0 0 0000  \cdots \\
		t = 7: && 0 0 0 0 0 0 0 0 0 0 0 0 0 {\color{magenta}3} 0 0 0 0 {\color{purple}2} 0 0 {\color{orange}5 1} 0 0 {\color{green}3 1} 0 0 0 {\color{blue}4 1 1} 0 0 0 {\color{red}5
			2 2} 0000000000  \cdots 
	\end{align*}

	The grounding observation in the $\kappa$-color BBS with finitely many balls of positive colors is that the system eventually decomposes into \textit{solitons}, which are   sequences of consecutive balls of positive and non-increasing colors, whose length and content are preserved by the BBS dynamics in all future steps. For instance, all of the non-increasing consecutive sequences of balls in $\xi^{(6)}$ in the example  (specifically, {\color{magenta}3}, {\color{purple}2}, {\color{orange}51}, {\color{green}31}, {\color{blue}411}, {\color{red}522})  above are solitons and they are preserved in $\xi^{(7)}$ up to their location and will be so in all future configurations. Note that a soliton of length $k$ travels to the right with speed $k$. Therefore, the lengths of solitons in a soliton decomposition must be non-decreasing from left to right.  In the early dynamics, longer solitons can collide into shorter solitons (e.g., $\underline{321}$ during $t=0,1,2$) and undergo a nonlinear interaction. 
	
	The soliton decomposition of the BBS trajectory initialized at $\xi^{(0)}$  can be encoded in a Young diagram $\Lambda=\Lambda(\xi^{(0)})$ having $j^{\text{th}}$ column equal in length to the $j^{\text{th}}$-longest soliton. For instance, the Young diagram corresponding to the soliton decomposition of the instance of the 5-color BBS given before is
	\begin{equation} 
		%\Lambda(X_{0})=\young(12113121,53421,52,3)
		\Lambda(\xi^{(0)})=\yng(8,5,2,1)
		\vspace{0.2cm}
	\end{equation}
	Note that the $i$th row of the Young diagram $\Lambda (\xi^{(0)})$ is precisely the number of solitons of length at least $i$.
	
	\subsection{Overview of main results}
	
	We consider the $\kappa$-color BBS initialized by a random BBS configuration of system size $n$, and analyze the limiting shape of the random Young diagrams as $n$ tends to infinity. We consider two models that we call the `permutation model' and `independence model'. For both models, we denote the $k$th row and column lengths of the Young diagram encoding the soliton decomposition by $\rho_{k}(n)$ and $\lambda_{k}(n)$, respectively,
	
	In the permutation model, the BBS is initialized by a uniformly chosen random permutation $\Sigma^{n}$ of colors $\{1,2,\cdots, n\}$. A classical way of associating a Young diagram to a permutation is via the Robinson-Schensted correspondence (see \cite[Ch. 3.1]{Sagan}). A famous result of Baik, Deift, and Johansson \cite{baik1999distribution} tells us that the row and column lengths of the random Young diagram constructed from $\Sigma^{n}$ via the RS correspondence scale as $\sqrt{n}$. In Theorem \ref{thm:permutation}, we show that for the random Young diagram constructed via BBS, the columns scale as $\sqrt{n}$ but the rows scale as $n$. Namely, 
	\begin{align}\label{eq:permutation_main_summary}
		\rho_{k}(n) \sim \frac{n}{k(k+1)}, \qquad  \lambda_{k}(n)  \sim   \frac{2\sqrt{n}}{\sqrt{k-1}+\sqrt{k}}.
	\end{align}
	While the row lengths in RS-constructed Young diagram are related to the longest increasing subsequences, we show that the row lengths in the BBS-constructed Young diagram are related to the number of \textit{ascents} (Lemma \ref{lemma:GK_invariants}). This will show that the majority of solitons have a length of order $O(1)$. Hence the row and column scalings in \eqref{eq:permutation_main_summary} are consistent.

	In the independence model, which we denote $\xi^{n,\mathbf{p}}$, the colors of the sites in the interval $[1,n]$ are independently drawn from a fixed distribution $\mathbf{p}=(p_{0},p_{1},\cdots,p_{\kappa})$ on $\mathbb{Z}_{\kappa+1}$. Recently, Lyu and Kuniba obtained sharp asymptotics for the row lengths as well as their large deviations principle in this independence model \cite{kuniba2020large}. In Theorems \ref{thm:iid_subcritical}-\ref{thm:iid_supercritical}, we establish a sharp scaling limit for the column lengths for the independence model, as summarized in Table \ref{table:scalings} and as bullet points below. Let $p^{*}:=\max(p_{1},\dots,p_{\kappa})$ denote the density of the maximum positive color and let $r$ denote the multiplicity of $p^{*}$ (i.e., number of $p_{i}$'s such that $p_{i}=p^{*}$ for $i=1,\dots,\kappa$).

	\begin{itemize}[itemsep=0.1cm]

		\item In the subcritical regime ($p_{0}>p^{*}$), top soliton lengths 
		have sharp scaling  $\log_{\theta} n + (r-1)\log_{\theta} \log n +O(1)$, where $\theta=p^{*}/p_{0}$. 
		
		\item In the critical regime ($p_{0}=p^{*}$), $n^{-1/2}\lambda_{1}(n)$ converges weakly to the maximum $L_{1}$-norm of a $\kappa$-dimensional semimartingale reflecting Brownian motion (SRBM).

		\item In the supercritical regime ($p_{0}<p^{*}$), $\lambda_{1}(n)= (p^{*}-p_{0})n+\Theta(\sqrt{n})$. If $r=1$, then all subsequent top solitons are of order $\log n$; If $r\ge 2$, they are of order $\sqrt{n}$. 
		
		\item The fluctuation of $\lambda_{1}(n)$ depends explicitly on a $\kappa$-dimensional SRBM, which arises as the diffusive scaling limit of the associated carrier process. 
	\end{itemize}

	\begin{table*}[h]
		\centering
		\includegraphics[width=1 \linewidth]{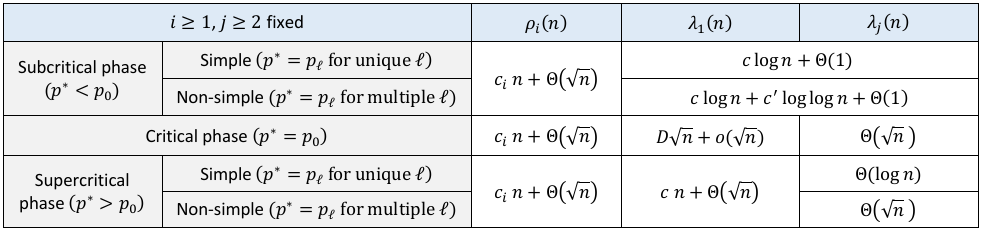}
		\vspace{0cm}
		\caption{ Asymptotic scaling of the $i$th row length $\rho_{i}$ and the $j$th column length $\lambda_{j}$ for the independence model with ball density $\mathbf{p}=(p_{0},p_{1},\cdots,p_{\kappa})$ and $p^{*}=\max(p_{1},\cdots,p_{\kappa})$. The asymptotic soliton lengths undergo a similar `double-jump' phase transition depending on $p^{*}-p_{0}$ as in the $\kappa=1$ case established in \cite{levine2020phase}, but the scaling inside the subcritical and supercritical regimes depends on the multiplicity of the maximum positive color $p^{*}$. 
			Sharp asymptotics for the row lengths have been obtained in \cite{kuniba2020large}. $c_{i}$'s are constants depending on $\mathbf{p}$ and $i$; Constnts $c,c'$ do not depend on $j$; $D$ is a nonnegative and non-degenerate random variable. 
		}
		\label{table:scalings}
	\end{table*}

	We establish a similar `double-jump' phase transition for the $\kappa=1$ case established by Levine, Lyu, and Pike \cite{levine2020phase}. We find that in the multicolor ($\kappa\ge 2$) case, the maximum positive ball density $p^{*}=\max(p_{1},\cdots,p_{\kappa})$ compared to the zero density $p_{0}$ dictates general phase transition structure. However, we find that the scaling inside the subcritical and supercritical regimes depends on the multiplicity $r$ of the maximum positive color $p^{*}$. Furthermore, the fluctuation of the top soliton length $\lambda_{1}(n)$ about its mean behavior is described by a $\kappa$-dimensional semimartingale reflecting Brownian motion (SRBM) lurking behind, whose covariance matrix depends on $\p$ explicitly. Such SRBM arises as the diffusive scaling limit of the associated carrier process.
	
	A large part of our analysis is devoted to studying the associated carrier process, which is a Markov chain on the $\kappa$-dimensional nonnegative integer orthant, whose excursions and running maxima are closely related to soliton lengths (see Lemmas \ref{lemma:queue_formula_soliton}-\ref{lemma:soliton_lengths_excursions}). We establish its sharp scaling of ruin probabilities, strong law of large numbers, and weak diffusive scaling limit to an SRBM with explicit parameters (Theorems \ref{thm:carrier_subcritical}-\ref{thm:SRBM_weak_convergence}). We also establish and utilize alternative descriptions of the soliton lengths and numbers in terms of the modified Greene-Kleitman invariants for the box-ball systems (Lemma \ref{lemma:GK_invariants}) and associated circular exclusion processes.

	\subsection{Background and related works}
	
	The $\kappa$-color BBS was introduced in \cite{takahashi1993some}, generalizing the original $\kappa=1$ BBS first invented by Takahashi and Satsuma in 1990 \cite{takahashi1990soliton}. In the most general form of the BBS, each site accommodates a semistandard tableau of rectangular shape with letters from  $\{0,1,\cdots,\kappa \}$ and the time evolution is defined by successive application of the combinatorial $R$ (cf. \cite{fukuda2000energy, hatayama2001, kuniba2006crystal,inoue2012integrable}). For a friendly introduction to the combinatorial $R$, see \cite[Sec. 3]{kuniba2020large}. The $\kappa$-color BBS treated in this paper corresponds to the case where the tableau shape is a single box, which was called the \textit{basic} $\kappa$-color BBS in \cite{kuniba2020large, kondo2020dynamics}. The BBS is known to arise both from the quantum and classical integrable systems by the procedures called crystallization and ultradiscretization, respectively. This double origin of the integrability of BBS lies behind its deep connections to quantum groups, crystal base theory, solvable lattice models, the Bethe ansatz, soliton equations, ultradiscretization of the Korteweg-de Vries equation, tropical geometry, and so forth; see for example the review \cite{inoue2012integrable} and the references therein. 
	
	BBS with random initial configuration is an emerging topic in the probability literature and has gained considerable attention with a number of recent works \cite{levine2020phase, croydon2018dynamics, kuniba2020large, ferrari2018bbs, kuniba2020large, croydon2019duality, croydon2019invariant}. There are roughly two central questions that the researchers are aiming to answer: 1) If the random initial configuration is one-sided, what is the limiting shape of the invariant random Young diagram as the system size tends to infinity? 2) If one considers the two-sided BBS (where the initial configuration is a bi-directional array of balls), what are the two-sided random initial configurations that are invariant under the BBS dynamics? Some of these questions have been addressed for the basic $1$-color BBS \cite{levine2020phase, ferrari2018soliton, ferrari2018bbs, croydon2018dynamics} as well as for the multicolor case \cite{kuniba2020large, kuniba2018randomized, kondo2020dynamics}. More recently, invariant measures of the discrete KdV and Toda-type systems have been investigated \cite{croydon2020detailed}.

	Three important works are strongly related to this paper. In \cite{levine2020phase}, Levine, Lyu, and Pike studied various soliton statistics of the basic $1$-color BBS when the system is initialized according to a Bernoulli product measure with ball density $p$ on the first $n$ boxes. One of their main results is that the length of the longest soliton is of order $\log n$  for $p<1/2$, order $\sqrt{n}$ for $p=1/2$, 
	and order $n$ for $p>1/2$. Additionally, there is a condensation toward the longest soliton in the supercritical $p>1/2$ regime in the sense that, for each fixed $j\geq 1$, the top $j$ soliton lengths have the same order as the longest for $p\leq 1/2$, whereas all but the 
	longest have order $\log n$ for $p>1/2$. Their analysis is based on geometric mappings from the associated simple random walks to the invariant Young diagrams, which enable a robust analysis of the scaling limit of the invariant Young diagram. However, this connection is not apparent in the general $\kappa\ge 1$ case. In fact, one of the main difficulties in analyzing the soliton lengths in the multicolor BBS is that within a single regime, there is a mixture of behaviors that we see from different regimes in the single-color case.
	
	The row lengths in the multicolor BBS are well-understood due to recent works by Kuniba, Lyu, and Okado \cite{kuniba2018randomized} and Kuniba and Lyu \cite{kuniba2020large}. The central observation is that, when the initial configuration is given by a product measure, then the sum of row lengths can be computed via some additive functional (called `energy') of carrier processes of various shapes, which are finite-state Markov chains whose time evolution is given by combinatorial $R$. In \cite{kuniba2018randomized}, the `stationary shape' of the Young diagram for the most general type of BBS is identified by the logarithmic derivative of a deformed character of the KR modules (or Schur polynomials in the basic case). In \cite{kuniba2020large}, for the (basic) $\kappa$-color BBS that we consider in the present paper, it was shown that the row lengths satisfy a large deviations principle and hence the Young diagram converges to the stationary shape at an exponential rate, in the sense of row scaling. 
	
	The central subject of this paper is the column lengths of the Young diagram for the basic $\kappa$-color BBS. We develop two main tools for our analysis, which are a modified version of Greene-Kleitman invariants for BBS (Section \ref{subsection:GK_invariants}) and the carrier process (see Def. \ref{def:carrier_process}). For the independence model, we obtain the scaling limit of the carrier process as an SRBM \cite{williams1995semimartingale} and it plays a central role in our analysis. For the permutation model, the carrier process gives rise to a `circular exclusion process', which can be regarded as a circular version of the well-known Totally Asymmetric Simple Exclusion Process (TASEP) on a line (see, e.g., \cite{ferrari2018tasep, borodin2007fluctuation, borodin2008transition}). For its rough description, consider the following process on the unit circle $S^{1}$. Starting from some finite number of points, at each time, a new point is added to $S^{1}$ independently from a fixed distribution, which then deletes the nearest counterclockwise point already on the circle. Equivalently, one can think of each point in the circle trying to jump in the clockwise direction. It turns out that this process is crucial in analyzing the permutation model (Section \ref{subsection:permutation_rows}), whereas for the independence model, the relevant circular exclusion process is defined on the integer ring $\mathbb{Z}_{\kappa+1}$ where points can stack up at the same location (Section \ref{subsection:carrier_process_infinite}).  Interestingly, a cylindric version of Schur functions has been used to study rigged configurations and BBS \cite{lam2014rigged}.  
	
	\subsection{Organization}

	In Section \ref{sec:statement}, we define the carrier process, state the permutation and the independence model for the $\kappa$-color BBS, and state our main results.  We also provide numerical simulation to validate our results empirically. In Section \ref{section:key_lemmas}, we introduce infinite and finite capacity carrier processes for the $\kappa$-color BBS and state the three key combinatorial lemmas (Lemmas \ref{lemma:queue_formula_soliton},  \ref{lemma:carrier_row_lengths}, \ref{lemma:GK_invariants}). In Section \ref{section:permutation}, we prove our main result for the permutation model (Theorem \ref{thm:permutation}) by using the modified GK invariants for BBS (Lemma \ref{lemma:GK_invariants}) and analyzing the associated circular exclusion process. In Section \ref{sec:carrier_subcritical}, we prove Theorem \ref{thm:carrier_subcritical} \textbf{(i)} about the stationary behavior of the subcritical carrier process. Next, in Section \ref{sec:decoupled_carrier_main}, we introduce the
	`decoupled carrier process' and develop the `Skorokhod decomposition' of the carrier process. These will play critical roles in the analysis in the following sections. In Section \ref{sec:prob_decoupled_carrier}, we analyze the decoupled carrier process over the i.i.d. ball configuration. In Section \ref{section:subcritical}, we prove Theorem \ref{thm:carrier_subcritical} \textbf{(ii)} and Theorem \ref{thm:iid_subcritical}. In Sections \ref{sec:linear_scaling_carrier} and \ref{sec:carrier_Skorokhod}, we establish a linear and diffusive scaling limit of the carrier process, which is stated in Theorem \ref{thm:SRBM_weak_convergence}. Background on SRBM and an invariance principle for SRBM are also provided in Section \ref{sec:carrier_Skorokhod}. In Section \ref{section:supercri}, we prove Theorems \ref{thm:iid_critical} and \ref{thm:iid_supercritical}.  Lastly, in Section  \ref{section:proof_combinatorial_lemmas} we provide postponed proofs for the combinatorial lemmas stated in Section \ref{section:key_lemmas}.

	\subsection{Notation}\label{section:notation}
	
	We use the convention that summation and product over the empty index set equal zero and one, respectively. For any probability space $(\Omega,\mathcal{F},\mathbb{P})$ and any event $A\in \mathcal{F}$, we let $\mathbf{1}(A)$ denote the indicator variable of $A$. 
	Let $C^{d}(0,\infty)$ denote the space of continuous functions $f:[0,\infty)\rightarrow \R^{d}$ endowed with the topology of uniform convergence on compact intervals.  We let $\textup{tridiagonal}_{d}(a,b,c)$ denote the $d\times d$ matrix which has $a$ on its subdiagonal, $b$ on its diagonal, and $c$ on its superdiagonal entries, and zeros elsewhere. 
	
	We adopt the notations $\R^{+}=[0,\infty)$, $\N=\{1,2,3,\ldots\}$, and $\Z_{\ge 0}=\N\cup\{0\}$ throughout. For a sequence of events $(A_{n})_{n\ge 1}$, we say $A_{n}$ occurs \textit{with high probability} if $\P(A_{n})\rightarrow 1$ as $n\rightarrow \infty$. We employ the Landau notations $O(\cdot),\, \Omega(\cdot),\, \Theta(\cdot)$ in the sense of stochastic boundedness. That is, given 
	$\{a_{n}\}_{n=1}^{\infty}\subset \R^{+}$ and a sequence $\{W_{n}\}_{n=1}^{\infty}$ of nonnegative random variables, we say that 
	$W_{n}=O(a_{n})$ with high probability if for each $\eps>0$, there is a constant $C\in(0,\infty)$ such that $\P( W_{n}<Ca_{n} )\ge 1-\eps$  for all sufficiently large $n$. We say that $W_{n}=\Omega(a_{n})$ if for each $\eps>0$, there is a $c\in(0,\infty)$ such that $\P(W_{n}>ca_{n})\ge 1-\eps$  for all sufficiently large $n$, and we say 
	$W_{n}=\Theta(a_{n})$ with high probability if $W_{n}=O(a_{n})$ and $W_{n}=\Omega(a_{n})$ both with high probability. 
	In all of these Landau notations, we require that the constants $c,C$ do not depend on $n$.

	\section{Statement of results}
	\label{sec:statement}

	Our main results concern the asymptotic behavior of top soliton lengths associated with the $\kappa$-color BBS trajectory for two models of random initial configuration $\xi$: (1) $\kappa=n$ and $\xi[1,n]$ is a random uniform permutation of length $n$; (2) $\kappa$ is fixed and $\xi_{x}=i$ independently with a fixed probability $p_{i}$, $i\in \mathbb{Z}_{\kappa+1}$ for each $x\in [1,n]$.
	
	\subsection{The permutation model}

	For the \textit{permutation model}, let $(U_{x})_{x\ge 1}$ be a sequence of i.i.d. $\textup{Uniform}([0,1])$ random variables. For each integer $n\ge 1$, we denote by $V_{1:n}<V_{2:n}<\cdots<V_{n:n}$ the order statistics of $U_{1},U_{2},\cdots, U_{n}$. Then it is easy to see that the random permutation $\Sigma^{n}$ on $[n]$ such that $V_{i:n}=U_{\Sigma^{n}(i)}$ for all $1\le i \le n$ is uniformly distributed among all permutations on $[n]$. Define 
	\begin{align}\label{eq:def_permutation_model}
		\xi^{n}_{x} := \Sigma^{n}(x) \cdot \mathbf{1}(1\le x \le n).
	\end{align}

	We now state our main result for the permutation model. We obtain a precise first-order asymptotic for the largest $k$ rows and columns, as stated in the following theorem. 
	
	\begin{theorem}[The permutation model]\label{thm:permutation}
		Let $\xi^{n}$ be the permutation model as above. For each $k\ge 1$, denote $\rho_{k}(n)=\rho_{k}(\xi^{n})$ and $\lambda_{k}(n)=\lambda_{k}(\xi^{n})$. Then for each fixed $k\ge 1$, almost surely,
		\begin{align}
			\lim_{n\rightarrow \infty} n^{-1}\rho_{k}(n) &= \frac{1}{k(k+1)}, \qquad \lim_{n\rightarrow\infty} n^{-1/2}\lambda_{k}(n)  =   \frac{2}{\sqrt{k-1}+\sqrt{k}}.
		\end{align}
	\end{theorem}

	Our proof of Theorem \ref{thm:permutation} proceed as follows. We first establish a combinatorial lemma (Lem. \ref{lemma:GK_invariants}) that associates the soliton lengths and numbers with a modified version of Greene-Kleitman invariants for BBS. We then utilize the tail bounds on longest increasing subsequences in uniformly random permutations in  
	Baik, Deift, and Johansson \cite{baik1999distribution} for establishing the scaling limit for the lengths of the columns. For the row lengths, we use the characterization of soliton numbers as an additive functional of finite-capacity carrier processes \cite{kuniba2020large}. Such a process becomes an exclusion process on the unit circle for the permutation model.

	\subsection{The independence model}

	To define the \textit{independence model}, fix integers $n,\kappa\ge 1$. Let $\mathbf{p}=(p_{0},p_{1},\cdots,p_{\kappa})$ be a probability distribution on $\{ 0,1,\cdots,\kappa \}$. Let $\xi=\xi^{\mathbf{p}}$ be the sequence $(\xi_{x})_{x\in \N}$ of i.i.d.  random variables $\xi_{x}$ where 
	\begin{align}
		\P(\xi_{x}=i) = p_{i} \qquad \textup{for $i=0,1,\dots,\kappa$}. 
	\end{align}
	For each integer $n\ge 1$, define $\kappa$-color BBS configuration $\xi^{n,\mathbf{p}}$ of size $n$  by 
	\begin{align}\label{eq:def_iid_model}
		\xi^{n,\mathbf{p}}_{x} &= \xi^{\mathbf{p}}_{x} \cdot \mathbf{1}(1\le x \le n).
	\end{align} 
	We may further assume, without loss of generality, that $p_{i}>0$ for all $1\le i\le \kappa$. Indeed, if $p_{i}=0$ for some $i$, then we can omit the color $i$ entirely and consider the system as a $(\kappa-1)$-color BBS by shifting the colors $\{i+1,\cdots, \kappa\}$ to $\{i,\cdots, \kappa-1 \}$.

	Through various combinatorial lemmas (see Section \ref{section:key_lemmas}), we will establish that the soliton lengths $\lambda_{j}(n)$ of for the i.i.d. model are closely related to the extreme behavior of a Markov chain $(W_{x})_{x\in \N}$ defined on the nonnegative integer orthant $\Z_{\ge 0}^{\kappa}$, which we call the `$\kappa$-color carrier process'. Denote $\e_{i}\in \Z^{\kappa}$ whose coordinates are all zero except the $i$th coordinate being 1.
	
	\begin{figure*}[h]
		\centering
		\includegraphics[width=0.7 \linewidth]{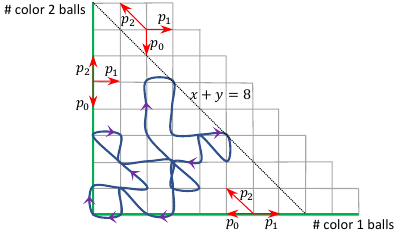}
		\caption{ State space diagram for the carrier process $W_{x}$ for $\kappa=2$. Red arrows illustrate the transition kernel at the `interior' (gray) and `boundary' (green) points in the state space. A single excursion (starting and ending at the origin) of `height' $8$ is shown in a blue path with arrows.
		}
		\label{fig:MC_diagram}
	\end{figure*}
	
	\begin{definition}[$\kappa$-color carrier process]\label{def:carrier_process}
		Let $\xi:=(\xi_{x})_{x\in \N}$ be $\kappa$-color ball configuration. 
		The  \textit{($\kappa$-color) carrier process} over $\xi$ is a process $(W_{x})_{x\in \N}$ on the state space $\Omega:=\Z^{\kappa}_{\ge 0}$ defined by the following evolution rule: Denoting $i:=\xi_{x+1}$ if $\xi_{x+1}\in \{1,\dots,\kappa\}$ and $i:=\kappa+1$ if $\xi_{x+1}=0$,
		\begin{align}\label{eq:W_x_recursion}
			W_{x+1} - W_{x} =	
			\begin{cases}
				%\e_{1} & \textup{if $i=1$} \\
				\e_{i} - \mathbf{1}(i_{*}\ne 0) \, \e_{i_{*}} & \textup{if $1\le  i \le \kappa$} \\
				-\mathbf{1}(i_{*}\ne 0) \,  \e_{i_{*}} &  \textup{if $i=\kappa+1$},
			\end{cases}
		\end{align}
		where $	i_{*}:=\sup\{ 1\le j< i \,:\, W_{x}(j)\ge 1 \}$ with the convention $\sup \emptyset = 0$. Unless otherwise mentioned, we take $W_{0}=\mathbf{0}$ and $\xi=\xi^{\p}$ with density $\p=(p_{0},\dots,p_{\kappa})$. 
	\end{definition}

	In words, at location $x$, the carrier holds $W_{x}(i)$ balls of color $i$ for $i=1,\dots,\kappa$. When a new ball of color $1\le \xi_{x+1}\le \kappa$ is inserted into the carrier $W_{x}$, then a ball of the largest available color that is smaller than $\xi_{x}$ is excluded from $W_{x}$; if there is no such ball in $W_{x}$, then no ball is excluded. If $\xi_{x+1}=0$, then no new ball is inserted, and a ball of the largest available color that is smaller than $\xi_{x}$ is excluded from $W_{x}$. The resulting state of the carrier is $W_{x+1}$. We call the transition rule \eqref{eq:W_x_recursion} as the `circular exclusion' (since a ball in the carrier's possession is excluded from the carrier upon the insertion of a new ball according to the circular ordering). One can also view the carrier process as a multi-type queuing system, where $W_{x}$ denotes the state of the queue and $W_{x}(i)$ is the number of jobs of `cyclic hierarchy' $i$ to be processed.

	A large portion of this paper will be devoted to analyzing scaling limits of the carrier process $W_{x}$ over the i.i.d. configuration $\xi^{\p}$. In this case, $W_{x}$ is a Markov chain on the state space of the nonnegative integer orthant $\Omega$. See Figure \ref{fig:MC_diagram} for an illustration.

	Theorem \ref{thm:carrier_subcritical} states the behavior of the carrier process in the subcritical regime $p_0>\max(p_1, \cdots, p_{\kappa})$. Define a function $\pi:\Omega \rightarrow \R$  by 
	\begin{align}\label{eq:def_stationary_distribution_carrier_sub}
		\pi(n_{1},n_{2},\cdots, n_{\kappa}) =  \prod_{i=1}^{\kappa} \left(1-\frac{p_{i}}{p_{0}} \right)  \left(  \frac{p_{i}}{p_{0}} \right)^{n_{i}},
	\end{align}  
	This is a valid probability distribution on $\Omega$ when $p_{0}>\max(p_{1},\cdots,p_{\kappa})$ since 
	\begin{align}
		\sum_{n_{1}=0}^{\infty}\cdots \sum_{n_{\kappa}=0}^{\infty} \prod_{i=1}^{\kappa} \left( \frac{p_{i}}{p_{0}} \right)^{n_{i}} = \prod_{i=1}^{\kappa} \left(1-\frac{p_{i}}{p_{0}} \right)^{-1} \in (0,\infty).
	\end{align}
	Note that $\pi$ is the the product of geometric distributions of means $p_{i}/(p_{0}-p_{i})>0$ for $i=1,\dots,\kappa$.

	\begin{theorem}[The carrier process at the subcritical regime]\label{thm:carrier_subcritical}
		Let $p^{*}:=\max(p_{1},\cdots, p_{\kappa})$ and suppose $p_{0}>p^{*}$. Let $r$ denote the multiplicity of $p^{*}$ (i.e., number of $i$'s in $\{1,\dots,\kappa\}$ s.t. $p_{i}=p^{*}$). 
		
		\begin{description}
			\item[(i)] (Convergence) The carrier process $W_{x}$ is an irreducible,  aperiodic, and positive recurrent Markov chain on $\Z^{\kappa}_{\ge 0}$ with $\pi$ in \eqref{eq:def_stationary_distribution_carrier_sub} as its unique stationary distribution. Thus, writing $d_{TV}$ for the total variation distance and denoting the distribution of $W_{x}$ by $\pi_{x}$, then 
			\begin{align}\label{eq:carrier_process_mixing}
				\lim_{x\rightarrow \infty} d_{TV}(\pi_{x},\pi) = 0.
			\end{align}
			
			\item[(ii)] (Multi-dimensional Gambler's ruin) Let $T_{1}$ denote the first return time of $W_{x}$ to the origin and let $h_{1}:=\max_{0\le x \le T_{1}} \lVert W_{x} \rVert_{1}$. Then for all $N\ge 1$, there exists a constant $\delta>0$ such that 
			\begin{align}\label{eq:carrier_gambler_formula}
				\delta\,   \binom{N+r-1}{r-1}  \left( \frac{p^{*}}{p_{0}} \right)^{N} \, \le\,  	\P( h_{1}\ge N ) \, \le\,  C \binom{N+r-1}{r-1}  \left( \frac{p^{*}}{p_{0}} \right)^{N},
			\end{align}
			where $C=1$ if $r=\kappa$ and $C=\left( \frac{p^{*}}{p^{*}-p^{(2)}}  \right)^{\kappa-r}$ if $r<\kappa$ with $p^{(2)}$ being the second largest value among $p_{1},\dots,p_{\kappa}$.
		\end{description}
	\end{theorem}
	
	By using Theorem \ref{thm:carrier_subcritical}, we establish sharp scaling limit of soliton lengths for the independence model in the subcritical regime, which is stated in Theorem \ref{thm:iid_subcritical} below. (See Section \ref{section:notation} for a precise definition of Landau notations.)
	\begin{theorem}[The independence model -- Subcritical regime] \label{thm:iid_subcritical}
		Fix $\kappa\ge 1$ and let $\xi^{n,\mathbf{p}}$ be as the i.i.d. model above. Denote $\lambda_{j}(n)=\lambda_{j}(\xi^{n,\mathbf{p}})$, $p^{*} := \max_{1\le i \le \kappa} p_{i}$, and  $r:=| \{ 1\le i \le \kappa\,\colon \, p_{i}=p^{*}\}|$.  Suppose $p_{0}>p^{*}$ and denote $\theta:= p^{*}/p_{0}$. Then for each fixed $j\ge 1$, 
		\begin{align}\label{eq:iid_sub_thm0}
			\lambda_{j}(n) = \log_{\theta} n +  (r-1)\log_{\theta} \log n  + \Theta(1).
		\end{align}

		Furthermore, denote	$\nu_{n}:=(1+\delta_{n}) \log_{\theta}\left( \sigma n /(r-1)! \right)$, where $\sigma:= \prod_{i=1}^{\kappa}\left( 1-\frac{p_{i}}{p_{0}}\right)$ and $\delta_{n} := \frac{(r-1) \log  \log_{\theta}\left( \sigma n /(r-1)! \right) + \log(r-1)! }{\log \sigma n / (r-1)!}$. Then for all $x\in\R$, 
		\begin{align}\label{eq:iid_sub_thm1}
			\exp(-\delta \theta^{-x}) 
			&\leq \liminf_{n\rightarrow \infty}\P\left( \lambda_{j}(n) \leq x+\nu_{n} \right)  \\
			&\leq \limsup_{n\rightarrow \infty}\P\left( \lambda_{j}(n) \leq x+\nu_{n} \right)
			\leq \exp\left(-\frac{C}{(r-1)!} \theta^{-(x-1)} \right)\sum_{k=0}^{j-1} \frac{\theta^{-k(x-1)}}{k! (r-1)! },
		\end{align}
		where $\delta>0$, $C\ge 1$  are constants in Theorem \ref{thm:carrier_subcritical}. 
	\end{theorem}

	Next, we turn our attention to the critical and the supercritical regime, where $p_{0} \le  \max(p_{1},\cdots, p_{\kappa})$. In this regime, the carrier process does not have a stationary distribution and we are interested in identifying the limit of the carrier process in the linear and diffusive scales. A natural candidate for the diffusive scaling limit (if it exists) would be the semimartingale reflecting Brownian motion (SRBM) \cite{williams1995semimartingale}, whose definition we recall in Section \ref{sec:carrier_Skorokhod}. Roughly speaking, an SRBM on a domain $S\subseteq \R^{\kappa}$ is a stochastic process $\mathcal{W}$ that admits a Skorokhod-type decomposition 
	\begin{align}
		\mathcal{W} = X + RY,
	\end{align}
	where $X$ is a $\kappa$-dimensional Brownian motion with drift $\theta$, covariance matrix $\Sigma$, and initial distribution $\nu$. The `interior process' $X$ gives the behavior of $\mathcal{W}$ in the interior of $S$. When it is at the boundary of $S$, it is pushed instantaneously toward the interior of $S$ along the direction specified by the `reflection matrix' $R$ and an associated `pushing process' $Y$. We say such $\mathcal{W}$ a SRBM associated with $(S, \theta, \Sigma, R, \nu)$. If $R=I-Q$  for some nonnegative matrix $Q$ with spectral radius less than one, then such $\mathcal{W}$ is unique (pathwise) for possibly degenerate $\Sigma$ when $S=\R^{\kappa}_{\ge 0}$ 
	\cite{harrison1981reflected}. If $\Sigma$ is non-degenerate and $S$ is a polyhedron, a necessary and sufficient condition for the existence and uniqueness of such SRBM is that $R$ is `completely-$\mathcal{S}$' (see Def. \ref{def:completely_S}) \cite{williams1995semimartingale, kang2007invariance}.
	
	A crucial observation for analyzing the carrier process in the critical and supercritical regimes is the following. Of all the $\kappa$ coordinates of $W_{x}$, some have a negative drift and some others do not. We call an integer $1\le i\le \kappa$ an \textit{unstable color} if $p_{i}\ge \max(p_{i+1},\cdots,p_{\kappa},p_{0})$ and a \textit{stable color} otherwise. 
	Since balls of color $i$ can only be excluded by balls of colors in  $\{i+1,\dots,\kappa,0\}$, then the coordinate $W_{x}(i)$ is likely to diminish if the color $i$ is stable but not if $i$ is unstable. Denote the set of all unstable colors by $\mathcal{C}_{u}^{\mathbf{p}}=\{\alpha_{1},\cdots, \alpha_{r}\}$ with $\alpha_{1}<\cdots<\alpha_{r}$ and let  $\mathcal{C}_{s}^{\mathbf{p}}:=\{0,1,\cdots,\kappa\}\setminus \mathcal{C}_{u}^{\mathbf{p}}$ denote the set of stable colors. (See Figure \ref{fig:circular_exclusion_rules} for illustration.) By definition, we have
	\begin{align}\label{eq:ball_density_weakly_decreasing_unstable}
		p_{\alpha_{1}}\ge p_{\alpha_{2}} \ge \dots \ge p_{\alpha_{r}} \ge p_{\alpha_{r+1}}:=p_{0}. 
	\end{align}
	Now, we will construct a new process $X_{x}$, which we call the `decoupled carrier process' (see Section \ref{sec:decoupled_carrier}), that mimics the behavior of $W_{x}$ but the values of $X_{x}$ on the unstable colors are unconstrained and thus can be negative. Since $W_{x}$ is confined in the nonnegative orthant $\Z^{\kappa}_{\ge 0}$ but $X_{x}$ is not, we need to add some correction process to $X_{x}$ that `pushes' it toward the orthant $\Z^{\kappa}_{\ge 0}$ whenever $X_{x}$ has some of its coordinates going to negative. More precisely, in Lemma \ref{lemma:Skorokhod}, we identify a `reflection matrix' $R\in \R^{\kappa\times \kappa}$ and a  `pushing process' $Y_{x}$ on $\Z^{\kappa}$ such that 
	\begin{align}\label{eq:Skorokhod_decomp_identity}
		W_{x} = X_{x} + R Y_{x} \quad \textup{for $x\ge 0$},
	\end{align}
	where $Y_{0}=\mathbf{0}$ and for each $i\in \{1,\dots,\kappa\}$, the $i$th coordinate of $Y_{x}$  is non-decreasing in $x$ and can only increase when $W_{x}(i)=0$. We call the above as a \textit{Skorokhod decomposition of the carrier process}  (Our definition is motivated by the Skorokhod problem, see Def. \ref{def:SP}.) This and the classical invariance principle for SRBM \cite{reiman1988boundary} is the key to establishing the following result on the scaling limit of the carrier process.
	
	\begin{figure*}[h]
		\centering
		\includegraphics[width=1 \linewidth]{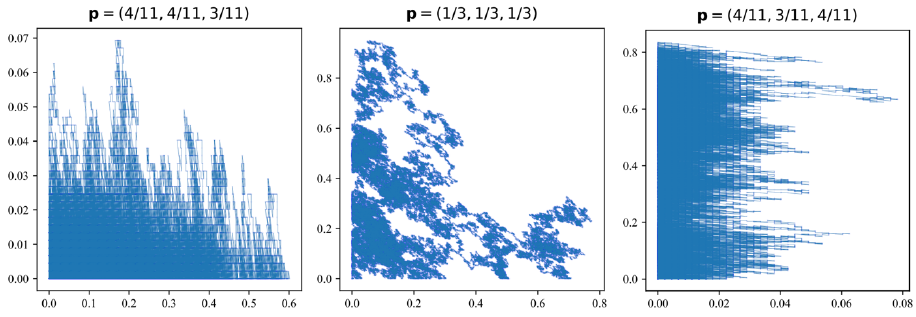}
		\caption{ Simulation of the carrier process $W_{x}$ in diffusive scaling for $\kappa=2$, $n=2\times 10^{5}$, at three critical ball densities  (left) $\mathbf{p}=(4/11,4/11,3/11)$,  (middle) $\mathbf{p}=(1/3,1/3,1/3)$ , and  (right) $\mathbf{p}=(4/11,3/11,4/11)$. In all cases, the process converges weakly to a semimartingale Reflecting Brownian motion on $\R^{2}_{\ge 0}$ whose covariance matrix is non-degenerate in the middle and degenerate in the other two cases.
		}
		\label{fig:SRBM_cri}
	\end{figure*}

	\vspace{-0.3cm}

	\begin{theorem}[Linear and diffusive scaling limit of the carrier process]\label{thm:SRBM_weak_convergence}
		Suppose $p_{0} \le  \max(p_{1},\cdots, p_{\kappa})$. Let $\alpha_{1}<\cdots<\alpha_{r}$ as before and define  
		\begin{align}\label{eq:bmu_def}
			\bmu=(\mu_{1},\dots,\mu_{\kappa}):=\sum_{j=1}^{r} \e_{\alpha_{j}} (p_{\alpha_{j}}- p_{\alpha_{j+1}}),
		\end{align}
		where we let $p_{\alpha_{r+1}}=p_{0}$. 
		\begin{description}
			\item[(i)] (Linear scaling) Almost surely, 
			\begin{align}
				\lim_{x\rightarrow\infty} \, x^{-1} W_{x} =\lim_{x\rightarrow\infty} \, x^{-1} \left( \max_{0\le t \le x}W_{t}(i)\,;\, i=1,\dots,\kappa  \right)= \bmu.
			\end{align}
			
			\item[(ii)] (Diffusive scaling) Let $(\overline{W}_{t})_{t\in \R_{\ge 0}}$ denote the linear interpolation of $(W_{x}-x \bmu)_{x\in \mathbb{N}}$. Then as $n\rightarrow\infty$, 
			\begin{align}
				(x^{-1/2}\overline{W}_{xt} \,;\,  0\leq t \leq 1 ) \Longrightarrow  \mathcal{W} \, \text{ in }\, C([0,1]),
			\end{align}
			where $\mathcal{W}$ is an SRBM associated with data $(S, \mathbf{0}, \Sigma, R, \delta_{\mathbf{0}})$ (see Def. \ref{def:SRBM}) with  $S:=\{ (x_{1},\dots,x_{\kappa})\in \R^{\kappa} \,:\, x_{i}\ge 0 \,\, \textup{if $\mu_{i}=0$}  \}$,  $\Sigma$  the limiting covariance matrix (possibly degenerate) in \eqref{eq:def_limiting_cov_mx}, $R := \textup{tridiag}_{\kappa}(0,1,-1)$, and $\delta_{\mathbf{0}}$ the point mass at $\mathbf{0}$. 
		\end{description}	
	\end{theorem}

	In Figures \ref{fig:SRBM_cri} and \ref{fig:SRBM_sup}, we provide simulations of the carrier process $W_{x}=(W_{x}(1), W_{x}(2))$ for $\kappa=2$ in various regimes, numerically verifying Theorem \ref{thm:SRBM_weak_convergence}. In Figure \ref{fig:SRBM_cri}, we show the carrier process in diffusive scaling ($n^{-1/2}$) at three different critical ball densities $\p$. The carrier process in diffusive scaling converges weakly to an SRBM in $\R^{2}_{\ge 0}$, whose covariance matrix depends on $\p$ and can be degenerate. For instance, at $\p=(4/11,4/11,3/11)$, $W_{x}(2)$ is subcritical (since $p_{2}=3/11<4/11=p_{0}$), and $W_{x}(1)$ is critical, so the SRBM degenerates in the second axes.

	In Figure \ref{fig:SRBM_sup}, we show the carrier process in diffusive scaling at three different supercritical ball densities $\p$. The carrier process has a nonzero drift $\bmu=(\mu_{1},\mu_{2})\in \R^{2}_{\ge 0}$. If $\mu_{1},\mu_{2}>0$, then the centered carrier process $W_{x}-x \bmu$ converges weakly to a 2-dimensional Brownian motion in diffusive scaling. If either $\mu_{1}$ or $\mu_{2}$ equals zero, then the diffusive scaling limit is an SRBM on $\R_{\ge 0} \times \R$ or $\R\times \R_{\ge 0}$, which is the domain $S$ in the statement of Theorem \ref{thm:SRBM_weak_convergence} \textbf{(ii)}. For instance, for $\p=(3/11, 6/11, 2/11)$ as in Figure \ref{fig:SRBM_sup} (d), the SRBM is on domain $S=\R\times \R_{\ge 0}$ and has a degenerate covariance matrix, since $W_{x}(2)$ is subcritical and vanishes in the diffusive scale.

	\begin{figure*}[h]
		\centering
		\includegraphics[width=1 \linewidth]{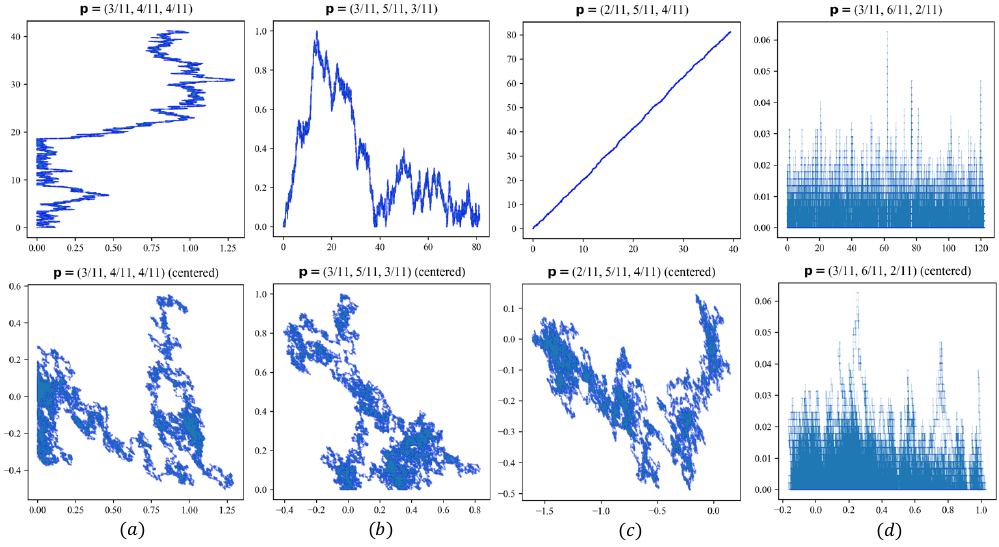}
		\caption{ Simulation of the carrier process $W_{x}$ in diffusive scaling for $\kappa=2$, $n=2\times 10^{5}$, at four supercritical ball densities  $(a)$ $\mathbf{p}=(3/11,4/11,4/11)$,  $(b)$ $\mathbf{p}=(3/11,5/11,3/11)$, $(c)$ $\mathbf{p}=(2/11,5/11,4/11)$, and $(d)$ $\mathbf{p}=(3/11,6/11,2/11)$. The processes grow linearly at least in one dimension (the top row shows uncentered processes in diffusive scaling). As shown in the second row, after centering by the mean drift $\bmu$, the processes converge weakly to semimartingale Reflecting Brownian motion on domains $(a)$ $\R_{\ge 0}\times \R$, $(b)$ $\R\times \R_{\ge 0}$, $(c)$ $\R^{2}$ (no reflection), and $(d)$ $\R\times \R_{\ge 0}$ (with a degenerate covariance matrix). 
		}
		\label{fig:SRBM_sup}
	\end{figure*}

	\vspace{-0.3cm}
	Using the linear and the diffusive scaling limit of the carrier process in Theorem \ref{thm:SRBM_weak_convergence}, we obtain a sharp scaling limit of soliton lengths for the independence model in the critical and subcritical regimes. These results are stated in Theorems \ref{thm:iid_critical} and \ref{thm:iid_supercritical} below. 
	
	%\vspace{-0.3cm}
	\begin{theorem}[The independence model -- Critical regime] \label{thm:iid_critical}
		%Keep the same notation as in Theorem \ref{thm:iid_subcritical}. 
		Suppose $p^{*}=p_{0}$. Then for each fixed $j\ge 1$,  $\lambda_{j}(n)=\Theta(\sqrt{n})$. Furthermore, let $\Sigma$ be a $\kappa\times \kappa$ covariance matrix defined explicitly in 	\eqref{eq:def_limiting_cov_mx} and $R=\textup{tridiag}_{\kappa\times \kappa}(0,1,-1)$. Let $\mathcal{W}$ be a semimartingale reflecting Brownian motion associated with data $(\R^{\kappa}_{\ge 0}, \mathbf{0}, \Sigma, R, \delta_{\mathbf{0}})$ (see Def. \ref{def:SRBM}). Then as $n\rightarrow\infty$, 
		\begin{align}\label{eq:thm_iid_cri_lambda1}
			n^{-1/2} \lambda_{1}(n) \Longrightarrow \sup\,  \lVert W \rVert_{1},
		\end{align}
		where $\Longrightarrow$ denotes weak convergence. 
	\end{theorem}

	\begin{theorem}[The independence model -- Supercritical regime] \label{thm:iid_supercritical}
		%Keep the same notation as in Theorem \ref{thm:iid_subcritical}.  
		Suppose $p^{*}>p_{0}$.  
		\begin{description}
			\item[(i)] (Top soliton length in the supercritical regime) It holds that 
			\begin{align}
				\lim_{n\rightarrow \infty} n^{-1}\lambda_{1}(n) \overset{a.s.}{=} p^{*}-p_{0} \quad \textup{and} \quad  \lambda_{1}(n)=(p^{*}-p_{0})n + \Theta(\sqrt{n}). 
			\end{align}
			
			More precisely,  let $\alpha_{1}<\dots<\alpha_{r}$ denote the unstable colors and let $\alpha_{r+1}:=0$. Let $\bmu=(\mu_{1},\dots,\mu_{\kappa})$ be as in \eqref{eq:bmu_def} and $J:=\{ i \,:\, \mu_{i}>0 \}$. Let $\mathcal{W}=(\mathcal{W}^{1},\dots,\mathcal{W}^{\kappa})$ denote the SRBM in Theorem \ref{thm:SRBM_weak_convergence} \textup{\textbf{(ii)}}. Then 
			\begin{align}\label{eq:lambda_1_fluctuation_thm}
				\sum_{i=1}^{\kappa} \mathcal{W}^{i}(1) 	& \preceq	\liminf_{n\rightarrow\infty} 	\frac{\lambda_{1}(n)-(p_{*}-p_{0})n}{\sqrt{n }}  \\
				&\preceq	\limsup_{n\rightarrow\infty} 	\frac{\lambda_{1}(n)-(p_{*}-p_{0})n}{\sqrt{n }} \\
				& \preceq  \sum_{j\in J} B^{j}(1) + \sup_{0\le v \le 1} \sum_{j\in \{1,\dots,\kappa\}\setminus J }^{\kappa} \mathcal{W}^{i}(v), 
			\end{align}
			where $\preceq$ denotes stochastic dominance and $B=(B^{1},\dots,B^{\kappa})$ is a Brownian motion in $\R^{\kappa}$ with zero drift and the same covariance matrix with $\mathcal{W}$. 
			
			\medskip
			
			\item[(ii)] (Subsequent soliton lengths in the simple supercritical regime) Suppose  $r=1$. Then for any fixed $j\ge 2$, $\lambda_{j}(n)=\Theta(\log n)$ with high probability.

			\medskip
			
			\item[(iii)] (Subsequent soliton lengths in the non-simple supercritical regime) Suppose $r\ge 2$. Then for any fixed $j\ge 2$, $\lambda_{j}(n)=\Theta(\sqrt{n})$ with high probability,  that is, for each $\eps>0$, there exists constants $c_{1},c_{2}>0$ such that $\liminf_{n\rightarrow\infty} \P(\lambda_{j}(n)/\sqrt{n}\in [c_{1},c_{2}] ) \ge 1-\eps$.
		\end{description}
	\end{theorem}

	Multiple remarks on Theorems \ref{thm:iid_subcritical}-\ref{thm:iid_supercritical} are in order. These results extend the `double-jump' phase transition on soliton lengths for the $\kappa=1$ case established by Levine, Lyu, and Pike \cite{levine2020phase} to the multicolor case. As in the $\kappa=1$ case, we find that there exists three regimes -- subcritical ($\lambda_{1}(n)=\Theta(\log n)$), critical ($\lambda_{1}(n)=\Theta(\sqrt{n})$), and supercritical ($\lambda_{1}(n)=\Theta(n)$) -- depending whether the maximum ball density $p^{*}=\max(p_{1},\dots,p_{\kappa})$ exceeds the empty box density $p_{0}$. However, we find that the scaling behavior of the soliton lengths inside each regime is significantly more nuanced in the multicolor case than in the single-color case. 
	
	In the subcritical regime $p^{*}<p_{0}$, we find all top soliton lengths $\lambda_{j}(n)$ for $j\ge 1$ is concentrated around $\log_{\theta} n + (r-1)\log_{\theta} \log n$, where $\theta=p^{*}/p_{0}$ and $r$ denotes the multiplicity of the maximum positive color $p^{*}$, and the tail of $\lambda_{n}(n)$ has a Gumbel-type tail distribution. While this scaling coincides with that in the $\kappa=1$ case for $r=1$, if $r\ge 2$, then the top solitons are an asymptotically `a tad' longer by $(r-1)\log_{\theta} \log n$, which is caused by the competition between multiple maximal colors. 
	
	In the critical regime $p^{*}=p_{0}$, we find that $\lambda_{1}(n)/\sqrt{n} \Rightarrow D$, where the distirbution of the non-degenerate random variable $D$ depends on a SRBM on the orthant $\R^{\kappa}_{\ge 0}$ with zero drift and an explicit covariance matrix $\Sigma$. This is the same SRBM to which the entire carrier process converges weakly in diffusive scaling as in Theorem \ref{thm:SRBM_weak_convergence}. 
	For instance, if $p^{*}$ is uniquely achieved, then the SRBM $\mathcal{W}$ is degenerate in all but one dimension. In particular, for $\kappa=1$, our result recovers the corresponding result in \cite{levine2020phase}. In general, $\Sigma$ can depend on the entire $\p$, capturing the intertwined interaction between balls of all colors in the multicolor case.
	
	In the supercritical regime $p^{*}>p_{0}$, Theorem \ref{thm:iid_supercritical} shows that $\lambda_{1}(n)/n \rightarrow  p^{*}-p_{0}$ almost surely and the fluctuation of $\lambda_{1}(n)$ about its mean is of order $\sqrt{n}$. While a central limit theorem (CLT) for $\lambda_{1}(n)$ in the supercritical regime was shown in \cite{levine2020phase} for the $\kappa=1$ case, we find in the multicolor case that the distribution of the fluctuation of $\lambda_{1}(n)$ does not always satisfy CLT. More precisely, the following corollary shows that CLT holds for $\lambda_{1}(n)$ if and only if the ball density is strictly decreasing on the unstable colors. (Recall \eqref{eq:ball_density_weakly_decreasing_unstable}.)
	
	\begin{corollary}(Fluctuation of $\lambda_{1}$ in the supercritical regime)\label{cor:iid_supercritical_fluctuation}
		Keep the same setting as in Theorem \ref{thm:iid_subcritical}. Suppose supercritical regime $p^{*}>p_{0}$. Let $\alpha_{1}<\dots<\alpha_{r}$ denote the unstable colors. 
		\begin{description}[itemsep=0.1cm]
			\item[(i)]  Further assume $p_{\alpha_{1}}>\dots > p_{\alpha_{r}}$, Then $\lambda_{1}(n)$ satisfies the following central limit theorem 
			\begin{align}
				\frac{\lambda_{1}(n)-(p_{*}-p_{0})n}{\sqrt{n }} \Longrightarrow N(0, \lVert \Sigma \rVert_{1}),
			\end{align}
			where the limiting distribution is the normal distribution with mean zero and variance $\lVert \Sigma \rVert_{1}$ for $\Sigma$ the covariance matrix in Theorem \ref{thm:SRBM_weak_convergence}. 
			
			\item[(ii)] If $p_{\alpha_{j}}=p_{\alpha_{j+1}}$ for some $1\le j \le r-1$, then 
			\begin{align}\label{eq:super_fluctuation_cor}
				\E\left[ 	\liminf_{n\rightarrow\infty} \frac{\lambda_{1}(n)-(p_{*}-p_{0})n}{\sqrt{n }} \right] >0. 
			\end{align}
			In particular, $\lambda_{1}(n)$ does not satisfy the central limit theorem. 
		\end{description}
	\end{corollary}

	Indeed, suppose $p_{\alpha_{1}}>\dots>p_{\alpha_{r}}$ as in Corollary \ref{cor:iid_supercritical_fluctuation} \textbf{(i)}. Then Theorem \ref{thm:SRBM_weak_convergence} states that $x^{-1/2}(W_{x}-\bmu x)$ converges weakly to the (non-reflecting) Brownian motion in $\R^{\kappa}$ with covariance matrix $\Sigma$. Hence in this case  Theorem \ref{thm:iid_supercritical} \textbf{(i)} immediately implies that 
	\begin{align}
		\frac{\lambda_{1}(n)-(p_{*}-p_{0})n}{\sqrt{n }} \Longrightarrow \sum_{i=1}^{\kappa} B^{i}(1),
	\end{align}
	where $B=(B^{1},\dots,B^{\kappa})$ is a Brownian motion in $\R^{\kappa}$ with zero drift and covariance matrix $\Sigma$ in Theorem \ref{thm:SRBM_weak_convergence}. Since $B(1)$ is a standard normal vector with mean zero and covariance matrix $\Sigma$, the result in Corollary \ref{cor:iid_supercritical_fluctuation} \textbf{(i)} follows. 
	
	If we are in the situation as in Corollary \ref{cor:iid_supercritical_fluctuation} \textbf{(ii)}, then some of the consecutive unstable colors have the same ball density, i.e.,  $p_{\alpha_{j}}=p_{\alpha_{j+1}}$. For every such $\alpha_{j}$, the corresponding coordinate has to remain nonnegative in the limiting SRBM. So in this case, the fluctuation of $\lambda_{1}$ about its mean in the diffusive scaling has a positive expectation. As an example, consider the case $\p=(p_{0},p_{1},p_{2})$ with $p_{1}>p_{2}=p_{0}$ (see Figure \ref{fig:SRBM_sup} (b)). In this case, the limiting SRBM $W=(W^{1},W^{2})$ is on the domain $\R\times \R_{\ge 0}$, so the lower bound $W^{1}(1)+W^{2}(1)$ on the fluctuation in \eqref{eq:lambda_1_fluctuation_thm} has a positive expectation. This can be understood for the following reasons. Since $p_{1}>\max(p_{0},p_{2})$, the number of color $1$ balls in the carrier grows linearly and makes the dominant contribution (of order $n$) to $\lambda_{1}(n)$. However, the number of color $2$ balls in the carrier still contributes to $\lambda_{1}(n)$ by order $\sqrt{n}$ since $p_{2}=p_{0}$. While the fluctuation of the number of color 1 balls around its mean $(p_{1}-p_{0})n$ has mean zero, the contribution of color 2 balls of order $\sqrt{n}$ is only visible in the diffusive scaling and it is almost always of a positive amount. 
	
	Another interesting behavior of the multicolor BBS is the order of subsequent soliton lengths, $\lambda_{j}(n)$ for $j\ge 2$, in the supercritical regime, which depends drastically on the multiplicity $r$ of the maximal ball density $p^{*}$. That is, $\lambda_{j}(n)$ for all $j\ge 2$ is of order $\log n$ if  $r=1$, but  they are of order $\sqrt{n}$ if $r\ge 2$. The former case agrees with the results for the $\kappa=1$ case in \cite{levine2020phase}.  There, it was shown that $\lambda_{2}(n)$  comes from the subexcursions of the carrier process below its running maximum. The height of such subexcursions has exponential tails, so we have order $\log (n)$ as the order of the maximum of $n$ subexponential random variables. However, if $r\ge 2$ in the multicolor case, the discrepancy between the number of balls of two maximal colors is of order $\sqrt{n}$ and contributes to $\lambda_{2}(n)$ (see the proof of Theorem \ref{thm:iid_supercritical} \textbf{(iii)}). We remark that a duality between the subcritical and the supercritical regimes for $\kappa=1$ established in \cite{levine2020phase}, in the sense that $\lambda_{j+1}$ in the superciritcal regime corresponds to $\lambda_{j}$ in the subcritical regime for $j\ge 1$. Our results confirm a similar correspondence still holds asymptotically for the simple ($r=1$) supercritical regime; but $\lambda_{j+1}$ in the non-simple ($r\ge 2$) supercirital regime, corresponds to $\lambda_{j}$ in the critical regime.

	\vspace{-0.3cm}
	\section{Key combinatorial lemmas}
	\label{section:key_lemmas}

	\subsection{Infinite capacity carrier process and soliton lengths}
	\label{subsection:carrier_process_infinite}

	The definition of $\kappa$-color BBS dynamics we gave in the introduction involves the non-local movement of balls. It can instead be defined using a `carrier', which gives a localized characterization of the process and reveals a number of important invariants that fully determine the resulting solitons.  For the simplest case $\kappa=1$, imagine a carrier of infinite capacity sweeps through the time-$t$ configuration $\xi^{(t)}$ from the left, picking up each ball it encounters and depositing a ball into each empty box whenever it can. We will see that after we run this carrier over $\xi^{(t)}$, the resulting configuration is in fact $\xi^{(t+1)}$. Moreover, the maximum number of balls in the carrier during the sweep is in fact the first soliton length $\lambda_{1}$. 
	
	Now we introduce the \textit{infinite-capacity carrier process} and the carrier version of the $\kappa$-color BBS dynamic. Denote 
	\begin{align}
		\mathcal{B}_{\infty} := \left\{ \mathbf{x}\in \{0,1,\cdots ,\kappa\}^{\mathbb{N}} \mid \text{$\mathbf{x}$ is non-increasing and has finite support} \right\},
	\end{align}
	which is the set of `reversed'  semi-standard Young tableaux of shape $1\times \infty$ and letters from $\{0,\dots,\kappa\}$. Namely, an element in this set is an infinite string of letters consisting of finitely many non-increasing nonzero letters followed by an infinite string of zeros. An element $\x$ in $		\mathcal{B}_{\infty}$ describes the state of the infinite-capacity carrier. If the carrier at state $\x$ encounters a new ball of color $y$, it produces a new carrier state $\x'$ and a new ball color $y'$ according to the `circular exclusion rule': \textit{Inserting $y$ into $\x$, $y'$ is the largest letter in $\x$ with $y'<y$, and $\x'$ is obtained by replacing the leftmost letter $y'$ in $\x$ with $y$.} More precisely, define a map $\Psi:\mathcal{B}_{\infty}\times \{0,1,\cdots, \kappa \} \rightarrow \{0,1,\cdots, \kappa \} \times \mathcal{B}_{\infty}$, $(\mathbf{x},y)\mapsto (y',\mathbf{x}')$ by
	\begin{description}
		\item[(i)] Suppose $y\ge 1$ and denote $i^{*}=\min\{ i\ge 1 \mid \mathbf{x}(i)< y \}$. Then $y'=\mathbf{x}(i^{*})$ and 
		\begin{align}
			\mathbf{x}'(i) = \mathbf{x}(i)\mathbf{1}(i\ne i^{*}) +  y\mathbf{1}(i=i^{*}) \qquad \forall i\ge 1.
		\end{align}
		\item[(ii)] Suppose $y=0$. Then $y'=\mathbf{x}(1)=\max(\mathbf{x})$ and
		\begin{align}
			\mathbf{x}'(i) = \mathbf{x}(i+1) \qquad \forall i\ge 1.
		\end{align}
	\end{description}

	\begin{figure*}[h]
		\centering
		\includegraphics[width=0.9 \linewidth]{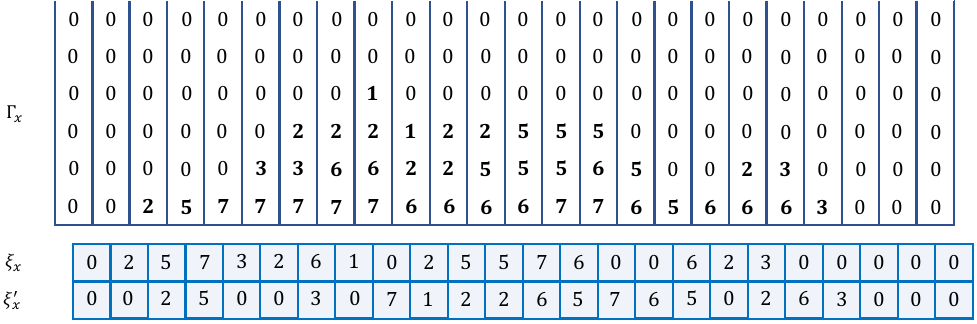}
		\caption{ Time evolution of the infinite capacity carrier process $(\Gamma_{x})_{x\ge 0}$ over the $7$-color initial configuration $\xi$, producing new configuration $\xi'$ consisting of exiting ball colors. For instance, $\xi_{2}=2$, $\Gamma_{2}=[2,0,0,\cdots]$, and $\xi'_{4}=5$. Notice that $\xi'$ can also be obtained by the time evolution of the 7-color BBS applied to $\xi$.
		}
		\label{fig:carrier_process}
	\end{figure*}

	Fix a $\kappa$-color BBS configuration $\xi:\mathbb{N}\rightarrow \{0,1,\cdots, \kappa \}$. Fix $\Gamma_{0} \in \mathcal{B}_{\infty}$, and recursively define a new $\kappa$-color BBS configuration $\xi'$ and a sequence $(\Gamma_{x})_{x\ge 0}$ of elements of $\mathcal{B}_{\infty}$ by  
	\begin{align}\label{eq:def_inf_carrier}
		(\xi'_{x+1},\Gamma_{x+1})  = \Psi(\Gamma_{x}, \xi_{x+1}) \qquad \forall x\in \N.
	\end{align}
	We call the sequence $(\Gamma_{x})_{x\ge 0}$ the \textit{infinite capacity carrier process over $\xi$}. The carrier state $\Gamma_{x}$ is determined by the balls in the interval $[1,x]$ (see Figure \ref{fig:carrier_process} for an illustration). Unless otherwise mentioned, we will assume $\Gamma_{0} =\mathbf{0}=[0,0,0,\cdots]\in \mathcal{B}_{\infty}$. The induced update map $\xi\mapsto \xi'$ turns out to coincide with the $\kappa$-color BBS evolution \eqref{eq:BBS_def_nonlocal}. See Remark \ref{rmk:carrier_BBS_evolution} for more details.

	It is important to note that the carrier process $(W_{x})_{n\in \N}$ we introduced in \eqref{eq:W_x_recursion} can be derived from the infinite-capacity carrier process $(\Gamma_{x})_{x\in \N}$ above by simply recording the number of balls of each color $i=1,\dots,\kappa$. That is, 
	\begin{align}
		W_{x} = (m_{1}(\Gamma_{x}),\dots,m_{\kappa}(\Gamma_{x})) \quad \textup{for all $x\ge 0$}, 
	\end{align}
	where $m_{i}(\Gamma_{x})$ denotes the number of balls of color (letter) $i$ in $\Gamma_{x}$ for $i=1,\dots,\kappa$.

	Lemma~\ref{lemma:queue_formula_soliton} below states that the first soliton length $\lambda_{1}$ equals the maximum number of balls of positive colors in the associated carrier process.

	\begin{lemma}\label{lemma:queue_formula_soliton}
		Suppose the initial $\kappa$-color BBS configuration $\xi$ has finite support. Let $(W_{x})_{x\ge 0}$ and  $(\Gamma_{x})_{x\ge 0}$ be as before. Then 
		\begin{align}
			\lambda_{1}(\xi)= \max_{x\ge 0}  \,\, \lVert W_{x} \rVert_{1} =  \max_{x\ge 0} \left( \text{$\#$ of positive letters in $\Gamma_{x}$} \right).
		\end{align} 	
	\end{lemma}

	For $\kappa=1$, it is possible to precisely characterize all subsequent soliton lengths $\lambda_{2},\lambda_{3},\dots$ by applying the `excursion operator' to the carrier process multiple times and taking maximum \cite{levine2020phase}. Roughly speaking, given the 1-dimensional carrier process $W=(W_{x})_{x\ge 0}$ for $\kappa=1$, which starts at 0 and takes value 0 for all large $x$, let  $\mathcal{E}(W)$ denote the new lattice path that describes the excursion heights above the record minimum of $W$ away from the rightmost global maximizer of $W$. Then $\lambda_{2}=\max(\mathcal{E}(W))$, and $\lambda_{3}=\max (\mathcal{E}^{2}(W))$, and so on. We currently do not have a similar $\kappa$-dimensional excursion operator for exactly describing the subsequent soliton lengths for the general multicolor case. However, we provide a lower bound on $\lambda_{j}$ in terms of the $j$th largest `excursion height' of the carrier process, which is enough to obtain sharp asymptotics for $\lambda_{j}$ in the subcritical regime. 
	
	We introduce some notation. Let $\mathbf{0}=(0,0,\cdots,0)\in (\mathbb{Z}_{\ge 0})^{\kappa}$ denote the origin, and write
	\begin{align}\label{eq:def_Mn}
		M_{n} := \sum_{x=1}^{n}\mathbf{1}(W_{x}=\mathbf{0})
	\end{align}
	for the number of visits of $W_{x}$ to $\mathbf{0}$ during $[1,n]$. For each $k\ge 1$, let $T_{k}$ denote the time of the $k$th visit of  $W_{x}$ to $\mathbf{0}$ and set $T_{0}=0$. We say that the trajectory of $W_{x}$ restricted to the time intervals $[T_{k-1},T_{k}]$ between consecutive visits to $\mathbf{0}$ are its \textit{excursions}. 
	Also note that $M_{n}$ defined at \eqref{eq:def_Mn} equals the number of complete excursions of the carrier process during $[1,n]$. We will define the \textit{height} of the carrier at site $x$ by 
	\begin{align}\label{eq:def_carrier_height_def}
		\lVert W_{x} \rVert_{1} = W_{x}(1) + \dots + W_{x}(\kappa), 
	\end{align}
	which equals the number of balls of positive color that the carrier possesses at site $x$. Define the \textit{$k$th excursion height $h_{k}$} and \textit{height of the final meander $r_{n}$} by 
	\begin{align}\label{eq:def_excursion_heights}
		h_{k} = \max_{T_{k-1}\le t \le T_{k}} \lVert W_{x} \rVert_{1}, \qquad r_{n} = \max_{T_{M_{n}}\le t \le n} \lVert W_{x} \rVert_{1}.
	\end{align}
	Denote by $\mathbf{h}_{1}(n)\ge  \mathbf{h}_{2}(n)\ge \cdots \ge \mathbf{h}_{M_{n}}(n)$ the order statistics of the excursion heights  $h_{1},\cdots,h_{M_{n}}$. We then have the following lemma.
	\begin{lemma}\label{lemma:soliton_lengths_excursions}
		Soliton decomposition of $\xi$ is obtained as the union of the soliton decomposition of the support of each excursion of the carrier process over $\xi$. In particular, for $j,n\ge 1$, $\lambda_{j}(n)\ge \mathbf{h}_{j}(n)$.
	\end{lemma}
	
	Proofs of Lemmas \ref{lemma:queue_formula_soliton} and \ref{lemma:soliton_lengths_excursions} are relagated to  Section~\ref{section:proof_combinatorial_lemmas}.

	\subsection{Finite capacity carrier processes and soliton numbers}
	\label{subsection:carrier_process_finite}

	In \cite{kuniba2020large}, it is shown that the row lengths of the invariant Young diagram of any $\kappa$-BBS trajectory can be extracted by running carrier processes of finite capacities, as we will summarize in this subsection. This will provide one of the key combinatorial lemmas in the present paper. 
	
	First, fix an integer parameter $c\ge 1$ that we call \textit{capacity}. Denote
	\begin{align}
		\mathcal{B}_{c} = \{ [x_{1},\cdots,x_{c}]\in \{0,1,\cdots ,\kappa\}^{c}\mid x_{1}\ge \cdots \ge x_{c} \},
	\end{align}
	which can also be identified as the set of all $(1\times c)$ semistandard tableaux with letters from $\{0,1,\cdots, \kappa \}$. Define a map $\Psi_{c}:\mathcal{B}_{c}\times \{0,1,\cdots, \kappa \} \rightarrow \{0,1,\cdots, \kappa \} \times \mathcal{B}_{c}$, $([x_{1},\cdots,x_{c}],y)\mapsto (y',[x_{1}',\cdots,x_{c}'])$ by the following `circular exclusion rule':
	\begin{description}
		\item[(i)] Suppose $y>x_{c}$ and denote $i^{*}=\min\{ i\ge 1\mid x_{i}< y \}$. Then $y'=x_{i^{*}}$ and 
		\begin{align}
			[x_{1}',\cdots,x'_{c}] = [x_{1},\cdots,x_{i^{*}-1},y,x_{i^{*}+1},\cdots,x_{c}].
		\end{align}
		\item[(ii)] Suppose $x_{c}\ge y$. Then $y'=x_{1}$ and
		\begin{align}
			[x_{1}',x_{2}',\cdots,x_{c}'] = [x_{2},\cdots,x_{c},y].
		\end{align}
	\end{description}
	
	Fix a $\kappa$-color BBS configuration $\xi:\mathbb{N}\rightarrow \{0,1,\cdots, \kappa \}$. Let $\Gamma_{0}=[0,\cdots,0]\in \mathcal{B}_{c}$, and recursively define a new $\kappa$-color BBS configuration $\xi'$ and a sequence $(\Gamma_{x})_{x\ge 0}$ of elements of $\mathcal{B}_{c}$ by  
	\begin{align}
		(\xi'_{x+1},\Gamma_{x+1})  = \Psi_{c}(\Gamma_{x}, \xi_{x+1}) \qquad \forall x\in \N.
	\end{align}
	We call the sequence $(\Gamma_{x})_{x\ge 0}$ the \textit{capacity-$c$ carrier process over $\xi$}. See Figure \ref{fig:carrier_process_finite} for an illustration.

	\begin{figure*}[h]
		\centering
		\includegraphics[width=0.9 \linewidth]{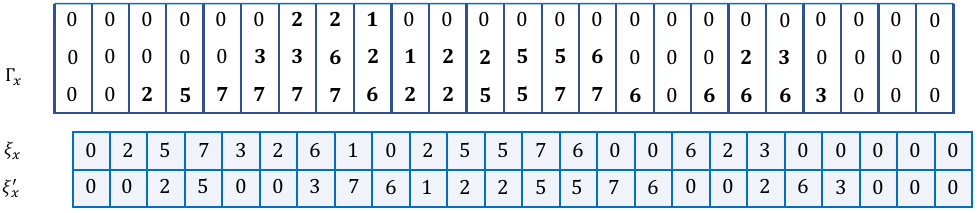}
		\caption{ Time evolution of the capacity-3 carrier process $(\Gamma_{x})_{x\ge 0}$ over the $7$-color initial configuration $\xi$, with new configuration $\xi'$ consisting of exiting ball colors. For instance, $\xi_{2}=2$, $\Gamma_{2}=[2,0,0]$, and $\xi'_{4}=5$. Notice that while $\xi$ is the same as in the example in Figure \ref{fig:carrier_process}, the new $7$-color BBS configuration $\xi'$ is different. In this case, the map $\xi\mapsto \xi'$ does \textit{not} agree with the $7$-color BBS time evolution. 
		}
		\label{fig:carrier_process_finite}
	\end{figure*}

	The following lemma, which is proven in \cite{kuniba2020large}, gives a closed-form expression of the row sums of the invariant Young diagram:
	
	\begin{lemma}\label{lemma:carrier_row_lengths}
		Let $(\xi^{(t)})_{t\ge 0}$ be a $\kappa$-color BBS trajectory such that $\xi^{(0)}$ has finite support. For each $c\ge 1$, let $(\Gamma_{x;c})_{x\ge 0}$ denote the capacity-$c$ carrier process over $\xi^{(t)}$. Then for all $k\ge 1$ and $t\ge 0$, we have 
		\begin{align}
			\rho_{1}(\xi^{(0)}) + \cdots+ \rho_{k}(\xi^{(0)}) \equiv \sum_{x=1}^{\infty} \mathbf{1}(\xi^{(t)}_{x}>\min \Gamma_{x-1;k}),
		\end{align}
		where $\min \Gamma_{x-1;k}$ denotes the smallest letter in $\Gamma_{x-1;k}$.
	\end{lemma}
	
	\begin{proof}
		See eq. (13) and Prop. 4.5 in \cite{kuniba2020large}. We also provide a self-contained proof in Section~\ref{section:pf_lem_row_GK}.
	\end{proof}

	\begin{remark}\label{rmk:carrier_BBS_evolution}
		It is well-known that, if the capacity $c\ge 1$ is large enough compared to the number of balls of color $\ge 1$ in the system, then the induced update map $\xi\mapsto \xi'$ agrees with the $\kappa$-color BBS time evolution (see, e.g., \cite{hatayama2001factorization}). 
		Also, once the capacity $c$ is large enough, the capacity-$c$ carrier process is equivalent to the infinite capacity carrier process in the sense that they always contain the same number of each positive letter. Hence it follows that 
		the map $\xi\mapsto \xi'$ defined in \eqref{eq:def_inf_carrier} coincides with the $\kappa$-color BBS time evolution defined in the introduction. 
		In other words, the $\kappa$-color BBS dynamic can be equivalently defined by repeatedly applying the infinite-capacity carrier process to the current ball configuration, analogously as in the $\kappa=1$ case in \cite{levine2020phase}. 
	\end{remark}
	
	\subsection{Modified Greene-Kleitman invariants for BBS}
	\label{subsection:GK_invariants}
	One natural way to associate a Young diagram with a given permutation is to use the celebrated Robinson-Schensted correspondence (see \cite[Ch. 3.1]{Sagan}), which gives a bijection between permutations and pairs of standard Young tableaux of the same shape. For each permutation $\sigma$, record the common shape of the Young tableaux as $\Lambda_{\textup{RS}}(\sigma)$. Let $\rho_{i}^{\textup{RS}}(\sigma)$ and $\lambda_{j}^{\textup{RS}}(\sigma)$ denote its $i$th row length and its $j$th column lengths, respectively. According to Greene's theorem \cite{greene1982extension}, the sum of the lengths of the first $k$ columns (resp. rows) of $\Lambda_{RS}(\sigma)$ is equal to 
	the length of the longest subsequence in $\sigma$ that can be obtained by taking the union of $k$ decreasing (resp. increasing) subsequences. That is, for each $k\ge 1$, 
	\begin{align}
		\rho_{1}^{\textup{RS}}(\sigma)) + \cdots+ \rho_{k}^{\textup{RS}}(\sigma)) &= \max\left(\left| \text{$\bigsqcup$ $k$ increasing subsequences of $\sigma$} \right| \right), \\
		\lambda_{1}^{\textup{RS}}(\sigma)) + \cdots+ \lambda_{k}^{\textup{RS}}(\sigma)) &= \max\left(\left| \text{$\bigsqcup$ $k$ decreasing subsequences of $\sigma$} \right| \right).
	\end{align}
	The quantities on the right-hand sides are called the \textit{Greene-Kleitman invariants}.
	
	If we consider the $\kappa$-color BBS trajectory started at $\xi^{(0)}=\sigma \mathbf{1}([1,n])$, then we obtain another Young diagram  $\Lambda(\sigma) := \Lambda(\xi^{(0)})$, whose $j^{\text{th}}$ column length equals the $j^{\text{th}}$ longest soliton length. Then a natural question arises: Do the sums of the first $k$ rows and columns of $\Lambda(\sigma)$ relate to some type of Greene-Kleitman invariants? For the rows, we find that the correct modification is to \textit{localize} the length of an increasing sequence into the number of ascents in a subsequence. On the other hand, for the columns, it turns out that we just need to impose that the $k$ decreasing subsequences be \textit{non-interlacing}. In fact, in Lemma \ref{lemma:GK_invariants}, we establish these modified Greene-Kleitman invariants for BBS in the more general setting when $\sigma$ is an arbitrary $\kappa$-color BBS configuration with finite support, where having 0's and repetitions are both allowed. 
	
	Let $\xi:\mathbb{N}\rightarrow \{0,1,\cdots, \kappa\}$ be a $\kappa$-color BBS configuration with finite support. For subsets $A,B\subseteq \mathbb{N}$, denote $A\prec B$ if $\max(A)<\min(B)$. We say $A,B$ are \textit{non-interlacing} if $A\prec B$ or $B \prec A$. We say $\xi$ is \textit{non-increasing} on $A\subseteq \mathbb{N}$  if $\xi_{a_{1}}\ge \xi_{a_{2}}$ for all $a_{1},a_{2}\in A$ such that $a_{1}\le a_{2}$. Denoting the elements of $A$ by $a_{1}<a_{2}<\cdots$, define the \textit{number of ascents of $\xi$ in $A$} by 
	\begin{align}
		\NA(A,\xi):= 1+\sum_{i=2}^{|A|} \mathbf{1}(\xi_{a_{i-1}}<\xi_{a_{i}}).
	\end{align}
	Moreover, define the \textit{penalized length of $A$ with respect to $\xi$} by 
	\begin{align}\label{eq:def_penalized_length}
		\L(A,\xi):=   \left[ |A| - \sum_{i = \min A}^{\max A} \mathbf{1}(\xi_{i}=0) \right] \mathbf{1}(\text{$\xi$ is non-increasing on $A$}).
	\end{align}
	Note that the summation in \eqref{eq:def_penalized_length} is over the interval $[\min A, \max A]\cap \mathbb{Z}$, which may contain $A$ properly.

	\begin{lemma}\label{lemma:GK_invariants}
		Let $(\xi^{(t)})_{t\ge 0}$ be a $\kappa$-color BBS trajectory such that $\xi^{(0)}$ has finite support. Then for each $k,t\ge 0$, we have   
		\begin{align}
			\rho_{1}(\xi^{(0)}) + \cdots+ \rho_{k}(\xi^{(0)}) &\equiv  \max_{A_{1}\sqcup \cdots \sqcup A_{k} = \mathbb{N} } \sum_{i=1}^{k} \NA(A_{i}, \xi^{(t)}), \label{eq:GK_row}\\
			\lambda_{1}(\xi^{(0)}) + \cdots+ \lambda_{k}(\xi^{(0)}) &\equiv  \max_{A_{1}\prec\cdots \prec A_{k} \subseteq \mathbb{N} } \sum_{i=1}^{k} \L(A_{i},\xi^{(t)}).
		\end{align}
	\end{lemma}
	
	The proof of Lemma~\ref{lemma:GK_invariants} may be found in Section~\ref{section:pf_lem_row_GK}.

	\section{Proof of Theorem \ref{thm:permutation}}
	\label{section:permutation}
	
	In this subsection, we prove our first main result, Theorem \ref{thm:permutation}. Let $\Sigma^{n}$ be a uniformly chosen random permutation of the set $\{1,2,\cdots,n\}$, and let $\xi^{n}=\Sigma^{n}\mathbf{1}([1,n])$ be the random $n$-color BBS configuration induced from $\Sigma^{n}$. Let $\lambda_{k}(n)=\lambda_{k}(\xi^{n})$ denote the length of the $k^{\text{th}}$ longest soliton in $\xi^{n}$.

	\subsection{Proof of Theorem \ref{thm:permutation} for the columns}
	
	Our proof of Theorem \ref{thm:permutation} for the columns relies on Lemma \ref{lemma:GK_invariants} and the sharp asymptotic of longest decreasing subsequence of a uniform random permutation due to Baik, Deift, and Johansson \cite{baik1999distribution}.
	
	\begin{proof}[\textbf{Proof of Theorem \ref{thm:permutation} for the columns}]
		Fix an integer $k\ge 1$. It suffices to show that, almost surely, 
		\begin{align}
			\lim_{n\rightarrow\infty} n^{-1/2}\sum_{i=1}^{k} \lambda_{i}(n)  =  2\sqrt{k}.
		\end{align}
		For each integer $k\ge 1$, let $L(k)$ denote the length of the longest increasing subsequence in a uniformly random permutation of $k$ letters. By Lemma \ref{lemma:GK_invariants}, recall that 
		\begin{align}\label{eq:pf_thm_permutation_0}
			\lambda_{1}(n) + \cdots + \lambda_{k}(n) =  \max \left\{ \sum_{i=1}^{k} L(A_{i}, \xi^{n}) \,|\, A_{1}\prec \dots \prec A_{k} \subseteq [1,n]  \right\}.
		\end{align}
		We view a random permutation as a ranking among $n$ i.i.d. $\textup{Uniform}([0,1])$ random variables $U_{1},\cdots, U_{n}$. If $A\subseteq \{1,\cdots, n\}$, then the ranking of $U_{i}$ for $i\in A$ gives a uniform random permutation of $A$, which we call a random permutation of $[n]$ restricted on $A$. Moreover, one can also see that if we restrict a random permutation on multiple disjoint subsets, then these smaller permutations are independent. Hence,  if $A_{1}\prec\cdots \prec A_{k}$ are non-interlacing subsets of $[0,n]$, then the permutations restricted on these subsets are independent. Moreover, since the random permutation model $\xi^{n}$ does not assign color $0$ on any site in $[0,n]$, for any increasing subsequence $A\subseteq [0,n]$ and its supporting interval $I=[\min A, \max A]$, 
		\begin{align}
			\L(A,\xi^{n}) = |A| \le |I| = \L(I,\xi^{n}) \overset{d}{=} L(|I|). 
		\end{align}
		It follows that 
		\begin{align}\label{eq:pf_thm_permutation_1}
			\sum_{i=1}^{k}	\lambda_{i}(n) \overset{d}{=}  \max \left\{ \sum_{i=1}^{k} L(n_{i}) \bigg|\, \sum_{i=1}^{k} n_{i}=n,\, \textup{$L(n_{1}),\dots, L(n_{k})$ are indepenent}  \right\}.
		\end{align}

		Baik, Deift, and Johansson \cite{baik1999distribution} proved the following tail bounds for $L(n)$ (see also equations (1.7) and (1.8) in \cite{baik1999distribution} or p. 149 in \cite{romik2015surprising}):
		There exist positive constants $M, c, C$ such that  for all $m \ge 1$,
		\begin{align}
			\text{(Lower tail):} &\,\,  \P\left( m^{-1/6} (L(m)  - 2 \sqrt{m} ) \le  -  t\right) \le  C \exp(-c t^3) \,\, \text{ for all $t \in [M, 2m^{1/3}]$};\\
			\text{(Upper tail):} &\,\, \P\left( m^{-1/6} (L(m)  - 2 \sqrt{m} ) \ge   t\right) \le  C \exp(-c t^{3/5}) \,\, \text{for all $t \in [M, m^{5/6} -2m^{1/3}  ]$}.
		\end{align}
		Taking  $t = (\log m)^2$, we obtain
		\begin{align}
			\P\left(  |L(m)  - 2 \sqrt{m} | \ge  (\log m)^{2} m^{1/6} \right) \le 2C \exp(-c(\log m)^{6/5}).
		\end{align} 
		Fix $\eps>0$. Note that if $m \ge \eps \sqrt{n}$, then for any fixed $d>0$,
		\begin{align}\label{eq:BDJ_bd}
			\P\left(  |L(m)  - 2 \sqrt{m} | \ge  (\log m)^{2} m^{1/6} \right) =  O(n^{-d}).
		\end{align}
		
		Now, denote the random variable in the right-hand side of \eqref{eq:pf_thm_permutation_1} by $X$. We write $X=\max(Y,Z)$, where 
		\begin{align}
			Y &= \max \{ L(n_1) + \cdots + L(n_k) : n_1 +  \cdots + n_k =n, n_i \ge \eps \sqrt{n} \text{ for all $i$} \} ,\\
			Z &=  \max \{ L(n_1) + \cdots + L(n_k) : n_1  + \cdots + n_k =n, n_i < \eps \sqrt{n} \text{ for at least one $i$ }\}.
		\end{align} 
		Denote $\mathcal{A}:=\{ (n_{1},\dots,n_{k})\,:\, n_{1}+\dots+n_{k}=n,\, n_{i}\ge \eps\sqrt{n} \textup{ for all $i$}  \}$. For each $\eta=(n_{1},\dots,n_{k})\in \mathcal{A}$, denote $Y_{\eta}:=L(n_{1})+\dots+L(n_{k})$ and $M_{\eta}:=2(\sqrt{n_{1}}+\dots+\sqrt{n_{k}})$. Then by a union bound and \eqref{eq:BDJ_bd}, 
		\begin{align}
			\P(| Y_{\eta} - M_{\eta}| \ge k  (\log m)^{2} m^{1/6} ) = O(n^{-d}). 
		\end{align}
		Note that $Y=\max_{\eta\in \mathcal{A}} Y_{\eta}$ and 
		since there are at most $n^k$ partitions of $[n]$ into $k$ intervals, $|\mathcal{A}|\le n^{k}$. So by a union bound we have
		\begin{align}
			\P\left( \big|  Y - \max_{\eta\in \mathcal{A}} M_{\eta} \big|\right) \le \sum_{\eta\in \mathcal{A}} 	\P\left(| Y_{\eta} - M_{\eta}| \ge k  (\log m)^{2} m^{1/6} \right) = O(n^{-d}). 
		\end{align}
		for any fixed $d>0$. 
		%\begin{align}
		%	\P\left( | Y  -  \max \{ 2 \sqrt{n_1} + \cdots + 2 \sqrt{n_k}  : n_1  + \cdots + n_k =n, n_i \ge \eps \sqrt{n} \text{ $\forall \, i$} \}|  > k  (\log n)^{2} n^{1/6} \right)  = O(n^{-d})   
		%\end{align}
		The deterministic optimization problem 
		\begin{align}
			\max_{\eta\in \mathcal{A}} M_{\eta} = \max \{ 2 \sqrt{n_1} + \cdots + 2 \sqrt{n_k}  : n_1  + \cdots + n_k =n, n_i \ge \eps \sqrt{n} \text{ $\forall \, i$} \}
		\end{align}
		achieves its maximum when $\sum_{i=1}^{k} |n_{i}-(n/k)|$ is minimized, in which case we have $|n_{i}-(n/k)|\le 1$ for all $1\le i \le k$. Denoting the maximizer as $n_{1},\cdots,n_{k}$, it follows that, for all $1\le i \le k$,
		\begin{align}
			|\sqrt{n_{i}} - \sqrt{n/k}| \le \frac{1}{\sqrt{n_{i}} + \sqrt{n/k}} \le \frac{1}{2\sqrt{(n/k)-1}}.
		\end{align} 
		So this yields, for all sufficiently large $n\ge 1$,
		\begin{align}\label{eq:Y_bd}
			&\P\left(	| Y  -  2\sqrt{kn}| >  2k  (\log n)^{2} n^{1/6}  \right)  \\ \nonumber  
			&\qquad \le \P\left(	| Y  -  2\sqrt{kn}| >  k  (\log n)^{2} n^{1/6} + \frac{k}{\sqrt{(n/k)-1}}  \right) = O(n^{-d}) 
		\end{align}
		for any fixed $d>0$. 
		
		Next, if $n_i < \eps \sqrt{n}$, then we use the trivial upper bound $L(n_i) \le n_i \le \eps \sqrt{n}$, otherwise if $n_i > \eps \sqrt{n}$, we continue to use the tail bound for $|L(n_{i})-2\sqrt{n_{i}}|$ in \eqref{eq:BDJ_bd}. Hence 
		\begin{align}\label{eq:Z_bd}
			\P\left(  Z > 2\sqrt{(k-1)n} +  2 k  (\log n)^2 n^{1/6}  +  k \eps \sqrt{n} \right) = O(n^{-d}),
		\end{align}
		where the first term bounds the contribution from at most $k-1$ intervals of size $\ge \eps\sqrt{n}$, the second term is given by the BDJ tail bound in \eqref{eq:BDJ_bd}, and the last term gives a trivial bound for intervals of size $<\eps\sqrt{n}$. Hence if we choose $\eps<2/k(\sqrt{k-1}+\sqrt{k})$, then \eqref{eq:Y_bd} and \eqref{eq:Z_bd} give us
		\begin{align}\label{eq:YZ_bd}
			\P\left( Z>Y \right) &\le \P\left( Y< 2\sqrt{kn} +  2k  (\log n)^{2} n^{1/6} \right) \\ \nonumber  
			&\qquad + \P\left( Z>  2\sqrt{(k-1)n}+ 2k  (\log n)^{2} n^{1/6} + \frac{2\sqrt{n}}{\sqrt{k-1}+\sqrt{k}} \right) \\
			&= O(n^{-d}) \nonumber  
		\end{align} 
		for each fixed $d>0$.
		Now note that, for each $t>0$, 
		\begin{align}
			\P\left( \left| \left(\frac{1}{\sqrt{n}}\sum_{i=1}^{k} \lambda_{i}(n) \right) - 2\sqrt{k} \right| >  t  \right) &= \P\left( |\max(Y,Z)-2\sqrt{kn}|>t\sqrt{n} \right) \\
			&\le \P\left( |Y-2\sqrt{kn}|>t\sqrt{n} \right) + \P\left( Z> Y \right). \nonumber  
		\end{align}
		Hence by choosing $t=1/\log n$, for any fixed $d>0$, \eqref{eq:Y_bd} and \eqref{eq:YZ_bd} yield 
		\begin{align}
			\P\left( \left| \left(\frac{1}{\sqrt{n}}\sum_{i=1}^{k} \lambda_{i}(n) \right) - 2\sqrt{k} \right| >  \frac{1}{\log n}  \right) = O(n^{-d}).
		\end{align}  
		Then the assertion follows from the Borel-Cantelli lemma.
	\end{proof}

	\subsection{Circular exclusion process and the row lengths}
	\label{subsection:permutation_rows}
	In this subsection, we prove Theorem \ref{thm:permutation} for the rows. By Lemma \ref{lemma:carrier_row_lengths}, this can be done by analyzing the carrier process over the uniform random permutation $\xi^{n}$. Let $\mathbf{X}:=(U_{x})_{x\ge 1}$ be a sequence of i.i.d. $\textup{Uniform}([0,1])$ random variables. For each capacity $k\ge 1$, we may define the carrier process $(\boldsymbol{\Gamma}_{x})_{x\ge 0}$ over $\mathbf{X}$ using the same `circular exclusion rule' we used to define the map $\Psi$ in Section \ref{subsection:carrier_process_finite}. More precisely, denote $\mathcal{C}_{k}=\{ (x_{1},\cdots,x_{k})\in [0,1]^{k}\mid x_{1}\ge \cdots \ge x_{k} \}$. Define a map $\phi:\mathcal{C}_{k}\times [0,1]\rightarrow \mathcal{C}_{k}$, $[x_{1},\cdots,x_{k},y]\mapsto [x_{1}',\cdots,x_{k}']$ by 
	\begin{description}
		\item[(i)] If $y>x_{k}$, then denote $i^{*}=\min\{ i\ge 1\mid x_{i}< y \}$ and let
		\begin{align}
			[x_{1}',\cdots,x'_{k}] = [x_{1},\cdots,x_{i^{*}-1},y,x_{i^{*}+1},\cdots,x_{k}].
		\end{align}
		\item[(ii)] If $x_{k}\ge y$, then $[x_{1}',\cdots,x_{k}'] = [x_{2},\cdots,x_{k},y]$. 
	\end{description}
	Then the \textit{$k$-point circular exclusion process $(\boldsymbol{\Gamma}_{x})_{x\ge 0}$ over $\mathbf{X}$} is defined recursively by 
	\begin{align}
		\boldsymbol{\Gamma}_{x+1} = \phi(\boldsymbol{\Gamma}_{x}, U_{x+1}).
	\end{align}
	See Figure \ref{fig:circular_exclusion_ex} for an illustration. Note that $(\boldsymbol{\Gamma}_{x})_{x\ge 0}$ forms a Markov chain on state space  $\mathcal{C}_{k}$. When $\boldsymbol{\Gamma}_{0}=[0,0,\cdots,0]$, we call $(\boldsymbol{\Gamma}_{x})_{x\ge 0}$ the \textit{carrier process over $\mathbf{X}$ with capacity $k$}.

	\begin{figure*}[h]
		\centering
		\includegraphics[width=1 \linewidth]{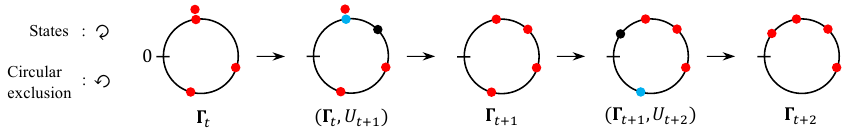}
		\caption{  Evolution of a 4-point circular exclusion process. The states in the unit circle are ordered clockwise. Each newly inserted point (black dot) annihilates the closest pre-existing point in the counterclockwise direction (light blue dot).  
		}
		\label{fig:circular_exclusion_ex}
	\end{figure*}

	In the following lemma, which will be proved in Section \ref{subsection:circular_exclusion_process_convergence}, we show that the $k$-point circular exclusion process converges to its unique stationary measure $\pi$, which is the distribution of the order statistics from $k$ i.i.d. $\textup{Uniform}([0,1])$ variables.

	\begin{lemma}\label{lemma:circular_exclusion}
		Fix an integer $k\ge 1$ and let $(\boldsymbol{\Gamma}_{x})_{x\ge 0}$ denote the  $k$-point circular exclusion process with an arbitrary initial configuration.
		\begin{description}
			\item[(i)] Let $\pi$ denote the distribution of the order statistics from $k$ i.i.d. uniform random variables on $[0,1]$. Then $\pi$ is the unique stationary distribution for the Markov chain $(\boldsymbol{\Gamma}_{x})_{x\ge 0}$.
			\item[(ii)]  For each $x \ge 0$, let $\pi_{x}$ denote the distribution of $\boldsymbol{\Gamma}_{x}$. Then $\pi_{x}$ converges to $\pi$ in total variation distance. More precisely, 
			\begin{align}
				d_{TV}(\pi_{x}, \pi):=\sup_{A\subseteq [0,1]^{k}} |\pi_{x}(A)-\pi(A)| \le \left(1 - \frac{1}{(2k)^{k-1}k!}\right)^{\lfloor x/k\rfloor},
			\end{align}
			where the supremum runs over all Lebesgue measurable subsets $A\subseteq [0,1]^{k}$.
		\end{description}
	\end{lemma} 
	
	Now we derive Theorem \ref{thm:permutation} for the row asymptotics.
	
	\begin{proof}[\textbf{Proof of Theorem \ref{thm:permutation} for the rows}]
		Let $\mathbf{X}=(U_{x})_{x\ge 1}$ denote an infinite  sequence of i.i.d. $\textup{Uniform}([0,1])$ random variables, $\Sigma^{n}$ be the random permutation on $[n]$ induced by $U_{1},\cdots,U_{n}$, and $\xi^{n}=\Sigma^{n}\mathbf{1}([1,n])$ be the random $n$-color BBS configuration as defined at \eqref{eq:def_permutation_model}. Fix an integer $k\ge 1$ and let $(\boldsymbol{\Gamma}_{x})_{x\ge 0}$ be the $k$-point circulr exclusion process over $\mathbf{X}$. Also, let $(\Gamma_{x})_{x\ge 0}$ be the capacity-$k$ carrier process over $\xi^{n}$ as defined in Section \ref{subsection:carrier_process_finite}. By construction, for each $1\le x \le n$, we have 
		\begin{align}
			\mathbf{1}(\xi^{n}(x)>\min \Gamma_{x-1}) = \mathbf{1}(U_{x}>\min \boldsymbol{\Gamma}_{x-1}). 
		\end{align}
		Thus according to Lemma \ref{lemma:carrier_row_lengths}, almost surely, 
		\begin{align}
			n^{-1}\left( \rho_{1}(\xi^{n}) + \cdots+ \rho_{k}(\xi^{n}) \right) = n^{-1}\sum_{x=1}^{n} \mathbf{1}(U_{x}> \min \boldsymbol{\Gamma}_{x-1}).
		\end{align}
		By Lemma \ref{lemma:circular_exclusion} and Markov chain ergodic theorem, almost surely, 
		\begin{align}
			\lim_{n\rightarrow \infty} n^{-1}\left( \rho_{1}(\xi^{n}) + \cdots+ \rho_{k}(\xi^{n}) \right) &= \P\left(U_{k+1}> \min(U_{1},\cdots,U_{k}) \right) = \frac{k}{k+1}.
		\end{align}
		Then the assertion follows.
	\end{proof}

	\subsection{Stationarity and convergence of the circular exclusion process}
	\label{subsection:circular_exclusion_process_convergence}
	We prove Lemma \ref{lemma:circular_exclusion} in this subsection. We will assume the stationarity of the circular exclusion process as asserted in the following proposition, which will be proved at the end of this section.
	
	\begin{prop}\label{prop:circular_exclusion_stationarity}
		Fix an integer $k\ge 1$ and let $\pi$ denote the distribution of the order statistics from $k$ i.i.d. uniform random variables on $[0,1]$. Then $\pi$ is a stationary distribution of the $k$-point circular exclusion process.
	\end{prop}

	\begin{proof}[\textbf{Proof of Lemma \ref{lemma:circular_exclusion}}]
		
		For convergence, we use a standard coupling argument. Namely, fix arbitrary distributions $\pi_{0}$ and $\bar{\pi}_{0}$ on $\mathcal{C}_{k}$ and let \( \mathbf{X} = (U_x)_{x \geq 1} \) denote a sequence of i.i.d. $\textup{Uniform}([0,1])$ variables. Let $(\boldsymbol{\Gamma}_{x})_{x\ge 0}$ be $k$-point circular exclusion processes over $\mathbf{X}$ with  initial distribution \( \pi_0 \) and let  $(\bar{\boldsymbol{\Gamma}}_{x})_{x\ge 0}$ be $k$-point circular exclusion processes over $\mathbf{X}$ with  initial distribution \( \bar{\pi}_0 \). These two processes are naturally coupled since they evolve simultaneously over the same environment $\mathbf{X}$.  Let $\tau=\inf\{ x\ge 0\mid \boldsymbol{\Gamma}_{x}=\bar{\boldsymbol{\Gamma}}_{x} \}$ denote the first meeting time of the two chains (see Figure \ref{fig:circular_exclusion_convergence}). By the coupling, $\boldsymbol{\Gamma}_{s}=\bar{\boldsymbol{\Gamma}}_{s}$ and $s\le x$ imply $\boldsymbol{\Gamma}_{x}=\bar{\boldsymbol{\Gamma}}_{x}$. A standard argument shows
		\begin{align}
			d_{TV}(\pi_{x},\bar{\pi}_{x}) \le \P(\boldsymbol{\Gamma}_{x} \ne \bar{\boldsymbol{\Gamma}}_{x}) = \P(\tau> x),
		\end{align} 
		where $\pi_{x}$ and $\bar{\pi}_{x}$ denote the distributions of $\boldsymbol{\Gamma}_{x}$ and $\bar{\boldsymbol{\Gamma}}_{x}$. We claim that 
		\begin{align}\label{lemma:circular_exc_conv_pf}
			\P(\tau> t) \le \P(\boldsymbol{\Gamma}_{0}\ne \bar{\boldsymbol{\Gamma}}_{0}) \left(1 - \frac{1}{(2k)^{k-1}k!} \right)^{\lfloor t/k \rfloor}.
		\end{align} 
		According to Proposition \ref{prop:circular_exclusion_stationarity}, this will imply Lemma \ref{lemma:circular_exclusion} by choosing $\bar{\pi}_{0}=\pi$.

		To bound the tail probability of meeting time $\tau$, we will show that two circular exclusion processes `synchronize' after $k$ steps with probability at least $1/k!$, in the sense that 
		\begin{align}
			\P(\boldsymbol{\Gamma}_{x+k}= \bar{\boldsymbol{\Gamma}}_{x+k}\mid \boldsymbol{\Gamma}_{x}\ne \bar{\boldsymbol{\Gamma}}_{x}) \ge  \frac{1}{(2k)^{k-1} k!} \qquad \text{for all $x\ge 0$}.
		\end{align}
		Then the claim \eqref{lemma:circular_exc_conv_pf} follows since 
		\begin{align}
			\P(\tau>Nk) &= \P(\boldsymbol{\Gamma}_{Nk}\ne \bar{\boldsymbol{\Gamma}}_{Nk}\mid \boldsymbol{\Gamma}_{0}\ne \bar{\boldsymbol{\Gamma}}_{0}) \P(\boldsymbol{\Gamma}_{0}\ne \bar{\boldsymbol{\Gamma}}_{0})  \\
			&\le \P(\boldsymbol{\Gamma}_{0}\ne \bar{\boldsymbol{\Gamma}}_{0})\prod_{i=1}^{N} \P(\boldsymbol{\Gamma}_{ik}\ne \bar{\boldsymbol{\Gamma}}_{ik}\mid \boldsymbol{\Gamma}_{(i-1)k}\ne \bar{\boldsymbol{\Gamma}}_{(i-1)k}) \\
			&\le \P(\boldsymbol{\Gamma}_{0}\ne \bar{\boldsymbol{\Gamma}}_{0}) \left(1-\frac{1}{(2k)^{k-1}k!} \right)^{N}.
		\end{align}

		We begin with the following simple observation for a sufficient condition of meeting. Let $\mathbf{X}=(U_{t})_{t\ge 1}$ be a sequence of i.i.d. $\textup{Uniform}([0,1])$ variables. Fix $t\ge 1$ and let $\boldsymbol{\Gamma}_{x}=[x_{1},\cdots,x_{k}]$ and $\bar{\boldsymbol{\Gamma}}_{x}=[\bar{x}_{1},\cdots,\bar{x}_{k}]$ be arbitrary elements of $\mathcal{C}_{k}$. Superpose the two $k$-point configurations into a one $2k$-point configuration $0\le y_{1}\le y_{2} \le \cdots \le y_{2k}\le 1$. For a special case, suppose $y_{2k}<1$. Observe that on the event $\{y_{2k} < U_{t+k}<\cdots <U_{t+1}\le 1 \}$, we have 
		\begin{align}
			\boldsymbol{\Gamma}_{x+k} = [U_{t+1},U_{t+2},\cdots, U_{t+k}] = \bar{\boldsymbol{\Gamma}}_{x+k},
		\end{align}
		as all of the $k$ points in $\boldsymbol{\Gamma}_{x}$ and $\boldsymbol{\Gamma}_{x}$ will be successively annihilated from the largest to the smallest by inserting $U_{t+1},\cdots, U_{t+k}$.

		\begin{figure*}[h]
			\centering
			\includegraphics[width=1 \linewidth]{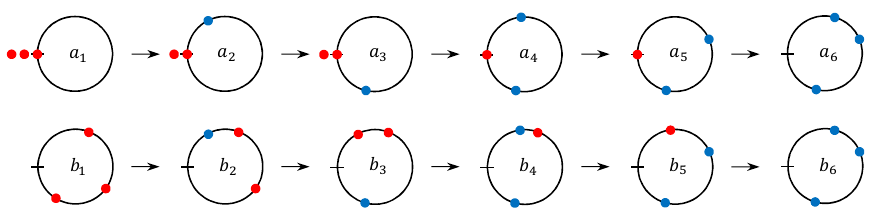}
			\caption{ Joint evolution of two 3-point circular exclusion processes. The states in the unit circle are ordered clockwise.  A newly inserted point annihilates one of the closest pre-existing points in the counterclockwise direction. Blue (resp., red) dots represent points that are shared (resp., not shared) in both processes. The two chains meet after the fifth transition. 
			}
			\label{fig:circular_exclusion_convergence}
		\end{figure*}

		For the general case, regard each $U_{s}$ as a uniformly chosen point from the unit circle $S^{1}$. Then the $2k$ points $y_{1},\cdots,y_{2k}$ will divide $S^{1}$ into disjoint arcs of lengths, say, $\ell_{1},\cdots,\ell_{m}$, for some $2\le m \le 2k$. If the points $U_{t+1},\cdots, U_{t+k}$ are strictly decreasing in the counterclockwise order within one of the $m$ arcs, then by circular symmetry and a similar observation, we will have $\boldsymbol{\Gamma}_{x+k} = \bar{\boldsymbol{\Gamma}}_{x+k}$. Noting that 
		\begin{align}
			\P\left( \begin{matrix}
				\text{$U_{t+1},\cdots, U_{t+k}$ are strictly  decreasing in the} \\
				\text{ counterclockwise order within an arc of length $\ell$}
			\end{matrix} \right) 
			= \frac{\ell^{k}}{k!}
		\end{align}
		and $\ell_{1}+\cdots+\ell_{m}=1$, H\"older's inequality yields 
		\begin{align}
			&\P\left(\boldsymbol{\Gamma}_{x+k}= \bar{\boldsymbol{\Gamma}}_{x+k}\mid  \boldsymbol{\Gamma}_{x}=[x_{1},\cdots,x_{k}],\,  \bar{\boldsymbol{\Gamma}}_{x}=[\bar{x}_{1},\cdots,\bar{x}_{k}]\right) \\
			&\qquad \ge  \sum_{i=1}^{m} \frac{\ell_{i}^{k}}{k!} \ge  \frac{1}{k!} \frac{\left(\ell_{1}+\cdots+\ell_{m} \right)^{k}}{m^{k-1}} = \frac{1}{m^{k-1} k!} \ge  \frac{1}{(2k)^{k-1} k!}. \nonumber  
		\end{align} 
		This shows the assertion. 
	\end{proof}
	
	Lastly in this section, we prove Proposition  \ref{prop:circular_exclusion_stationarity}.
	
	\begin{proof}[\textbf{Proof of Proposition \ref{prop:circular_exclusion_stationarity}}]
		
		We show $\pi$ is a stationary distribution for the Markov chain $(\boldsymbol{\Gamma}_{s})_{s\ge 0}$. Let $X_{(1)} < X_{(2)}< \cdots < X_{(k)}$ be the order statistics from $k$ i.i.d. uniform RVs on $[0,1]$. Let $Y$ be an independent $\textup{Uniform}([0,1])$ random variable. After a new point $Y$ is inserted to the preexisting list of $k$ points $X_{(1)} < X_{(2)} < \cdots < X_{(k)}$, the updated list of points will be 
		\begin{align}\label{eq:circular_exc_unif_pf_1}
			X_{(1)} < \cdots < X_{(I-1)} < Y < X_{(I+1)} < \cdots < X_{(k)},
		\end{align}
		where $I \in \{1, 2, \cdots, k\}$  is the random index such that $Y \in ( X_{(I)}, X_{(I+1)} )$. For $I = k$, the interval $( X_{(k)}, X_{(k+1)} )$ denotes the union of $(0, X_{(1)})$ and $(X_{(k)}, 1)$. In this case, the point $X_{(k)}$ is deleted and $Y$ is added as the smallest or largest point depending on which sub-intervals it falls. 
		
		We claim that \eqref{eq:circular_exc_unif_pf_1} is still the order statistics from $k$ i.i.d. uniforms on $[0,1]$, which would prove that the distribution of $k$ i.i.d. uniform points remains invariant under the transition rule. To show this, take a bounded test function $f:[0,1]^k \to \mathbb{R}$. First, we write  
		\begin{align}
			&\E \left[ f( X_{(1)},\cdots , X_{(I-1)}, Y, X_{(I+1)}, \cdots , X_{(k)})\right] \\
			&\quad = \sum_{ i=1}^k \E [ f( X_{(1)}, \cdots , X_{(i-1)}, Y, X_{(i+1)}, \cdots, X_{(k)}) \mathbf{1}_{ Y \in ( X_{(i)}, X_{(i+1)} ) } ] \\
			&\quad = \sum_{ i=1}^{k-1} \frac{1}{k!}   \int_{ z_{1} < \cdots < z_{i} <  y < z_{i+1}  < \cdots < z_{k} } f( z_{1}, \cdots, z_{i-1}, y, z_{i+1}, \cdots, z_{k}) \, dz_{1} \cdots dz_{k} dy  \\
			&\hspace{1cm} + \frac{1}{k!}   \int_{z_{1}< \cdots < z_{k}<y } f(z_{1}, \cdots, z_{k-1}, y) \, dz_{1} \cdots dz_{k} dy \\
			&\hspace{1.3cm} + \frac{1}{k!}   \int_{y<z_{1}< \cdots < z_{k} } f(y,z_{1}, \cdots, z_{k-1}) \, dz_{1} \cdots dz_{k} dy.
		\end{align}
		Integrating out $z_{i}$ and denoting $z_{0}:=0$, 
		\begin{align}
			&\quad = \sum_{ i=1}^{k-1} \frac{1}{k!}   \int_{ z_{1} < \cdots < z_{i-1} < y < z_{i+1} < \cdots < z_{k} } f( z_{1}, \cdots, z_{i-1}, y, z_{i+1}, \cdots, z_{k}) (y-z_{i-1}) \\
			&\hspace{9cm} dz_{1} \cdots z_{i-1} z_{i+1}\cdots dz_{k} dy  \\
			&\quad \qquad + \frac{1}{k!}   \int_{ z_{1}< \cdots < z_{k-1} < y } f( z_{1},\cdots, z_{k-1}, y) (y-z_{k-1}) \, dz_{1} \cdots dz_{k-1} dy \\ 
			&\quad\quad \qquad + \frac{1}{k!}   \int_{ y<z_{1} <\cdots < z_{k-1} } f( y,z_{1}, \cdots, z_{k-1}) (1-z_{k-1}) \, dz_{1} \cdots dz_{k-1} dy,
		\end{align}
		
		We then rename $y$ as $z_{i}$ for the first integral above and as $z_{k}$ for the second integral above. For the last integral, we rename $y$ as $z_{1}$ and $z_{i}$ as $z_{i+1}$ for $i=1,\dots,k-1$. This gives 
		\begin{align}
			&= \frac{1}{k!}   \int_{ z_{1} < \cdots  < z_{k} } f( z_{1}, \cdots, z_{k}) \left[ (1-z_{k})+\left(\sum_{i=1}^{k-1} z_{i}-z_{i-1} \right) + (z_{k}-z_{k-1})\right] \, dz_{1} \cdots dz_{k}  \\
			&=  \E\left[ f( X_{(1)},\cdots , X_{(I-1)}, X_{(I)}, X_{(I+1)}, \cdots , X_{(k)})\right].
		\end{align}
		This shows the assertion. 
	\end{proof}

	\section{Proof of Theorem \ref{thm:carrier_subcritical} \textbf{(i)}}
	\label{sec:carrier_subcritical}
	
	We prove Theorem \ref{thm:carrier_subcritical} \textbf{(i)} in this section. Recall the probability distribution $\pi$ in \eqref{eq:def_stationary_distribution_carrier_sub}. We assume $p_{0}>p^{*}:=\max(p_{1},\dots,p_{\kappa})$ in the following proof.

	\begin{proof}[\textbf{Proof of Theorem} \ref{thm:carrier_subcritical} \textup{\textbf{(i)}}]
		We first show the irreducibility and aperiodicity of the chain $W_{x}$. For its irreducibility, fix $\mathbf{x},\mathbf{y}\in \mathcal{B}_{\infty}$ and write $\mathbf{y}=[y_{1},y_{2},\cdots]$. Since all elements of $\mathcal{B}_{\infty}$ have finite support, there exists an integer $m\ge 1$ such that $\x(i)\equiv 0$ and $\mathbf{y}(i)\equiv 0$ for all $i\ge m$. Then note that 
		\begin{align}
			&\P( \Gamma_{x+2m}  = \mathbf{y}\mid \Gamma_{x}  = \mathbf{x}) \\
			&\quad \ge \P\big(\xi^{\mathbf{p}}(x+1)=0, \cdots,  \xi^{\mathbf{p}}(x+m)=0, \xi^{\mathbf{p}}(x+m+1)=y_{1}, \cdots, \xi^{\mathbf{p}}(x+2m)=y_{m}\big)\\
			& \quad = p_{0}^{m} p_{y_{1}} \cdots p_{y_{m}} >0.
		\end{align}
		Since $\mathbf{x},\mathbf{y}\in \mathcal{B}_{\infty}$ were arbitrary, this shows the Markov chain $W_{x}$ is also irreducible. Then for its  aperiodicity, it is enough to observe that 
		\begin{align}
			\P\big( \Gamma_{x+1}  = [0,0,\cdots]\mid \Gamma_{x}  = [0,0,\cdots]\big) =  p_{0}>0.
		\end{align}

		Next, we show that $\pi $ is a stationary distribution for $(W_{x} )_{t\ge 0}$. The uniqueness of stationary distribution and convergence in total variation distance will then follow from general results of countable state space Markov chain theory (see, e.g., \cite[Thm. 21.13 and Thm. 21.16]{levin2017markov}). We work with the original carrier process $\Gamma_{x}$. For each $\mathbf{x}\in \mathcal{B}_{\infty}$ and $i\in \{0,1,\cdots, \kappa \}$, denote 
		\begin{align}\label{eq:def_weight}
			\exp(\wt(\mathbf{x})) =  \prod_{i=1}^{\kappa} \left( \frac{p_{i}}{p_{0}} \right)^{m_{i}(\mathbf{x})}, \quad \exp(\wt(i))  =p_{i}.
		\end{align}	
		Recall the definition of the map $\Psi :\mathcal{B}_{\infty}\times \{0,1,\cdots, \kappa \} \rightarrow \{0,1,\cdots, \kappa \}\times \mathcal{B}_{\infty}$ given in Section \ref{subsection:carrier_process_infinite}. Note that for each pair $(\mathbf{x},y)\in \mathcal{B}_{\infty}\times \{0,1,\cdots, \kappa \}$ and $(y',\mathbf{x}')\in  \{0,1,\cdots, \kappa \}\times \mathcal{B}_{\infty}$ such that $\Psi (\mathbf{x},y) = (y',\mathbf{x}')$, $y'=y'(\x,y)$, we have 
		\begin{align}\label{eq:exp_wt_identity}
			\exp(\wt(\mathbf{x})) \exp(\wt(y)) &= p_{y} p_{0}^{-\lVert \x \rVert_{1}} \prod_{i=1}^{\kappa} p_{i}^{m_{i}(\x)}  \\
			&= p_{y'} p_{0}^{-\lVert \x' \rVert_{1}} \prod_{i=1}^{\kappa} p_{i}^{m_{i}(\x')}  =  \exp(\wt(y')) \exp(\wt(\mathbf{x}')).
		\end{align}
		Indeed, the total number of each letter $1\le i\le \kappa$ in both pairs $(\mathbf{x},y)$ and $(y',\mathbf{x}')$ is the same. So if $y'\ge 1$, then some ball of positive color in $\x$ is replaced by a ball of positive color $y'$, so $\lVert \x \Vert_{1}=\lVert \x' \Vert_{1}$ and the above identity holds; If $y'=0$ and $y\ge 1$, then $\x'$ has one more ball of color $y$ than $\x$ does so the above identity holds; If $y'=y=0$, then both $\x'$ and $\x$ do not contain any ball of positive color so the above identity holds. 
		
		Now, observe that for each fixed $\mathbf{x}'\in \mathcal{B}_{\infty} $, $\Psi $ gives a bijection between $\{0,1,\cdots,\kappa \} \times \{ \mathbf{x}' \}$ and its inverse image under $\Psi $. If we denote the second coordinate of $\Psi $ by $\Psi_{2} $, then this yields 
		\begin{align}
			\sum_{\substack{(\mathbf{x},y)\in \mathcal{B}_{\infty} \times \{0,1,\cdots, \kappa\} \\ \Psi_{2} (\mathbf{x},y)=\mathbf{x}' }} \exp(\wt(\mathbf{x})) \exp(\wt(y)) & =  \sum_{\substack{(\mathbf{x},y)\in \mathcal{B}_{\infty} \times \{0,1,\cdots, \kappa\} \\ \Psi_{2} (\mathbf{x},y)=\mathbf{x}' }} \exp(\wt(y'(\x,y)))\exp(\wt(\mathbf{x}')) \\
			&=  \exp(\wt(\mathbf{x}'))  \sum_{y'\in \{0,1,\cdots, \kappa \}} \exp(\wt(y'))\\
			& =  \exp(\wt(\mathbf{x}')).
		\end{align}
		Dividing both sides by   
		\begin{align}
			\sum_{\mathbf{x}\in \mathcal{B}_{\infty} } \exp(\wt(\mathbf{x})) = \sum_{n_{1}=0}^{\infty}\cdots \sum_{n_{\kappa}=0}^{\infty} \prod_{i=1}^{\kappa} \left( \frac{p_{i}}{p_{0}} \right)^{n_{i}} = \prod_{i=1}^{\kappa} \left(1-\frac{p_{i}}{p_{0}} \right)^{-1}>0,
		\end{align}
		we get 
		\begin{align}
			\sum_{\substack{(\mathbf{x},i)\in \mathcal{B}_{\infty} \times \{0,1,\cdots, \kappa\} \\ \Psi_{2} (\mathbf{x},i)=\mathbf{x}' }} \pi(m_{1}(\x), \cdots, m_{\kappa}(\x)) p_{i} = \pi(m_{1}(\x'), \cdots, m_{\kappa}(\x')).  
		\end{align}
		This shows that $\pi$ is a stationary distribution of the Markov chain $(W_{x})_{x\ge 0}$, as desired.
		
		Lastly, positive recurrence follows from the irreducibility and the existence of stationary distribution \cite[Thm. 21.13]{levin2017markov}. Convergence of the distribution of $W_{x}$ to the stationary distribution in total variation distance then follows from the irreducibility, aperiodicity, and positive recurrence (see \cite[Thm. 21.16]{levin2017markov}). 
	\end{proof}

	\begin{remark}
		The statement and the proof of Theorem \ref{thm:carrier_subcritical} \textbf{(i)} are reminiscent of \cite[Thm. 1]{kuniba2020large}, where the authors show that for all $\mathbf{p}=(p_{0},\cdots,p_{\kappa})$, the (finite) capacity-$c$ carrier process over $\xi^{\mathbf{p}}$ is irreducible with unique stationary distribution 
		\begin{align}\label{eq:stationary_measure_product_formula}
			\pi_{c}(\x) = \frac{1}{Z_{c}} \prod_{i=0}^{\kappa} p_{i}^{m_{i}(\x)}, \quad \x\in \mathcal{B}_{\infty},
		\end{align}
		where $Z_{c}$ denotes the partition function. In fact, their result applies to more general finite-capacity carriers whose state space is the set $B_{c}^{(a)}(\kappa)$ of all semistandard tableaux of rectangular shape $(c\times a)$ with letters from $\{0,1,\cdots, \kappa \}$. In this general case, the partition function $Z_{c}=Z_{c}^{(a)}(\kappa,\mathbf{p})$ is identified with the Schur polynomial associated with the $(a\times c)$ Young tableau with constant entries $c$ and parameters $p_{0},p_{1},\cdots,p_{\kappa}$.
	\end{remark}

	\vspace{-0.3cm}
	\section{The Skorohkod decomposition of the carrier process}
	\label{sec:decoupled_carrier_main}
	
	In this section, we develop the Skorohkod decomposition of the carrier process, which we briefly mentioned in the introduction. The idea is to write the carrier process, which is confined in the nonnegative integer orthant $\Z_{\ge 0}^{\kappa}$, as the sum of a less confined process and a boundary correction. Namely, let $(W_{x})_{x\ge 0}$ be the carrier process over an arbitrary ball configuration $\xi$ as in \eqref{eq:W_x_recursion}. We seek for the following decomposition 
	\vspace{-0.3cm}
	\begin{align}\label{eq:Skorokhod_decomp_identity1}
		W_{x} = X_{x} + R Y_{x} \quad \textup{for $x\ge 0$},
	\end{align}
	where 
	\begin{description}
		\item[1.] $(X_{x})_{x\ge 0}$ is the `decoupled carrier process', which is a version of the carrier process that allows the number of balls of certain `exceptional colors' to be negative; 
		
		\item[2.] $R=\textup{tridiag}_{\kappa}(0,1,-1)$ is the $\kappa\times \kappa$ `reflection matrix' (see \eqref{eq:def_reflection_mx}); 
		
		\item[3.] $(Y_{x})_{x\ge 0}$ is the `pushing process':  $Y_{0}=\mathbf{0}$ and for each $i\in \{1,\dots,\kappa\}$, the $i$th coordinate of $Y_{x}$  is non-decreasing in $x$ and can only increase when $W_{x}(i)=0$. 
	\end{description}
	
	We will first introduce the decoupled carrier process $(X_{x})_{x\ge 0}$ in Section \ref{sec:decoupled_carrier} and establish its basic properties in Proposition \ref{prop:carrier_comparison_localized}. In Section \ref{sec:carrier_Skorokhod_0}, we will introduce the reflection matrix $R$ and the pushing process $(Y_{x})_{x\ge 0}$ and verify the Skorohkod decomposition \eqref{eq:Skorokhod_decomp_identity1} in Lemma \ref{lemma:Skorokhod}. All results in this section are for a deterministic ball configuration $\xi$.

	\subsection{Definition of the decoupled carrier process}
	\label{sec:decoupled_carrier}

	In this section, we introduce a `decoupled version' of the carrier process $W_{x}$ in \eqref{eq:W_x_recursion}, which will be critical in proving Theorem \ref{thm:carrier_subcritical} \textbf{(ii)} as well as Theorems \ref{thm:iid_critical}-\ref{thm:iid_supercritical}.

	\begin{figure*}[h]
		\centering
		\includegraphics[width=0.6 \linewidth]{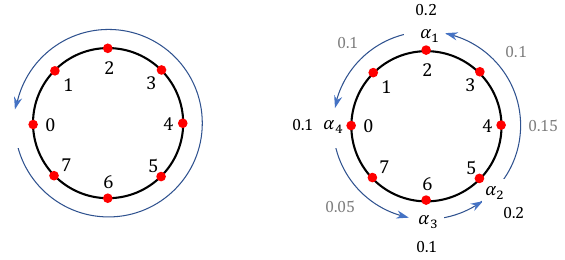}
		\vspace{-0.2cm}
		\caption{ Illustration of the original circular exclusion rule (left)  and its decoupled version (right) for $\kappa=7$ and ball density $\mathbf{p}=(.1,\,  .1,\,  .25,\,  .05,\,  .15,\,  .2,\,  .1,\,  .05)$. We take the set of exceptional colors $\mathcal{C}_{e}$ to be the set of unstable colors $\mathcal{C}_{u}^{\p}=\{2,5,6\}$. For instance, in the decoupled carrier process, inserting new balls of color $5$ into the carrier only excludes existing balls of colors $2,3$ and $4$.
		}
		\label{fig:circular_exclusion_rules}
	\end{figure*}
	
	\vspace{-0.7cm}
	To illustrate the idea, consider the carrier process $W_{x}$ with $\kappa=2$ as in Figure \ref{fig:MC_diagram}. While the transition kernel for this Markov chain depends on whether it is in the interior or at the boundary of the state space $\mathbb{Z}^{2}_{\ge 0}$, we may consider a similar Markov chain on the entire integer lattice $\mathbb{Z}^{2}$ that only uses the kernel in the interior, by allowing the counts of color 1 and 2 balls in $W_{x}$ to be negative. In the general construction of decoupled carrier processes, we will allow the freedom to choose positive colors $ \alpha_{1}<\dots<\alpha_{r}$ in $\{1,\dots,\kappa\}$ whose count can be negative. Recall that inserting a ball of color $i$ to the carrier $W_{n}$ will exclude the largest color $i_{*}$ in $W_{n}$ that is less than $i$. In the decoupled carrier process, the color wheel $\Z_{\kappa+1}$  is divided into intervals $[0,\alpha_{1}]$, $[\alpha_{1},\alpha_{2}],\dots, [\alpha_{r}, \kappa]$, and inserting a color $i$  in $(\alpha_{j},\alpha_{j+1}]$ can only exclude a color in the interval $[\alpha_{j},\alpha_{j+1}]$. Hence, the interaction between colors in distinct intervals is `decoupled'. See Figure \ref{fig:circular_exclusion_rules} for an illustration.

	\begin{definition}[Decoupled carrier process]\label{eq:def_decoupled_carrier}
		Let $\xi:=(\xi_{x})_{x\in \N}$ be $\kappa$-color ball configuration and fix a set $\mathcal{C}_{e}\subseteq\{ 1,\dots,\kappa\}$ of `exceptional colors'. Let 
		\begin{align}\label{eq:decoupled_carrier_state_space}
			\Omega:=\{ (x_{1},\dots,x_{\kappa})\in \Z^{\kappa}\,:\, \textup{$x_{i}\ge 0$ if $i\notin \mathcal{C}_{e}$} \}.
		\end{align}
		The  \textit{decoupled carrier process} over $\xi$ associated with $\mathcal{C}_{e}$ is a process $(X_{x})_{x\in \N}$ on the state space $\Omega$ defined as follows. If $\mathcal{C}_{e}=\emptyset$, the we take $X_{x}\equiv W_{x}$, where $W_{x}$ the carrier process in \eqref{eq:W_x_recursion}. Suppose $\mathcal{C}_{e}=\{ \alpha_{1},\dots,\alpha_{r} \}$ for some $r\ge 1$ with $\alpha_{1} < \dots < \alpha_{r}$. Denote $\alpha_{r+1}:=\kappa+1$. Having defined $X_{1},\dots,X_{x}$, denote $i:=\xi_{x+1}$ if $\xi_{x+1}\in \{1,\dots,\kappa\}$ and $i:=\kappa+1$ if $\xi_{x+1}=0$. Then 
		\begin{align}\label{eq:X_x_recursion}
			X_{x+1} - X_{x}:=
			\begin{cases}
				%\e_{1} & \textup{if $i=1$} \\
				\e_{i}- \mathbf{1}(i_{*}\ne 0) \,  \e_{i_{*}} & \textup{if $1\le i \le \alpha_{1}$} \\
				\e_{i} - \e_{i'} & \textup{if $\alpha_{1}< i \le \kappa$} \\
				-\e_{i'} &  \textup{if $i=\kappa+1$},
			\end{cases}
		\end{align}
		where $i_{*}:=	\sup\{j\,:\, 1\le j< i ,\,\,  X_{x}(j)\ge 1 \}$ (with the convention $\sup \emptyset = 0$) and 
		\begin{align}\label{eq:def_i'}
			i':=\begin{cases}
				\alpha_{j} & \textup{if $\alpha_{j}< i\le \alpha_{j+1}$ and $X_{x}(\alpha_{j})=\dots=X_{x}(i-1)\le 0$}\\ 
				q & \textup{if $\alpha_{j}< q < i\le \alpha_{j+1}$ and $X_{x}(q)\ge 1$, $X_{x}(q+1)=\dots=X_{x}(i-1)=0$.}
			\end{cases}
		\end{align}
		Unless otherwise mentioned, we take $X_{0}=\mathbf{0}$ and $\xi=\xi^{\p}$ with density $\p=(p_{0},\dots,p_{\kappa})$. 
	\end{definition}

	It is helpful to compare the recursion \eqref{eq:X_x_recursion} for the decoupled carrier process to that of the carrier process in \eqref{eq:W_x_recursion}. Notice that in \eqref{eq:W_x_recursion}, inserting $i$ into $W_{x}$ can decrease by one at coordiante $i_{*}$ only when $W_{x}(i_{*})\ge 1$. Hence $W_{x}$ is confined in the nonnegative orthant $\Z^{\kappa}_{\ge 0}$. In comparison, when a ball of color $i$ is inserted to the decoupled carrier $X_{x}$, it decreases by one at coordinate, say $\ell\in \{ i', i_{*} \}$. If $\ell\notin \mathcal{C}_{e}$, then the above construction ensures that $X_{x}(\ell)\ge 1$. From this, one can observe that $X_{x}(j)\ge 0$ for all $x\ge 0$ whenever $j\notin \mathcal{C}_{e}$. In contrast, if $\ell\in \mathcal{C}_{e}$, then $X_{x+1}(\ell)=X_{x}(\ell)-1$ regardless of whether $X_{x}(\ell)\ge 1$. Hence $X_{x}$ can take negative values on the exceptional colors. We call the recursion in \eqref{eq:X_x_recursion} as the `decoupled circular exclusion'.

	In the proposition below, we establish a basic coupling result between the carrier and the decouple carrier processes. For its proof, we will introduce the following notation. Define the following function $f_{W}:\Z^{\kappa}_{\ge 0}\times \{0,\dots,\kappa\}\rightarrow \{0,\dots,\kappa\}$ as 
	\begin{align}\label{eq:excluded_color_W_def}
		f_{W}(\w, y)&:=
		\begin{cases}
			0 & \textup{if [$W_{0}=\w$ and $\xi_{1}=y$ $\,\, \Longrightarrow \,\,$ $W_{1}-W_{0}=\e_{y}$]} \\
			j & \textup{if [$W_{0}=\w$ and $\xi_{1}=y$ $\,\, \Longrightarrow \,\,$ $W_{1}-W_{0}=\e_{y}-\e_{j}$ or $-\e_{j}$]}.
		\end{cases}
	\end{align}
	Roughly speaking, if $f_{W}(\w, y)=j$, then $j$ is the color of the ball that is excluded when a ball of color $y$ is inserted into the carrier of state $\w$. The circular exclusion rule says $f_{W}(\w, y)=\sup\{i \,:\, 1\le i < y,\,  \w(i)\ge 1  \}$ with the convention $\kappa+1\equiv 0$ and $\sup \emptyset = 0$. Similarly, define a function $f_{X}:\Omega \times \{0,\dots,\kappa\}\rightarrow \{0,\dots,\kappa\}$ as 
	\begin{align}\label{eq:excluded_color_X_def}
		f_{X}(\w, y)&:=
		\begin{cases}
			0 & \textup{if [$X_{0}=\w$ and $\xi_{1}=y$ $\,\, \Longrightarrow \,\,$ $X_{1}-X_{0}=\e_{y}$]} \\
			j & \textup{if [$X_{0}=\w$ and $\xi_{1}=y$ $\,\, \Longrightarrow \,\,$ $X_{1}-X_{0}=\e_{y}-\e_{j}$ or $-\e_{j}$]}.
		\end{cases}	
	\end{align}
	Intuitively, if $f_{X}(\w, y)=j$, then $j$ is the color of the ball that is excluded when a ball of color $y$ is inserted into the decoupled carrier of state $\w$. 
	
	For each $x\in \N$, define $\hat{X}_{x}\in \Z^{\kappa}_{\ge 0}$ by 
	\begin{align}\label{eq:hat_X_def}
		\hat{X}_{x}(i):=X_{x}(i) - \min_{0\le s \le x} X_{s}(i)  \quad \textup{for all $i=1,\dots,\kappa$}. 
	\end{align}
	Note that $\hat{X}_{x}(i)\ge \max(0, X_{x}(i))$ for all $i$ by definition and $X_{0}=\mathbf{0}$. Also, $\hat{X}_{x}(i)\equiv X_{x}(i)$ for all $i\notin \mathcal{C}_{e}$ since $X_{x}(i)\ge 0$ for all $x\in \N$ and  all $i\notin \mathcal{C}_{e}$. 
	
	\begin{prop}[Basic coupling between the carrier and the decoupled carrier processes] \label{prop:carrier_comparison_localized}
		Let $(W_{x})_{x\ge 0}$ be the carrier process in \eqref{eq:W_x_recursion} and let $(X_{x})_{x\ge 0}$ be the decoupled carrier process in \eqref{eq:X_x_recursion} associated with $\mathcal{C}_{e}=\{\alpha_{1},\dots,\alpha_{r}\}$ for some $r\ge 1$. Suppose these two processes evolve over the same ball configuration $\xi$ and $W_{0}=X_{0}=\mathbf{0}\in \Z_{\ge 0}^{\kappa}$. Then the following hold. 
		\begin{description}[itemsep=0.1cm]

			\item[(i)] $W_{x}(i)=X_{x}(i)$ for all $\alpha_{r}< i \le \kappa$ and $x\ge 0$. Furthermore, 
			\begin{align}
				W_{x}(\alpha_{r}) =	X_{x}(\alpha_{r}) \quad \textup{if $X_{1}(\alpha_{r}),\dots, X_{x-1}(\alpha_{r})\ge 1$}. 
			\end{align}
			
			\item[(ii)] $W_{x}(i) \le \hat{X}_{x}(i)$ for all $1 \le i \le \kappa$ and $x\ge 0$. Furthermore, for each $x\ge 0$, denoting $y:=\xi_{x+1}$ if $\xi_{x+1}\in \{1,\dots,\kappa\}$ and $y:=\kappa+1$ if $\xi_{x+1}=0$,  	\begin{align}\label{eq:excluded_colors_inequality}
				f_{W}(W_{x}, \xi_{x+1})  \le 		f_{X}(X_{x}, \xi_{x+1}) <y. 
			\end{align}
		\end{description}
	\end{prop}
	
	\begin{proof}
		In this proof, we denote $y_{X}:=f_{X}(X_{x},\xi_{x+1})$ and $y_{W}:=f_{W}(W_{x}, y)$. Note that $y_{W},y_{X}\in [0, \xi_{x+1})$ (recall that $\sup\emptyset=0$). 
		
		The second part of \textbf{(i)} follows from the first part of \textbf{(i)} and definition.  Now we show the first part of \textbf{(i)} by induction on $x\ge 0$. For $x=0$ we have $W_{0}=X_{0}=\mathbf{0}$. Denote $\ell:=\alpha_{r}$ and suppose $W_{x}(i)= X_{x}(i)$ for all $\ell < i \le \kappa$ for some $x\ge 0$. If $y\le \ell$, then inserting a ball of color $y$ into the carrier $W_{x}$ and the decoupled carrier $X_{x}$ does not affect their state for colors strictly larger than $\ell$. Hence $W_{x+1}(i)=W_{x}(i)=X_{x}(i)=X_{x+1}(i)$ for all $\ell < i \le \kappa$. So suppose $y>\ell$. In this case, $y_{W}=\sup\{ 1\le j< y \,:\, W_{x}(j)\ge 1 \}$ and  $y_{X}=\max\{ \ell,  \sup\{ 1\le j < y\,:\, X_{x}(j)\ge 1 \}\}$. Note that $W_{x+1}$ is obtained from $W_{x}$ by increasing its value on color $y$ by one and decreasing its value on color $y_{W}$ by one. If $y_{W}>\ell$, then by the induction hypthesis, $y_{W}=y_{X}$, so $X_{x+1}$ is obtained from $X_{x}|_{(\ell,\kappa]}=W_{x}|_{(\ell,\kappa]}$ by the same way, so $X_{x+1}|_{(\ell,\kappa]}=W_{x+1}|_{(\ell,\kappa]}$. Otherwise, suppose $y_{W}\le \ell$. Then $X_{x+1}$ is obtained from $X_{x}|_{(\ell,\kappa]}=W_{x}|_{(\ell,\kappa]}$ by increasing its value on color $y$ by one and decreasing its value on color $\ell$ by one. Hence $W_{x+1}|_{(\ell,\kappa]}=X_{x+1}|_{(\ell,\kappa]}$,  as desired. 
		
		Now we prove \textbf{(ii)} by an induction on $x\ge 0$. The base step when $x=0$ follows by definition ($W_{0}=\hat{X}_{0}=\hat{X}=\mathbf{0}$ and $0=y_{W}\le y_{X}<y$). For the induction step, suppose $W_{x}  \le \hat{X}_{x}$ coordinatewise for some $x\ge 0$.  We first show that $y_{W}\le y_{X} < y$.  That $y_{X}<y$ follows from the definition \eqref{eq:X_x_recursion}. To show $y_{W}\le y_{X}$, we assume $y_{W}\ge 1$ since otherwise the claim holds trivially. Since a ball of color $y_{W}\ge 1$ is excluded from the carrier $W_{x}$, we have $W_{x}(y_{W})\ge 1$. If $y_{W} \notin \mathcal{C}_{e}$, then by the induction hypothesis, $1\le W_{x}(y_{W})\le X_{x}(y_{W})$, so it follows that $y_{W}\le y_{X}$. Otherwise, suppose $y_{W} \in \mathcal{C}_{e}$. Then since $y_{X}$ is at least the largest exceptional color that is $<y$, it follows that $y_{W}\le y_{X}$, as desired. 
		
		It remains to show $W_{x+1}\le \hat{X}_{x+1}$ coordinatewise. First suppose $y_{W}=0$. Then  $W_{x}(1)=\dots=W_{x}(y-1)=0$, so  $W_{x+1}(1)=\dots=W_{x+1}(y-1)=0$ and $W_{x+1}-W_{x} = \e_{y}$. Hence $W_{x+1}(i)=0\le \hat{X}_{x+1}(y)$ for all $1\le i < y$. Noting that $X_{x+1}(y)=X_{x}(y)+1$, by definition we have $\hat{X}_{x+1}(y)=\hat{X}_{x}(y)+1$. Then by the induction hypothesis, we have $W_{x+1}(y)=W_{x}(y)+1 \le \hat{X}_{x}(y)+1=\hat{X}_{x+1}(y)$. Furthermore, $W_{x+1}(i)=W_{x}(i)\le \hat{X}_{x}(i) 
		= \hat{X}_{x+1}(i)$ for all $y< i \le \kappa$, where the middle inequality is from the induction hypothesis and the equalities are from the definition. Thus we have shown that $W_{x+1}\le \hat{X}_{x+1}$ coordinatewise.
		
		Lastly, we suppose $y_{W}\ge 1$ and show $W_{x+1}\le \hat{X}_{x+1}$ coordinatewise. Then  $1\le y_{W}\le y_{X}<y$,  $W_{x+1}-W_{x}=\e_{y}-\e_{y_{W}}$, and $X_{x+1}-X_{x}=\e_{y}-\e_{y_{X}}$. By the induction hypothesis and the definition, we only need to verify $W_{x+1}(y_{X})\le \hat{X}_{x+1}(y_{X})$. This holds when $y_{W}=y_{X}$ since then $W_{x+1}(y_{X})=W_{x}(y_{X})-1\le  \hat{X}_{x}(y_{X})-1 \le \hat{X}_{x+1}(y_{X})$. So we may assume $y_{W}<y_{X}$. 
		By definition of $y_{W}$, we have $W_{x}(y_{W}+1)=\dots=W_{x}(y-1)=0$ and so $W_{x+1}(y_{W}+1)=\dots=W_{x+1}(y-1)=0$. Then 
		by definition $W_{x+1}(y_{X})=0\le \hat{X}_{x+1}(y_{X})$. This completes the induction. 
	\end{proof}

	\subsection{Proof of the Skorokhod decomposition of the carrier process}
	\label{sec:carrier_Skorokhod_0}

	Now we give an explicit construction of the Skorokhod decomposition of $(W_{x})_{x\ge 0}$. First, let $R$ be the $\kappa \times \kappa$ tridiagonal matrix with 0 on the subdiagonal, 1 on the main diagonal, and -1 on the superdiagonal entries:
	\vspace{-0.3cm}
	\begin{align}\label{eq:def_reflection_mx}
		R :=\textup{tridiag}_{\kappa}(0,1,-1) = \begin{bmatrix}
			1 & -1  & 0  \\
			0 & 1 & -1  & 0  \\
			\vdots   && \ddots \\
			0 &   \cdots &0& 1  &  -1  \\
			0 &  & \cdots & 0  & 1  
		\end{bmatrix}
		=	I - Q,
	\end{align}
	where $I$ is the $\kappa\times \kappa$ identity matrix and $Q=I-R$. Notice that the spectral radius of $Q$ is zero for all $\kappa\ge 2$ being an upper triangular matrix with zero diagonal entries. The above reflection matrix also has the property of being `completely-$\mathcal{S}$', see Def. \ref{def:completely_S} and the proof of Theorem \ref{thm:SRBM_weak_convergence} for justification.

	Next, we define the pushing process $(Y_{x})_{x\ge 0}$ on  $\Z^{\kappa}_{\ge 0}$ recursively as follows: Set $Y_{0}=\mathbf{0}$. Having defined $Y_{x}$, denoting $y_{W}:=f_{W}(W_{x}, \xi_{x+1})$ (see \eqref{eq:excluded_color_W_def}) and $y_{X}:=f_{X}(X_{x}, \xi_{x+1})$ (see \eqref{eq:excluded_color_X_def}), define 
	
	\begin{align}\label{eq:def_Yn_SRBM}
		Y_{x+1} - 	Y_{x}  := 
		\begin{cases} 
			\mathbf{0}			 & \textup{if $y_{W}=y_{X}$} \\
			\e_{y_{W}+1}+\dots+\e_{y_{X}} & \textup{if $y_{W}<y_{X}$}.
		\end{cases}
	\end{align}
	Note that \eqref{eq:def_Yn_SRBM} covers all cases since $y_{W}\le y_{X}$ due to Proposition \ref{prop:carrier_comparison_localized}. From the definition, it is clear that every coordinate of $Y_{x}$ is non-decreasing. Also, clearly, $Y_{x}$ is determined by the first $x$ ball colors $\xi_{1},\dots,\xi_{x}$.

	\begin{lemma}[Skorokhod decomposition of the carrier process]\label{lemma:Skorokhod}
		Let $W_{x}$, $X_{x}$, $R$, and $Y_{x}$ as before. Then
		\begin{description}[itemsep=0.1cm]
			\item[(i)] $W_{x} = X_{x} + R Y_{x}$ for all $x\ge 0$;
			\item[(ii)] $Y_{0}=\mathbf{0}$ and for each $i\in \{1,\dots,\kappa\}$, the $i$th coordinate of $Y_{x}$  is non-decreasing in $x$ and can only increase when $W_{x}(i)=0$, i.e., $\sum_{x\ge 0} \mathbf{1}(W_{x}(i)\ge 1) (Y_{x+1}(i)-Y_{x}(i)) = 0$.
		\end{description}
	\end{lemma}
	
	\begin{proof}
		Let $y:=\xi_{x+1}$ if $\xi_{x+1}\ne 0$ and $y:=\kappa+1$ if $\xi_{x+1}=0$. Also let $y_{W}:=f_{W}(W_{x}, \xi_{x+1})$ and $y_{X}:=f_{X}(X_{x}, \xi_{x+1})$ (see \eqref{eq:excluded_color_W_def} and \eqref{eq:excluded_color_X_def}). We first show \textbf{(ii)}. According to  \eqref{eq:excluded_colors_inequality} in Proposition \ref{prop:carrier_comparison_localized}, we have $y_{W}\le y_{X}< y$. Also, by the definition of $y_{W}$, we have $W_{x}(y_{W}+1)=\dots=W_{x}(y-1)=0$. Hence if $Y_{x+1}(i)-Y_{x}(i)>0$, then $i\in \{ y_{W}+1,\dots, y-1 \}$ and hence $W_{x}(i)=0$. This shows \textbf{(ii)}. 
		
		Next, we show \textbf{(i)} by induction on $x\ge 0$. It holds trivially when $x=0$, so suppose for the induction step that it holds for some $x\ge 0$. We wish to show that 
		\begin{align}\label{eq:lem_SK_pf0}
			W_{x+1} = X_{x+1} + R Y_{x+1}. 
		\end{align}
		From \eqref{eq:excluded_color_W_def}-\eqref{eq:excluded_color_X_def}, note that 
		\begin{align}\label{eq:Skorokhod_decomp_pf_difference}
			(W_{x+1}-W_{x}) - (X_{x+1}-X_{x})  = 
			\begin{cases} 
				\mathbf{0}			 & \textup{if $y_{W}=y_{X}$} \\
				\e_{y_{X}} - \e_{y_{W}}  & \textup{if $1\le y_{W}<y_{X}$} \\
				\e_{y_{X}}   & \textup{if $0= y_{W}<y_{X}$}.
			\end{cases}
		\end{align}
		If $y_{W}=y_{X}$, then $R(Y_{x+1}-Y_{x}) = \mathbf{0}$ so \eqref{eq:lem_SK_pf0} holds by the induction hypothesis. Next, suppose $1\le y_{W}<y_{X}$. Note that 
		\begin{align}
			R(Y_{x+1}-Y_{x}) & = R(\e_{y_{W}+1}+\dots+\e_{y_{X}})  \\
			&= (\e_{y_{W}+1}-\e_{y_{W}}) + (\e_{y_{W}+2}-\e_{y_{W}+1}) + \dots + (\e_{y_{X}}-\e_{y_{X}-1})  \\
			&= \e_{y_{X}} - \e_{y_{W}}.
		\end{align}
		Lastly, suppose $0= y_{W}<y_{X}$. Then 
		\begin{align}
			R(Y_{x+1}-Y_{x}) & = R(\e_{1}+\dots+\e_{y_{X}})  \\
			&=\e_{1} + (\e_{2}-\e_{1}) + (\e_{3}-\e_{2}) + \dots + (\e_{y_{X}}-\e_{y_{X}-1})  = \e_{y_{X}}.
		\end{align}
		Hence in all cases, the induction step holds by the induction hypothesis, \eqref{eq:Skorokhod_decomp_pf_difference}, and \eqref{eq:def_Yn_SRBM}.
	\end{proof}

	\section{Probabilistic analysis of the decoupled carrier process}
	\label{sec:prob_decoupled_carrier}
	
	In the previous section, we defined the decoupled carrier process $(X_{x})_{x\ge 0}$  associated with an arbitrary set $\mathcal{C}_{e}=\{\alpha_{1},\dots,\alpha_{r}\}\subseteq\{1,\dots,\kappa\}$ of exceptional colors over a deterministic ball configuration $\xi$. In this section, we establish various important probabilistic results for the decoupled carrier process $(X_{x})_{x\ge 0}$ over the i.i.d. ball configuration $\xi^{\p}$ with a particular choice of the associated set $\mathcal{C}_{e}$ of exceptional colors.

	\subsection{Decomposition of the decoupled carrier process}
	
	Let $\p=(p_{0},\dots,p_{\kappa})$ be the ball density at each site. We choose the set of exceptional colors $\mathcal{C}_{e}$ so that it satisfies the following `stability condition':
	\begin{align} \label{eq:stability_color_assumption}
		\textup{For all $1\le j \le r$,} \quad 	\max \{p_{i}\,:\, \alpha_{j}<i<\alpha_{j+1} \} < p_{\alpha_{j+1}},
	\end{align}
	where we set $\alpha_{0}=0=\alpha_{r+1}$. Since  balls of a non-exceptional color $i$ in $(\alpha_{j}, \alpha_{j+1})$ can be excluded by balls of color $\alpha_{j+1}$ in the decoupled carrier, the above condition ensures that $(X_{x}(i))_{x\ge 0}$ do not blow up. A canonical choice of such $\mathcal{C}_{e}$ is the set of unstable colors $\mathcal{C}^{\p}_{u}$ that we defined above the statement of Theorem \ref{thm:SRBM_weak_convergence}. 
	
	Define  the following processes
	\begin{align}
		\begin{cases}
			X_{x}:=\textup{The decoupled carrier process over $\xi=\xi^{p}$ associated with $\mathcal{C}_{e}$ satisfying \eqref{eq:stability_color_assumption}} \\
			X_{x}^{s} := \left(\mathbf{1}(i \notin \mathcal{C}_{e}) \, X_{x}(i) \,;\, i=1,\dots,\kappa  \right)  \hspace{2cm} (\textup{$\triangleright$ The `stable part' of $X_{x}$}) \\
			X_{x}^{u} := \left(\mathbf{1}(i \in \mathcal{C}_{e}) \, X_{x}(i) \,;\, i=1,\dots,\kappa  \right)   \hspace{1.97cm} (\textup{$\triangleright$ The `unstable part' of $X_{x}$}).
		\end{cases} \label{eq:decoupled_carrier_on_stable_colors_00}
	\end{align}
	Namely, $X^{s}_{x}$ (resp., $X^{u}_{x}$) agrees with $X_{x}$ on the non-exceptional (resp., exceptional) colors but its coordinates on exceptional (resp., non-exceptional) colors are zero. 
	Clearly, we have the following decomposition 
	\begin{align}
		X_{x} = X_{x}^{s} + X_{x}^{u} \qquad \textup{for all $x\ge 0$}.
	\end{align}
	
	In Lemma \ref{lemma:localized_carrier_convergence}, we will show that $(X^{s}_{x})_{x\ge 0}$ defines an irreducible Markov chain whose empirical distribution converges to its unique stationary distribution $\pi^{s}$ defined as 
	\begin{align}\label{eq:def_stationary_distribution_carrier_localized}
		\pi^{s} \big(n_{1}, \dots, n_{\kappa}\big) = \prod_{j\in \mathcal{C}^{\p}_{u}} \mathbf{1}(n_{j}=0)  \prod_{j=0}^{r}\left[  \prod_{\alpha_{j}<i<\alpha_{j+1}} \left(1-\frac{p_{i}}{p_{\alpha_{j+1}}} \right)  \left(  \frac{p_{i}}{p_{\alpha_{j+1}}} \right)^{n_{i}} \right],
	\end{align}  
	where we set $\alpha_{0}=0=\alpha_{r+1}$. 
	Hence the expression in the bracket above is a non-degenerate geometric distribution. Thus the above is the product of $\kappa-r$ geometric distributions, so it is indeed a probability distribution on $\Omega^{s}$. Comparing \eqref{eq:def_stationary_distribution_carrier_localized} with \eqref{eq:def_stationary_distribution_carrier_sub}, we see that the exceptional color $\alpha_{j+1}$ plays the role of color 0 for the non-exceptional colors in the interval $(\alpha_{j},\dots,\alpha_{j+1})$.

	\begin{lemma}\label{lemma:localized_carrier_convergence}
		Let $(X^{s}_{x})_{x\ge 0}$ be the process defined in \eqref{eq:decoupled_carrier_on_stable_colors_00}. Then it is an aperiodic Markov chain on the state space $\Z_{\ge 0}^{\kappa}$ and has a unique communicating class with unique stationary distribution $\pi^{s}$ defined in \eqref{eq:def_stationary_distribution_carrier_localized}. Furthermore, if we denote the distribution of $X^{s}_{x}$ by $\pi^{s}_{x}$, then 
		\begin{align}\label{eq:carrier_process_localized_markov}
			\lim_{x\rightarrow \infty} d_{TV}(\pi^{s}_{x}, \pi^{s} ) = 0.
		\end{align}
	\end{lemma}
	
	\begin{proof}
		First we show $(X^{s}_{x})_{x\ge 0}$ defines a Markov chain. Clearly the full decoupled carrier process $(X_{x})_{x\ge 0}$ over $\xi=\xi^{\p}$ defines a Markov chain on $\Z^{\kappa}$. Hence it is enough to show that $X_{x+1}^{s}$ is determined from $X_{x}^{s}$ and $\xi_{x+1}$ for each $x\ge 0$. Fix $x\ge 0$ and denote $y:=\xi_{x+1}$. Fix a non-exceptional color $i$. Let $j$ be such that $\alpha_{j}<i<\alpha_{j+1}$. If $y\notin [i,\alpha_{j+1}]$, then $X_{x+1}^{s}(i)=X_{x}^{s}(i)$. If $y=i$, then $X_{x+1}^{s}(i)=X_{x}^{s}(i)+1$. If $y\in (i,\alpha_{j+1}]$, then $X_{x+1}^{s}(i)-X_{x}^{s}(i)=-1$ if $X_{x}^{s}(i)\ge 1$ and $X_{x}^{s}(i+1)=\dots=X_{x}^{s}(\alpha_{j+1}-1)=0$; otherwise $X_{x+1}^{s}(i)-X_{x}^{s}(i)=0$. In all cases, $X_{x+1}^{s}(i)$ is determined by $X_{x}^{s}$ and $y$. Since $i$ was an arbitrary non-exceptional color, this verifies that $(X^{s}_{x})_{x\ge 0}$ is a Markov chain. 
		
		Next, let $\Omega^{s}$ denote the subset of $\Z^{\kappa}_{\ge 0}$ consisting of all points whose coordinates on exceptional colors are zeroed out. Clearly $(X^{s}_{x})_{x\ge 0}$ lives in $\Omega^{s}$. We show the irreducibility of the chain $(X^{s}_{x})_{x\ge 0}$ on $\Omega^{s}$. Aperiodicity will follow from irreducibility by noting that $\mathbf{0}\in \Omega^{s}$ is aperiodic. Observe that $X^{s}_{x}$ visits every state eventually in $\Omega^{s}$  with positive probability starting from the initial state $\mathbf{0}$. Hence it suffices to show the converse transition. Fix $\x=(x_{1},\dots,x_{\kappa})\in \Omega^{s}$. Denote $n_{1}=x_{1}+\cdots+ x_{\alpha_{1}-1}$, which is the number of balls of color in $[1,\alpha_{1})$. Observe that inserting $n_{1}$ balls of color $\alpha_{1}$ into the decoupled  carrier $X_{x}$ removes all balls of colors in $[1,\alpha_{1})$ and leaves with $x_{\alpha_{1}}+n_{1}$ balls of color $\alpha_{1}$. Next, we insert $x_{\alpha_{1}}+n_{1}+n_{2}$ balls of color $\alpha_{2}$ into the decoupled carrier, where $n_{2}=x_{\alpha_{1}+1}+\dots+x_{\alpha_{2}-1}$. This will remove all remaining balls of colors in $[1,\alpha_{2})$ and leave $x_{\alpha_{2}}+(x_{\alpha_{1}}+n_{1}+n_{2})$ balls of color $\alpha_{2}$. Repeating this process, we can remove all balls of stable colors in the decoupled carrier, so $X^{s}_{x}$ visits $\mathbf{0}$ with a positive probability.

		Next, we can verify that $\pi^{s}$ is a stationary distribution of $(X^{s}_{x})_{x\ge 0}$ by using a similar argument as in the proof of Theorem \ref{thm:carrier_subcritical} \textbf{(i)}. The key idea is the following: The evolution of balls of colors in $(\alpha_{j}, \alpha_{j+1})$ in the decoupled carrier $X_{x}$ depends only on balls of colors in $(\alpha_{j}, \alpha_{j+1}]$ and inserting balls of color $\alpha_{j+1}$ can exclude any color in that interval. Moreover, the `stable component' $X_{x}^{s}$ of $X_{x}$ does not count the number of balls of color $\alpha_{j+1}$ and recall the `stability condition' \eqref{eq:stability_color_assumption}. So one can treat $\alpha_{j+1}$ as color 0 in the subcritical carrier. We omit the details. 
		
		Lastly, the convergence of the empirical distribution in \eqref{eq:carrier_process_localized_markov} follows from the same soft argument given at the end of the proof of Theorem \ref{thm:carrier_subcritical} \textbf{(i)}. 
	\end{proof}

	Next, we introduce a representation of the decoupled carrier process as a (truncated) partial sums process. By Lemma \ref{lemma:localized_carrier_convergence},  $(X^{s}_{x},\xi_{x+1})_{x\ge 0}$ defines an aperiodic Markov chain on $\mathbb{Z}^{\kappa}_{\ge 0}\times \{0,\dots,\kappa\}$ with unique stationary distribution $\pi^{s}\otimes \p$. 
	For each $\ell \in \{1,\dots, \kappa\}$, define a functional $g^{\ell}:\Z^{\kappa} \times \{0,\dots,\kappa\} \rightarrow \Z$ by
	\begin{align} \label{eq:def_functional_g_j}
		g^{\ell}(\w, i) := 
		\begin{cases}
			1 & \textup{if $i=\ell$} \\
			%	-1 & \textup{if $i=\ell+1$} \\
			-1 & \begin{matrix} \textup{if $\alpha_{j} \le \ell < i \le \alpha_{j+1}$ for some $j\in \{0,\dots,r-1\}$} \\ 
				\textup{and $\w(\ell+1)=\dots=\w(i-1)=0$} 
			\end{matrix}
			\\
			-1 & \begin{matrix} \textup{if $\alpha_{r} \le \ell $, $i=0$, and  $\w(\ell+1)=\dots=\w(i-1)=0$}  
			\end{matrix}
			\\
			0 & \textup{otherwise},
		\end{cases}
	\end{align}
	where we denoted $\alpha_{0}:=0$. It is easy to verify that, for each $\ell \in \{1,\dots, \kappa\}$ and $x\ge 0$, 
	\begin{align}\label{eq:increment_as_functional0}
		X_{x+1}(\ell) = 
		\begin{cases}
			X_{x}(\ell)   +	 g^{\ell}(X^{s}_{x}, \xi_{x+1})  & \textup{if $\ell\in \mathcal{C}_{e}$} \\
			\max(0, \, X_{x}(\ell)   +	 g^{\ell}(X^{s}_{x}, \xi_{x+1}))  & \textup{if $\ell\notin \mathcal{C}_{e}$}. 
		\end{cases}
	\end{align}
	In words, the random variable $g^{\ell}(X^{s}_{x}, \xi_{x+1})$ gives the increment of $X_{x+1}(\ell)$ for exceptional $\ell$; for non-exceptional $\ell$, the same holds but with additional truncation at 0 to ensure the value of $X_{x}(\ell)$ stays nonnegative. In particular, we can view $X_{x}(\ell)$ for non-exceptional $\ell$ as a Lindley process in queuing theory. 
	
	Another consequence of the observation in \eqref{eq:increment_as_functional0} is that  the decoupled carrier process $X_{x}^{u}$ on the exceptional colors (the unstable component of $X_{x}$) can be written as an additive function of the Markov chain $(X_{x}^{s}, \xi_{x+1})_{x\ge 0}$: 
	\begin{align}\label{eq:unstable_as_additive_functional}
		X^{u}_{x}= \sum_{z=1}^{x} \sum_{\ell\in \{ \alpha_{1},\dots,\alpha_{r} \}} g^{\ell}(X^{s}_{z}, \xi_{z+1})\,  \e_{\ell}. 
	\end{align}
	This representation will be used critically in Sections \ref{sec:prob_decoupled_carrier},  \ref{sec:linear_scaling_carrier}, and \ref{sec:carrier_Skorokhod}. 
	
	In the following proposition, we compute the stationary expectation of the increments $g^{\ell}(X^{s}_{x}, \xi_{x+1}) $ in \eqref{eq:increment_as_functional0}.

	\begin{prop}[Bias of the decoupled carrier] \label{prop:decoupled_carrier_bias}
		Let $g^{\ell}$ be the function in \eqref{eq:def_functional_g_j}. Then 
		\begin{align}
			\E_{\pi^{s}\otimes \p}[ g^{\ell}(X_{x}^{s},\, \xi_{x+1}) ] = p_{\ell} - p_{\ell^{+}},
		\end{align}
		where $\ell^{+}$ is the smallest exceptional color strictly larger than $\ell$. (If $\ell>\alpha_{r}$, then take $\ell^{+}=0$.)
	\end{prop}
	
	\begin{proof}
		Fix $j\in \{0,\dots,r\}$ and and $\alpha_{j}\le \ell < \alpha_{j+1}$. Denote $\ell^{+}:=\alpha_{j+1}$, where we take $\alpha_{0}=0$ and $\alpha_{r+1}=\kappa+1 \equiv 0 \, (\textup{mod} \, \kappa+1)$. Denote $\zeta_{x}:=g^{\ell}(X_{x}^{s},\, \xi_{x+1})$. It is clear from the definition that 
		\begin{align}
			\P_{\pi^{s}\otimes \p}\left( \zeta_{x}=1 \right)  = p_{\ell}. 
		\end{align}
		It remains to show 
		\begin{align}
			\P_{\pi^{s}\otimes \p}\left(\zeta_{x}=-1 \right)  = p_{\ell^{+}}. 
		\end{align}
		To this end, observe that  
		\begin{align}\label{eq:MC_functional_drift_pf}
			\mathbb{P}_{\pi^{s}\otimes \mathbf{p}}( \zeta_{x}  = -1)  &=  p_{\ell+1} +\sum_{i=\ell+2}^{\ell^{+}} \mathbb{P}_{\pi^{s}}( X^{s}_{x}(\ell+1) =\dots =X^{s}_{x}(i-1)=0) \, p_{i}.
		\end{align}
		Since $X^{s}_{x}$ is distributed as the stationary distribution $\pi^{s}$ for all $x\ge 0$, 
		\begin{align}\label{eq:stationary_increment_pf1}
			\E_{\pi^{s} }\left[ \sum_{\ell<i<\ell^{+}} X_{x+1}^{s}(i) -  \sum_{\ell<i<\ell^{+}} X_{x}^{s}(i) \right] =0.
		\end{align}  
		Let $T$ denote the random variable in the expectation above. Then 
		\begin{align}
			\mathbb{P}_{\pi^{s}\otimes \mathbf{p}}( T = -1) &=\left( 1- \mathbb{P}_{\pi^{s} }\left(  \sum_{\ell<i<\ell^{+}} X_{x}^{s}(i) =0 \right)  \right)p_{\ell^{+}}, \\
			\mathbb{P}_{\pi^{s}\otimes \mathbf{p}}( T = 1) &= p_{\ell+1} +  \sum_{i=\ell+2}^{\ell^{+}-1} \mathbb{P}_{\pi^{s}}( X^{s}_{x}(\ell+1) =\dots= X^{s}_{x}(i-1)=0) \, p_{i}.
		\end{align}
		Since $T\in \{-1,0,1\}$  and  \eqref{eq:stationary_increment_pf1} holds, this yields  
		\begin{align}
			p_{\ell^{+}} &=p_{\ell+1} +  \sum_{i=\ell+2}^{\ell^{+}} \mathbb{P}_{\pi^{s}}( X^{s}_{x}(\ell+1) =\dots =X^{s}_{x}(i-1)=0) \, p_{i}.
		\end{align}
		Note that the right-hand side equals $	\mathbb{P}_{\tilde{\pi}\otimes \mathbf{p}}( \zeta_{x}  = -1) $ in \eqref{eq:MC_functional_drift_pf}, as desired. This shows the assertion. 
	\end{proof}

	\subsection{Finite moments of return times of the decoupled carrier process}
	\label{sec:return_decoupled_carrier}
	
	The main goal of this section is to prove Theorem \ref{thm:excursion_length_MGF} below, which shows that the first return time to the origin of the stable part of the decoupled carrier process $(X^{s}_{x})_{x\ge 0}$  has finite moments of all orders. In fact, we prove this result in a more general setting that includes the excursions of $X_{x}(i)$ under the past maximum for exceptional colors $i$ with a positive drift. (Handling such a general setting will be useful in the proof of Proposition \ref{prop:supercritical_excursion_bd}.)
	Define a new process $(\widetilde{X}_{x})_{x\ge 0}$ on $\Z^{\kappa}_{\ge 0}$  by 
	\begin{align}\label{eq:decoupled_carrier_unstable_excursions}
		\widetilde{X}_{x}(i):= \begin{cases}
			X_{x}(i) & \textup{if $i\notin \mathcal{C}_{e}$} \\
			\max_{0\le t\le x} X^{u}_{t}(i) -  X^{u}_{x}(i) & \textup{if $i=\alpha_{j}\in \mathcal{C}_{e}$ for some $j$ and $p_{\alpha_{j}}>p_{\alpha_{j+1}}$} \\
			0 & \textup{otherwise}.
		\end{cases}
	\end{align}
	Notice that $(\widetilde{X}_{x})_{x\ge 0}$ defines a Markov chain on the nonnegative orthant $\Z^{\kappa}_{\ge 0}$. 
	
	\begin{theorem}\label{thm:excursion_length_MGF}
		Let $(\widetilde{X}_{x})_{x\ge 0}$ be the Markov chain on $\Z^{\kappa}_{\ge 0}$ in \eqref{eq:decoupled_carrier_unstable_excursions}. Assume \eqref{eq:decoupled_carrier_on_stable_colors_00} holds. Let $\tau$ denote its first return time to the origin. Then $\tau$ has finite moments of all orders. Furthermore, $(\widetilde{X}_{x})_{x\ge 0}$ is irreducible, aperiodic, positive recurrent and has a unique stationary distribution. 
	\end{theorem}

	We recall the following geometric ergodic theorem for Markov chains on a countable state space. It is an important tool for showing finite exponential moments of return times.

	\begin{theorem}[Geometric Ergodic theorem; Special case of Thm. 15.0.1 in  \cite{meyn2012markov}]\label{thm:geo_ergodic_thm}
		Let $(\mathtt{X}_{t})_{t\ge 0}$ be a Markov chain on a countable state space $\Omega$ with transition kernel $P$, which is irreducible and aperiodic. Then the following conditions are equivalent: 
		\begin{description}
			\item[(i)] There exists a state $x\in \Omega$ such that the return time of the chain to $x$ has a finite exponential moment; 
			
			\item[(ii)] The chain is geometrically ergodic, that is, there exists a function $V: \Omega \rightarrow [1, \infty)$, constant $\eps\in (0,1)$, and a finite set $\mathcal{C}$ such that 
			\begin{align}
				PV(x)  \le (1- \eps) V(x) \quad \textup{for all $x\in \Omega \setminus \mathcal{C}$}.
			\end{align} 
			
		\end{description}
	\end{theorem}

	In order to prove Theorem   \ref{thm:excursion_length_MGF}, 
	we will establish a general lemma on the first return time of Markov chains defined on the nonnegative integer orthant that abstracts important structure of the subcritical carrier process $W_{x}$. Its proof is relegated to the end of this section. 
	
	\begin{lemma}\label{lemma:gen_return_time}
		Let $Z_{x}=(Z_{x}(1),\dots,Z_{x}(d))$  be an aperiodic and irreducible Markov chain on $Z^{d}_{\ge 0}$. Suppose $Z_{0}=\mathbf{0}$ and assume the following three properties:
		\begin{description}
			\item[(A1)] (Geometric ergodicity of top coordinate) The return time of $Z_{x}(d)$ to zero has a finite exponential moment.

			\item[(A2)] (Hierarchical dependence) There is a sequence of i.i.d. random variables $(\xi_{x})_{x\in \N}$ with distribution $\p$ and functions  $f_{i}:\Z^{d-i-1}\times \R \rightarrow \{-1,0,1\}$ such that 
			\begin{align}
				\hspace{1cm} 	Z_{x+1}(i) =\max(0, Z_{x}(i) + f_{i}( Z_{x}^{>i}, \, \xi_{x+1} ))  \quad  \textup{for all $x\in \N$ and  $i\in \{0,\dots,d-1\}$}, 
			\end{align}
			where $Z_{x}^{> i}:=(Z_{x}(i+1),\dots,Z_{x}(d))$. Furthermore, $Z_{x}^{> i}$ has a unique stationary distribution, say  $\lambda^{>i}$.

			\item[(A3)] (Coordinatewise negative drift) For all $i=0,\dots,d-1$, 
			\begin{align}
				\E_{\lambda^{> i} \otimes \p} \left[ f_{i}( Z_{x}^{>i}, \, \xi_{x+1} )  \right] < 0.
			\end{align}
		\end{description}
		Now fix $i\in \{0,\dots,d-1\}$. For each $j\ge 1$, let $\tau_{j}$ be the $j$th return time of $(Z_{x}^{>i})_{x\ge 0}$ to the origin. Then $\tau_{1}$ has finite moments of all orders. Furthermore, denote $R_{j}:= Z_{\tau_{j}}(i)$ for $j\ge 0$. Then $(R_{j})_{j\ge 1}$ is a Markov chain on $\Z_{\ge 0}$ such that there exists constants $c,K>0$ for which 
		\begin{align}\label{eq:negative_conditional_drift_Q}
			\sup_{m\ge K}	\E [R_{1} - R_{0}  \,|\, R_{0}=m ] \le -c.
		\end{align}
		In addition, $(R_{j})_{j\ge 1}$ is geometrically ergodic (see Theorem \ref{thm:geo_ergodic_thm}).  
	\end{lemma}

	We now deduce Theorem \ref{thm:excursion_length_MGF} assuming Lemma \ref{lemma:gen_return_time}. 
	
	\begin{proof}[\textbf{Proof of Theorem \ref{thm:excursion_length_MGF}}]
		Let $\mathcal{C}_{u}^{\p}$ denote the set of unstable colors, which is empty in the subcritical regime $p_{0}>\max(p_{1},\dots,p_{\kappa})$ and non-empty in the critical and the supercritical regimes $p_{0}\le \max(p_{1},\dots,p_{\kappa})$. In the latter case, we let $ \alpha_{1}<\dots <\alpha_{r}$ denote the unstable colors. For each $x\ge 0$, we write $\widetilde{X}_{x}=(Y_{x}(0), Y_{x}(1),\dots, Y_{x}(r))$, where 
		\begin{align}
			Y_{x}(0):= \big( X_{x}(1) ,\cdots, \,X_{x}(\alpha_{1}-1) \big), 
		\end{align}	
		and for each $j\in \{1,\dots, r\}$ (setting $p_{\alpha_{r+1}}=p_{0}$),
		\begin{align}
			Y_{x}(j):= \begin{cases}
				\left( X_{x}(\alpha_{j}) , \, X_{x}(\alpha_{j}+1) ,\cdots, \,X_{x}(\alpha_{j+1}-1) \right)  & \textup{if $p_{\alpha_{j}} =  p_{\alpha_{j+1}}$} \\
				\left( \max_{1\le s \le x} X_{s}(\alpha_{j})- X_{x}(\alpha_{j}) , \, X_{x}(\alpha_{j}+1) ,\cdots, \,X_{x}(\alpha_{j+1}-1) \right) & \textup{if $p_{\alpha_{j}}> p_{\alpha_{j+1}}$}
			\end{cases}
		\end{align}	
		We will show that for each $j$, the return time to the origin of $(Y_{x}(j))_{x\ge 0}$ has finite moments of all orders. Then by an inductive argument (see the proof of Lemma \ref{lemma:gen_return_time}), it follows that the return time of $(\widetilde{X}_{x})_{x\ge 0}$ also has finite moments of all orders. 
		
		Denote $R_{x}:=Y_{x}(j)$.  Note that $R_{x}$ is a Markov chain on $\Z^{\ell^{+}-\ell}_{\ge 0}$ with $\ell=\alpha_{j}$.  We wish to show that the return time to the origin of $R_{x}$ has finite moments of all orders. We will only show this for the case of $p_{\alpha_{j}}>p_{\alpha_{j+1}}$, as a similar and simpler argument will show the desired statement for the case $p_{\alpha_{j}}=  p_{\alpha_{j+1}}$. 
		
		First, consider a partial sums process $S_{n}=\sum_{k=1}^{n}\eta_{k}$, $S_{0}=0$,  where the increments $\eta_{k}$ take values from $\{-1,0,1\}$ and they are not necessarily i.i.d.. Consider the new process $\overline{S}_{n}:=\max_{1\le k \le n} S_{k} - S_{n}$, which measures the height of the excursion of $(S_{k})_{1\le k \le n}$ below the running maximum. Note that $\overline{S}_{n}$  satisfies the following recursion: 
		\begin{align}
			\overline{S}_{n} - \overline{S}_{n-1} = 
			\begin{cases}
				-\eta_{n} & \textup{if $\eta_{n}=-1$ or $\overline{S}_{n-1}\ge 1$},\\
				0 & \textup{if $\overline{S}_{n-1}=0$ and $\eta_{n}\in \{0,1\}$}. 
			\end{cases}
		\end{align}
		Equivalently, we have 
		\begin{align}
			\overline{S}_{n}  = \max(0, \overline{S}_{n-1} - \eta_{n}).
		\end{align}

		Now suppose $\ell=\ell^{+}-1$ so that $R_{x}:= \max_{1\le s \le x} X_{s}(\ell)- X_{x}(\ell)$. In this case, $X_{x}(\ell)$ a simple random walk on $\Z$ with positive drift $p_{\ell}-p_{\ell^{+}}>0$, so $R_{x}$ is a birth-deatch chain on $\Z_{\ge 0}$ with negative drift $p_{\ell^{+}}-p_{\ell}<0$. In this case, the claim follows immediately. Hence we may assume $\ell<\ell^{+}-1$. Notice that $(X_{x}(\ell^{+}-1))_{x\in \N}$ is a birth-deach chain on $\Z_{\ge 0}$ which moves to the right with probability  $p_{\ell^{+}-1}$ and to the left with probability $p_{\ell^{+}}$. Since $\ell<\ell^{+}-1$,  by the choice of $\ell$ and $\ell^{+}$, we have $p_{\ell^{+}-1}<p_{\ell^{+}}$. Hence $X_{x}(\ell^{+}-1)$ has negative drift $p_{\ell^{+}-1}-p_{\ell^{+}}<0$ on $\Z_{>0}$. Thus the return time to the origin of $X_{x}(\ell^{+}-1)$ has a finite exponential moment. This verifies the hypothesis \textbf{(A1)} in Lemma \ref{lemma:gen_return_time}; \textbf{(A2)}  follows from the observation in the previous paragraph and \eqref{eq:increment_as_functional0}; \textbf{(A3)} follows from  Proposition \ref{prop:decoupled_carrier_bias}. Therefore, by Lemma \ref{lemma:gen_return_time} we deduce that the return time to the origin of $R_{x}$ has finite moments of all orders. 
		
		One can easily check the irreducibility of $\overline{X}_{x}$ by using a similar argument as in the proof of Lemma \ref{lemma:localized_carrier_convergence}. Aperiodicity is clear, as one can stay at the origin in one step when a color 0 is encountered. We have established that the return time to the origin of $\overline{X}_{x}$ has finite moments of all orders. This implies that the chain is positive recurrent. Hence the chain has a stationary distribution \cite[Thm. 21.13]{levine2020phase}, and it is unique from the irreducibility and Kac's theorem \cite[Lem. 21.12]{levine2020phase}. 
	\end{proof}

	We now prove Lemma \ref{lemma:gen_return_time}. The argument is soft and inductive in nature.

	\begin{proof}[\textbf{Proof of Lemma} \ref{lemma:gen_return_time}]
		We first claim the following:
		\begin{align}\label{eq:finite_exp_moment_claim1}
			\begin{matrix} 	\textup{For each $i\in \{0,\dots,d-1\}$, the first return time of $(Z_{x}^{>i})_{x\ge 0}$} \\ 
				\textup{to some state $\x$ has a finite exponential moment.}	
			\end{matrix}
		\end{align}
		We show the \eqref{eq:finite_exp_moment_claim1} by induction on $i=d-1,\dots,0$. 
		Fix $i\in \{0,\dots,d-1\}$. The base step for $i=d-1$ is given by the hypothesis \textbf{(A1)}. For the induction step, suppose the first return time of $(\Z_{x}^{> i})_{x\ge 0}$ to some state $\x'$ has a finite exponential moment. Let $\tau_{j}$ denote the $j$th return time of $(Z_{x}^{>i})_{x \ge 0}$ to $\x'$. 
		Consider a new process 
		\begin{align}\label{eq:geometric_ergodicity_pf0}
			(Q_{j}, \x') := (Z_{\tau_{j}}(i), Z_{\tau_{j}}(i+1), \dots, Z_{\tau_{j}}(d)).  
		\end{align}
		By the strong Markov property, this defines a Markov chain $(Q_{j})_{j\ge 1}$ on $\Z_{\ge 0}$. 
		
		\textbf{Step 1. \eqref{eq:negative_conditional_drift_Q} holds for $(Q_{j})_{j\ge 1}$. } We would like to show 
		\begin{align}\label{eq:negative_conditional_drift_QQ}
			\sup_{m\ge K}	\E [Q_{1} - Q_{0}  \,|\, Q_{0}=m ] \le -c
		\end{align}
		for some constants $c,K>0$.  Instead of $Z_{x}(i)$, we consider its `untruncated version' 
		\begin{align}
			\overline{Z}_{x}(i) := \sum_{\ell=1}^{j} f_{i}(Z_{x}^{>i}, \xi_{x+1})
		\end{align}
		with $\overline{Z}_{0}(i)=0$. (Note that $\overline{Z}_{x}(0)=Z_{x}(0)$ by the hypothesis.) Since $(Z^{>i}_{x})_{x\ge 0}$ is a Markov chain by the hypothesis \textbf{(A2)}, by the strong Markov property, excursions from $\x$ for the recurrent chain $Z^{>i}_{x}$ are i.i.d.. Hence $\overline{Q}_{j}:= 	\overline{Z}_{\tau_{j}}(i)$ for $j\ge 1$ forms a random walk, whose increments are i.i.d. and has the same distribution as $\overline{Q}_{1}$. We claim that this random walk has a negative drift: 
		\begin{align}\label{eq:geometric_ergodicity_pf1}
			\E[ \overline{Q}_{1} ]< 0.
		\end{align}
		To see this, first, note that 
		\begin{align}
			\lim_{x\rightarrow\infty} \, \frac{1}{x} \overline{Z}_{x}(i) =  	\E_{\lambda^{>i}\otimes \p}[    f_{i}(Z_{0}^{>i}, \xi_{1})   ]=:\alpha< 0 
		\end{align}
		by the hypothesis \textbf{(A3)}. Since $\tau_{1}, \tau_{2}-\tau_{1}, \tau_{3}-\tau_{2},\dots $ are i.i.d. by the strong Markov property and since $\tau_{1}$ has a finite exponential moment by the induction hypothesis, $\tau_{j}\rightarrow\infty$ almost surely. So  $\overline{Z}_{\tau_{j}}/\tau_{j}\rightarrow \alpha<0$ almost surely. Also, to the strong law of large numbers and the previous results, 
		\begin{align}
			\E[ \overline{Q}_{1} ] = 	\lim_{j\rightarrow \infty} \frac{\overline{Q}_{j}}{j} = 	\lim_{j\rightarrow \infty} \frac{\overline{Z}_{\tau_{j}}}{\tau_{j}} \frac{\tau_{j}}{j} = \alpha \E[\tau_{1}] < 0.
		\end{align}
		This shows the claim. 
		
		Now note that 
		\begin{align}
			\E[ Q_1  - Q_{0} \mid  W_{0} = m]  
			& = \E[ Z_{\tau_{1}}(i)  - Z_0(i) \mid  Z_0(i) = m]  \\
			&= \E[ (Z_{\tau_{1}}(i)  - Z_0(i))  \mathbf{1}_{ \tau_{1} \le m}  \mid  Z_0(i) = m] \\
			&\qquad + \E[ (Z_{\tau_{1}}(i)  - Z_0(i) ) \mathbf{1}_{ \tau_{1} > m}  \mid  Z_0(i) = m] \\
			& = \E[ \overline{Z}_{\tau_{1}}  \mathbf{1}_{ \tau_{1} \le m} ] + \E[ (Z_{\tau_{1}}(i) - Z_0(i)) \mathbf{1}_{ \tau_{1} > m}  \mid  Z_0(i) = m] \\
			& = \E[ \overline{Z}_{\tau_{1}}  ]  -   \E[ \overline{Z}_{\tau_{1}}  \mathbf{1}_{ \tau_{1} > m} ] + \E[ (Z_{\tau_{1}}(i)  - Z_0(i)) \mathbf{1}_{ \tau_{1} > m}  \mid  Z_0(i) = m].
		\end{align}
		For the third equality, we have used the fact that $\tau_{1}\le m$ and $Z_{0}(i)=m$ in conjunction with the hypothesis imply $Z_{x}(i)\ge 0$ for all $0\le x \le \tau_{1}$. Note that $|\overline{Z}_{\tau_{1}}|\le \tau_{1}$ and $\tau_{1}$ has a finite expectation by the induction hypothesis, so $\E[ \overline{Z}_{\tau_{1}}  \mathbf{1}_{ \tau_{1} > m} ]\rightarrow 0$ as $m\rightarrow \infty$ by the dominated convergence theorem. Also,  
		\begin{align}
			\E[ (Z_{\tau_{1}}(i)  - Z_0(i)) \mathbf{1}_{ \tau_{1} > m}  \mid  Z_0(i) = m] \le \E[ (\tau_{1} - m)^{+} \mathbf{1}_{ \tau_{1} > m}   ],
		\end{align}
		so again by the dominated convergence theorem, the above tends to zero as $m\rightarrow\infty$. Since $\E[\overline{Z}_{\tau_{1}}]<0$ by \eqref{eq:geometric_ergodicity_pf1}, we have shown \eqref{eq:negative_conditional_drift_QQ}. 
		
		\vspace{0.1cm}
		\textbf{Step 2. $(Q_{j})_{j\ge 1}$ is geometrically ergodic. }	Next, we show that the Markov chain $(Q_{j})_{j\ge 0}$ on $\Z_{\ge 0}$  is geometrically ergodic. 
		To this end, first note that  $|Q_{j+1} -  Q_j | \le \tau_{j+1}-\tau_{j}$, so it has finite exponential moment by the hypothesis. By the dominated convergence theorem, 
		\begin{align}
			\lim_{\beta\searrow 0} \E\left[    \frac{\exp ( \beta ( Q_{j+1} -  Q_j) ) -1}{\beta}   \,\bigg|\, Q_j=m \right] = \E[  Q_{j+1} -  Q_j \,|\,  Q_j = m].
		\end{align}
		Let $c, K>0$ be the constants in \eqref{eq:negative_conditional_drift_QQ}. Then by choosing sufficiently small $\beta > 0$,  we can find $\eps > 0$ such that  
		\begin{align}
			\E[ \exp ( \beta ( Q_{j+1} -  Q_j) )\mid Q_j = m] \le 1-\eps \qquad \text{$\forall m\ge K$}.
		\end{align}
		So, by taking $V(x)  = \exp( \beta x)$, we have $PV(x)  \le (1-\eps) V(x)$ for all $x$  outside the finite set $\{0, 1, \cdots, K\}$, verifying the geometric ergodicity condition for the chain $Q_{j}$.  
		
		\vspace{0.1cm}
		\textbf{Step 3. Completing the induction step.} By the geometric ergodic theorem (Theorem \ref{thm:geo_ergodic_thm}), the first return time $\sigma$ of the geometrically ergodic chain $(Q_{j})_{j\ge 1}$ to some sate $x'\in \Z_{\ge 0}$ has a finite exponential moment. Denote $\x=(x', \x')\in \Z_{\ge 0}^{d-i}$. We now show that the first return time $S$ of the chain $(Z_{x}^{\ge i})_{x\ge 0}$ to the state $\x$ has a finite exponential moment.  Note that $S=\tau_{\sigma}$. 
		Since $\sigma$ has a finite exponential moment, there exists a constant $c>0$ such that  $\P( \sigma=l)  \le  e^{-cl}$ for all $\ell\ge 1$. Also, by the induction hypothesis, $\tau_{1}$ has a finite exponential moment. Hence there exists $A > 1$ such that  $\E[ A^{\tau_{1}}] < \infty$. By choosing $A$ sufficiently close to 1, and applying dominated convergence, we can assume  $\E[ A^{2\tau_{1}}] \le e^{c/2}$. Now by Cauchy-Schwarz, 
		\begin{align}
			\E[ A^{S}] = \E[ A^{\tau_{\sigma}}] &=  \sum_{l=1}^\infty \E [ A^{ \tau_l}  \mathbf{1}_{ \sigma = l} ]  \le  \sum_{ l=1}^\infty \sqrt{\E [ A^{ 2\tau_l} ] }  \sqrt{\P(  \sigma = l)} \\
			&= \sum_{ l=1}^\infty \sqrt{ \E [ A^{ 2\tau_1} ]^l }  \sqrt{\P(  \sigma = l)} \le \sum_{ l=1}^\infty e^{cl/4}  e^{-cl/2} = \sum_{ l=1}^\infty e^{-cl/4}< \infty.
		\end{align}	
		This shows that $S$ has a finite exponential moment, as desired. Thus far, we have shown \eqref{eq:finite_exp_moment_claim1}. 
		
		\textbf{Step 4. Concluding for the return time to the origin.} 
		Fix $i\in \{0,1,\dots,d-1\}$. By \eqref{eq:finite_exp_moment_claim1}, there exists a state $\x\in \Z_{\ge 0}^{d-i}$  such that the first return time $\tau_{1}$ of $(\Z_{x}^{>i})_{x\ge 0}$ to $\x$ has a finite exponential moment. Thus, $\tau_{1}$ has finite moments of all orders. It is well-known that, for any recurrent and irreducible Markov chain on a countable state space, if for any state $i$ the first moment of the first return time is finite, then this also applies to any other state. This generalizes to moments all orders of the first return time \cite{hodges1953recurrence}. Therefore, we can conclude that the first return time of $(\Z_{x}^{>i})_{x\ge 0}$ to the origin has finite moments of all orders. 
		
		Lastly, let $\sigma_{j}$ denote the $j$th return time of $(Z_{x}^{>i})_{x\ge 0}$ to the origin and denote $R_{j}:=Z_{\tau_{j}}(0)$ for $j\ge 1$. We know that $\sigma_{1}$ has finite moments of all orders. We can repeat Steps 1-2 above for the chain $(R_{j})_{j\ge 1}$ to conclude \eqref{eq:negative_conditional_drift_Q} and its geometric ergodicity. This completes the proof. 
	\end{proof}

	\begin{remark}
		In \cite{aurzada2011moments},
		Aurzada, D\"{o}ring, Ortgiese, and Scheutzow show that having a finite exponential moment for first return times is actually \textit{not} a class property. Hence in the proof of Lemma \ref{lemma:gen_return_time}, knowing that the first return time to some state $\x$ has a finite exponential moment does not necessarily imply that the first return time to the origin also has a finite exponential moment. 
	\end{remark}

	\subsection{Linear and diffusive scaling limit of the decoupled carrier process}
	
	In this section, we establish linear and diffusive scaling limits of the decoupled carrier process. We start with an illustrating example.

	\begin{example}
		Suppose $\mathcal{C}_{e}=\{1,\dots,\kappa\}$ so that all positive colors are exceptional. Denote $\eta_{x}:=X_{x}-X_{x-1}$ for $x\ge 1$. Then  $(\eta_{k})_{k\ge 1}$ are i.i.d. random vectors in $\Z^{\kappa}$ with the following distribution: 
		\begin{align}
			\P\left(\eta_{i} = -\ee_{\kappa}  \right) = p_{0}, \quad	\P\left(\eta_{i} = \ee_{1}  \right) = p_{1},  \quad 	\P\left(\eta_{i} = \ee_{j}-\ee_{j-1}  \right) = p_{j} \,\, \text{for $j=2,\dots,\kappa$} .
		\end{align}
		Then note that 
		\begin{align}
			\bmu:=	\E[\eta_{i}] &= -p_{0} \ee_{\kappa} + p_{1} \ee_{1} + \sum_{j=2}^{\kappa} (\ee_{j}-\ee_{j-1}) p_{j}  \\
			&= \ee_{1} (p_{1}-p_{2}) + \ee_{2}(p_{2}-p_{3}) + \dots + \ee_{\kappa}(p_{\kappa}-p_{0}), \\
			\Sigma:= \E\left[ \eta_{i} \eta_{i}^{T} \right] &=   \ee_{\kappa} \ee_{\kappa}^{T} p_{0} + \ee_{1} \ee_{1}^{T} p_{1}  + \sum_{j=2}^{\kappa}  (\ee_{j}  - \ee_{j-1} ) (\ee_{j} - \ee_{j-1})^{T} p_{j} \\
			&= \begin{bmatrix} \label{eq:cov_mx_ex}
				p_{1} + p_{2} & - p_{2}  & 0  \\
				-p_{2}&  p_{2}+p_{3}  & -p_{3} & 0 \\  
				0 & -p_{3} & p_{3} + p_{4} & &  \\  
				&0&&\ddots& -p_{\kappa} \\	
				&&& - p_{\kappa}&  p_{\kappa} + p_{0} 
			\end{bmatrix}
			.
		\end{align}
		In this case, the decoupled carrier process $(X_{x})_{x\ge 0}$ is a Markov chain on $\Z^{\kappa}$ with the mean and the covariance matrix of the increments $\eta_{x}$ are given as above. Then the linear interpolation of the linear interpolation of the $d$-dimensional process $(\frac{1}{\sqrt{n}} (X_{ n }- n\bmu   )_{n\in \mathbb{N}}$ converges weakly to the $d$-dimensional Brownian motion with covariance matrix $\Sigma$ (see, e.g., \cite[Thm. 1]{doukhan1994functional} and the following remark). Note that $\bmu=\mathbf{0}$  if $p_{0}=p_{1}=\dots=p_{\kappa}=1/(\kappa+1)$, which is a special case of the critical regime for the multicolor BBS (i.e., $p_{0}=\max(p_{1},\dots,p_{\kappa})$). See the simulation in Figure \ref{fig:SRBM_cri} for $\kappa=2$ and uniform ball density. \hfill $\blacktriangle$. 
	\end{example}
	
	Next, we compute the mean and the variance of the increments of the unstable part of the decoupled carrier process.

	\begin{prop}[Mean and limiting covariance matrix]\label{prop:X_u_mean_cov_limiting}
		Let $(X_{x})_{x\ge 0}$ be the decoupled carrier process in \eqref{eq:decoupled_carrier_on_stable_colors_00}.    Denote $\zeta_{x}:=X_{x}^{u}-X_{x-1}^{u}$ for $x\ge 1$. Then the following hold:
		\begin{description}[itemsep=0.1cm]
			
			\item[(i)] We have 
			\begin{align}\label{eq:stat_mean_var_decoupled_carrier}
				&\bmu:=\E_{\pi^{s}\otimes \p}[\zeta_{1}] 
				\equiv \ee_{\alpha_{1}} (p_{\alpha_{1}}-p_{\alpha_{2}}) + \ee_{\alpha_{2}} (p_{\alpha_{2}}-p_{\alpha_{3}})+ \dots + \ee_{\alpha_{r}} (p_{\alpha_{r}}-p_{\alpha_{0}}), \\
				& \E_{\pi^{s}\otimes \p}\left[ \zeta_{1} \zeta_{1}^{T} \right]
				= \sum_{\ell\in \{\alpha_{1},\dots,\alpha_{r}\}} \e_{\ell} \e_{\ell}^{T} \left( p_{\ell} + \sum_{\ell< q \le \ell^{+}} p_{q} \prod_{\ell < j < q} \left( 1-\frac{p_{j}}{p_{\ell^{+}}} \right) \right) \\
				& \hspace{3.5cm} -  \sum_{\ell \in \{\alpha_{1},\dots,\alpha_{r-1} \}}  (\e_{\ell} \e_{\ell^{+}}^{T}  + \e_{\ell^{+}} \e_{\ell}^{T} ) \,  p_{\ell^{+}} \prod_{\ell < j < \ell^{+}} \left( 1-\frac{p_{j}}{p_{\ell^{+}}} \right).
			\end{align}
			\item[(ii)] Define the `limiting covariance matrix' $\Sigma\in \R^{\kappa\times \kappa}$ as 
			\begin{align}\label{eq:def_limiting_cov_mx}
				\Sigma 
				&:= \lim_{n\rightarrow\infty} n^{-1} \E_{\pi^{s}\otimes \p}\left[ (X^{u}_{n}- n \bmu) (X^{u}_{n}- n \bmu)^{T} \right]. 
			\end{align}
			Then $\Sigma$ is well-defined, nonzero, symmetric, and positive semidefinite. 
		\end{description}
	\end{prop}
	
	\begin{proof}

		We first show	\textbf{(i)}. The stationary expectation of $\zeta_{1}$ can be easily verified from Proposition \ref{prop:decoupled_carrier_bias}. Denote $m_{(a,b)}(X_{x}^{s}):=\sum_{a<i<b} X_{x}^{s}(i)$, which is set to zero if $b\le a+1$. From \eqref{eq:unstable_as_additive_functional}, we can write 
		\begin{align}\label{eq:G_unstable_X_increments}
			\zeta_{x}
			&= \sum_{\ell\in \{\alpha_{1},\dots,\alpha_{r} \}} \ee_{\ell} \left( \mathbf{1}(\xi_{x}=\ell) - \sum_{\ell+1\le  q\le \ell^{+}} \mathbf{1}(\xi_{x}=q) \mathbf{1}\left( m_{(\ell, q)}(X_{x-1}^{s}) =0\right) \right).
		\end{align}
		Then it is straightforward to compute
		\begin{align}
			\zeta_{x} \zeta_{x}^{T} 
			&= \sum_{\ell \in \{\alpha_{1},\dots,\alpha_{r} \}} \ee_{\ell} \e_{\ell}^{T} \left[  \mathbf{1}(\xi_{x}=\ell) + \sum_{\ell+1\le q\le \ell^{+}} \mathbf{1}(\xi_{x}=q) \mathbf{1}\left( m_{(\ell, q)}(X_{x-1}^{s}) =0\right) \right] \\
			&\qquad - \sum_{\ell \in \{\alpha_{1},\dots,\alpha_{r-1} \}}  (\e_{\ell} \e_{\ell^{+}}^{T}  + \e_{\ell^{+}} \e_{\ell}^{T} ) \mathbf{1}(\xi_{x} = \ell^{+}) \mathbf{1}\left( m_{(\ell, \ell^{+})}(X_{x-1}^{s}) =0\right).
		\end{align}
		Thus by taking the stationary expectation of  $\zeta_{x}\zeta_{x}^{T}$ in conjunction with  \eqref{eq:def_stationary_distribution_carrier_localized}, we obtain the second identity in \textbf{(i)}.

		Lastly, we show \textbf{(ii)}. Assuming $\Sigma$ is well-defined, that it is symmetric and positive semidefinite is clear from the definition. Next, we argue that $\Sigma$ is well-defined. 
		Let $\bar{\zeta}_{x}:=\zeta_{x} - \E[\zeta_{x}]$. For $i\ge 0$, let $\sigma_{i}$ denote the number of steps that the Markov chain $Z_{x}$ takes until it returns to the origin for the $i+1$st  time, By strong Markov property, $\sigma_{i}$'s are i.i.d.. Furthermore, the excursions of $Z_{x}$ from the origin (that is, $Z_{x}$ restricted on the time intervals $[0,\sigma_{0}]$, $[\sigma_{0},\sigma_{1}],\dots$) are i.i.d.. Furthermore, by Theorem \ref{thm:excursion_length_MGF} and the fact that $\xi_{x}$'s are i.i.d. with distribution $\p=(p_{0},\dots,p_{\kappa})$, $p_{0}>0$ (in fact, we assume $\min(p_{1},\dots,p_{\kappa})>0$), it follows that $\sigma_{0}$ has finite moments of all orders. Hence there exists some $\lambda>0$ such that $\E[\lambda^{\sigma_{0}}]<\infty$. Moreover, by
		Kac's theorem \cite[Lem. 21.12]{levine2020phase}, 
		\begin{align}
			\lim_{n\rightarrow\infty} \, \frac{1}{\E[\sigma_{1}]} = \pi^{s}\otimes \p (\mathbf{0},0) = \pi^{s}(\mathbf{0}) \, p_{0},
		\end{align}
		where the stationary distribution $\pi^{s}$ is explicitly given in \eqref{eq:def_stationary_distribution_carrier_localized}.

		Now consider decomposing the trajectory of $Z_{x}$ into excursions from the origin. Write $s_{i}:= \sum_{k=\sigma_{i}}^{\sigma_{i+1}-1} \bar{\zeta}_{k}\bar{\zeta}_{k}^{T}$. By the strong Markov property, $s_{1},s_{2},\dots$ are i.i.d. and also note that $\E[s_{i}]=\mathbf{0}$.  Denote $\Sigma_{n}:= \E\left[ (X^{u}_{n}- n \bmu) (X^{u}_{n}- n \bmu)^{T} \right]$. Observe that 
		\begin{align}
			\Sigma_{\sigma_{n}}=\E\left[ (X^{u}_{\sigma_{n}}- \sigma_{n} \bmu) (X^{u}_{\sigma_{n}}- \sigma_{n} \bmu)^{T} \right]  %&=	 \E\left[ \left( \sum_{k=1}^{n} \bar{\zeta}_{k} \right)\left( \sum_{k=1}^{n} \bar{\zeta}_{k} \right)^{T} \right] \\
			&=	\E\left[ \left( s_{1}+\dots + s_{n}   \right)\left( s_{1}+\dots + s_{n} \right)^{T} \right] = n \E[s_{1}s_{1}^{T}]. 
			%& = \E\left[ \sum_{i=1}^{n} s_{i} s_{i}^{T}  \right] \\
			%&= \E[ \E[s_{1}s_{1}^{T}] \, T(n) ] +\E[r_{n}r_{n}^{T}]   \\ 
			%&= \E[ T(n) ] \E[s_{1}s_{1}^{T}] +\E[r_{n}r_{n}^{T}].
		\end{align}
		So by the elementary renewal theorem, almost surely, 
		\begin{align}		
			\lim_{n\rightarrow\infty} \frac{1}{\sigma_{n}}\Sigma_{\sigma_{n}} =		\lim_{n\rightarrow\infty} \frac{n}{\sigma_{n}} \E[s_{1}s_{1}^{T}]   =  \frac{\E[s_{1}s_{1}^{T}]}{\E[\sigma_{1}]} = p_{0} \pi^{s}(\mathbf{0}) \, \E[s_{1}s_{1}^{T}].
		\end{align}
		
		To show the convergence holds along the whole sequence, let $T(n)$ denote the total number of visits of $Z_{x}$ to the origin in the first $n$ steps.  Denote $r_{n}:=\sum_{k=\sigma_{T(n)}}^{n} \bar{\zeta}_{k}\bar{\zeta}_{k}^{T}$. Then since $s_{1},\dots,s_{T(n)},r_{n}$ are independent and $\E[s_{i}]=0$, 
		\begin{align}\label{eq:cov_renewal_residual_pf0}
			\Sigma_{n}	
			&=	\E\left[ \left( s_{1}+\dots + s_{T(n)} + r_{n}  \right)\left( s_{1}+\dots + s_{T(n)} + r_{n} \right)^{T} \right] = \Sigma_{\sigma_{T(n)}}+  \E[r_{n}r_{n}^{T}].
		\end{align}
		Denote $\Lambda_{n}:= \E\left[ \sum_{x=1}^{n} \lVert \bar{\zeta}_{x}\bar{\zeta}_{x}^{T} \rVert \right]$, which is non-decreasing in $n$. Then similar argument as before shows that $\frac{1}{\sigma_{n}}\Lambda_{\sigma_{n}}$ converges a.s., and by the monotonicity of $\Lambda_{n}$, an elementary renewal theory argument shows that $n^{-1}\Lambda_{n}$ converges as $n\rightarrow\infty$. Now by Jensen's inequality,  
		\begin{align}\label{eq:cov_renewal_residual_pf1}
			\lVert \E[r_{n}r_{n}^{T}] \rVert \le \Lambda_{n} - \Lambda_{\sigma_{T(n)}}. 
		\end{align}
		Since $\sigma_{T(n)}\le n \le \sigma_{T(n)+1}$ and $\E[\sigma_{1}]<\infty$, it follows that $\sigma_{T(n)}/n\rightarrow 1$ a.s.  as $n\rightarrow\infty$. Hence deviding both sides of \eqref{eq:cov_renewal_residual_pf1} by $n$ and letting $n\rightarrow\infty$ shows that $n^{-1}	\lVert \E[r_{n}r_{n}^{T}] \rVert\rightarrow0 $ as $n\rightarrow\infty$. Then from \eqref{eq:cov_renewal_residual_pf0}, we deduce 
		\begin{align}
			\lim_{n\rightarrow\infty} n^{-1} \Sigma_{n} &= \lim_{n\rightarrow\infty} \frac{\sigma_{T(n)}}{n} \frac{1}{\sigma_{T(n)}} \Sigma_{\sigma_{T(n)}} + \lim_{n\rightarrow\infty}  \E[r_{n}r_{n}^{T}] \\
			&= 		\lim_{n\rightarrow\infty} \sigma_{n}^{-1}  \Sigma_{\sigma_{n}}    \\
			&=p_{0} \pi^{s}(\mathbf{0}) \, \E[s_{1}s_{1}^{T}] \\
			&= \pi^{s}(\mathbf{0}) \, p_{0} \, \E\left[ (\bar{\zeta}_{1}+\dots+\bar{\zeta}_{\sigma_{1}}) (\bar{\zeta}_{1}+\dots+\bar{\zeta}_{\sigma_{1}})^{T}  \right]. 
		\end{align}
		Finally, since $\bar{\zeta}_{x}$'s are uniformly bounded and  $\sigma_{1}$ has a finite expectation, the last expression is a matrix with finite entries by Wald's identity. From this formula, it is also easy to verify that $\Sigma$ is nonzero. 
	\end{proof}

	Now we establish linear and diffusive scaling limits of the decoupled carrier process on unstable colors. This is the main outcome of this section.

	\begin{prop}[Limit theorems for the decoupled carrier process on unstable colors]\label{prop:limit_thms_unstable_colors}
		Let $(X_{x})_{x\ge 0}$ be the decoupled carrier process in \eqref{eq:decoupled_carrier_on_stable_colors_00}.    Denote $\zeta_{x}:=X_{x}^{u}-X_{x-1}^{u}$ for $x\ge 1$.  Then the following hold.
		\begin{description}
			\item[(i)] (SLLN) Almost surely,
			\begin{align}\label{eq:X_u_SLLN_mean}
				\lim_{n\rightarrow \infty} n^{-1} X_{n} = \ee_{\alpha_{1}} (p_{\alpha_{1}}-p_{\alpha_{2}}) + \ee_{\alpha_{2}} (p_{\alpha_{2}}-p_{\alpha_{3}})+ \dots + \ee_{\alpha_{r}} (p_{\alpha_{r}}-p_{\alpha_{0}}):=\bmu.
			\end{align}

			\vspace{0.1cm}
			\item[(ii)] (FCLT) Let $(\overline{X}_{v})_{v\in \R_{\ge 0}}$ denote the linear interpolation of the lattice path $(X_{x} - x \bmu )_{x\in \mathbb{N}}$.  Let $B=(B_{t}\,:\, 0\le t \le 1)$ denote the standard Brownian motion. Then as $n\rightarrow \infty$,  
			\begin{align}
				(n^{-1/2}\overline{X}_{nt} \,;\,  0\leq v \leq 1 ) \Longrightarrow  (B_{t} \,;\,  0\leq t \leq 1) \, \text{ in }\, C([0,1]),
			\end{align}
			where $B=(B_{v}\,:\, 0\le v \le 1)$ is the Brownian motion in $\R^{\kappa}$ with mean zero and covariance matrix $\Sigma$ defined in \eqref{eq:def_limiting_cov_mx}. Here $\Longrightarrow$ denotes weak convergence in $C([0,1])$.
		\end{description}
	\end{prop}
	
	\begin{proof}
		Recall the decomposition $X_{x} = X_{x}^{u} + X_{x}^{s}$. From Lemma \ref{lemma:localized_carrier_convergence} and Theorem \ref{thm:excursion_length_MGF}, we know that $X_{x}^{s}$ is a geometrically mixing Markov chain on a subset of $\mathbb{Z}^{\kappa}_{\ge 0}$ with unique stationary distribution $\pi^{s}$ in \eqref{eq:def_stationary_distribution_carrier_localized}. Hence $n^{-1} X_{n}^{s}$ converges to zero almost surely. Also, the linear interpolation of $(X_{x }^{s})_{x\in \N}$ in diffusive scaling converges almost surely to zero in $C([0,1])$. Thus it is enough to verify \textbf{(i)} and \textbf{(ii)} with $X_{n}$ replaced by $X_{n}^{u}$.

		Recall the Markov additive function representation \eqref{eq:unstable_as_additive_functional} of $X^{u}_{x}$, where the underlying Markov chain $(X^{s}_{x}, \xi_{x})_{x\ge 0}$ has the unique stationary distribution $\pi^{u}\otimes \p$ and is geometrically ergodic (see Theorem \ref{thm:excursion_length_MGF}). Thus \textbf{(i)} follows from the standard Markov chain ergodic theorem for positive Harris chains (see, e.g., \cite[Thm. 17.1.7]{meyn2012markov}). Recall that the limiting covariance matrix $\Sigma$ defined in \eqref{eq:def_limiting_cov_mx} is well-defined and nontrivial by Proposition \ref{prop:X_u_mean_cov_limiting}. Then \textbf{(ii)} follows from the functional CLT for multivariate strongly mixing processes (see, e.g., \cite[Thm. 1]{doukhan1994functional} and the following remark). See also \cite[Thm. 3.1]{rohde2010uniform}. For a 
		functional central limit theorem for additive functionals (univariate) of a positive Harris chain, see \cite[Thm. 17.4.4 and eq. (17.38)]{meyn2012markov}.
	\end{proof}

	\vspace{-0.5cm}
	\section{Proofs of Theorem \ref{thm:carrier_subcritical} \textbf{(ii)} and  Theorem \ref{thm:iid_subcritical}}
	\label{section:subcritical}
	
	We prove Theorem \ref{thm:carrier_subcritical} \textbf{(ii)} and Theorem \ref{thm:iid_subcritical} in this section. Throughout this section, we fix a probability distribution $\mathbf{p}=(p_{0},p_{1},\cdots,p_{\kappa})$ on $\{0,1,\cdots, \kappa\}$, and let $(W_{x})_{x\ge 0}$ be the carrier process in \eqref{eq:W_x_recursion} over the i.i.d. configuration $\xi=\xi^{\mathbf{p}}$.

	\subsection{Strong stability of the subcritical carrier process }
	\label{section:carrier_excursions}

	In order to prove Theorem \ref{thm:carrier_subcritical} \textbf{(ii)}, we need stronger stability properties of the carrier process than what is stated in Theorem \ref{thm:carrier_subcritical}. More specifically, (1) if $W_{0}=\mathbf{0}$, then its first return time to the origin has finite moments of all orders; and (2) if $W_{0}\sim \pi$ and conditional on $\lVert W_{0} \rVert_{1}=N$, it has a uniformly positive probability to visit the origin before it visits `level' $N+1$. These results are established in the following proposition. In the remainder of this section, we will denote $W^{\ge a}_{x}:=(W_{x}(a), \dots, W_{x}(\kappa))$ and $W^{< a}_{x}:=(W_{x}(1), \dots, W_{x}(a-1))$ and use similar notation for $X^{\ge a}_{x}$ and $X^{<a}_{x}$. This is the content of Proposition \ref{prop:drift_sub_total} below, and proving this result is the main goal of this section. 
	
	\begin{prop}\label{prop:drift_sub_total}
		Suppose $p_{0}>p^{*}:=\max(p_{1},\cdots,p_{\kappa})$ and let  $(W_{x})_{x\ge 0}$ be the carrier process over $\xi^{\mathbf{p}}$. The following hold.
		\begin{description}[itemsep=0.1cm]
			\item[(i)] The first return time of $W_{x}$ to the origin has finite moments of all orders.
			
			\item[(ii)] For each $m\in \N$, let $\tau_{m}=\inf\{ x \ge 1\,:\, \lVert W_{x} \rVert_{1} = m \}$. There exists constants $L_{0},N_{0}\ge 1$  and $c_{0}>0$ such that 
			\begin{align}\label{eq:drift_sub_total2}
				\inf_{N\ge N_{0}}	\P_{\pi} \left( \tau_{0}<\min(\tau_{N}, c_{0}N^{2} + L_{0} )  \,|\, \lVert W_{0} \rVert_{1}=N  \right)  > 0. 
			\end{align}
		\end{description}
	\end{prop}

	We prove a series of lemmas in order to prepare for the proof of Proposition \ref{prop:drift_sub_total}.

	\begin{lemma}[Birth-deach chain domination of excursions of the carrier]
		\label{lemma:SRW_coupling_W}
		Let $(W_{x})_{x\ge 0}$ be the carrier process in \eqref{eq:W_x_recursion} and suppose $p_{0}>p^{*}:=\max(p_{1},\dots,p_{\kappa})$.  Fix $a\in \{1,\dots,\kappa\}$ and define a birth-deach chain $(S_{x})_{x\ge 0}$ on $\Z_{\ge 0}$ by $S_{0}:=W_{x}(a)$ and 
		\begin{align}
			S_{x+1} - S_{x}   = \begin{cases}
				1 & \textup{if $\xi_{x+1}=a$} \\
				-1 & \textup{if $\xi_{x+1}=0$ and $S_{x}\ge 1$} \\
				0 & \textup{otherwise}.
			\end{cases}
		\end{align}
		Note that $(S_{x})_{x\ge 0}$ is a birth-death chain on $\Z_{\ge 0}$ with negative drift $p_{a}-p_{0}<0$. 
		For all $x\ge 0$, 
		\begin{align}\label{eq:W_S_domination_pf}
			\lVert W_{x}^{\ge a} \rVert_{1}  \le S_{x} \quad \textup{if $\min_{0\le t \le x}\, W_{t}(a) \ge 1$}.
		\end{align}
	\end{lemma}
	
	\begin{proof}
		The proposition says that as long as the carrier has at least one ball of color $a$, then the total load $\lVert W_{x}^{\ge a} \rVert_{1} $ is dominated by $S_{x}$. This is easy to verify by induction. The inequality could be violated when $W_{x}(a)=0$, since then the total load can increase by inserting balls of color $> a$ while $S_{x}$ does not. 
	\end{proof}

	In the statement and proofs below, we  denote $\P_{\x}(\cdot) = \P(\cdot \,|\, W_{0}=\mathbf{x})$.

	\begin{lemma}[Quadratic first hitting time of the origin of the subcritical carrier]
		\label{lemma:1st_hitting_time_W_poly}
		Let $(W_{x})_{x\ge 0}$ be the carrier process in \eqref{eq:W_x_recursion} and suppose $p_{0}>p^{*}:=\max(p_{1},\dots,p_{\kappa})$. There exists a constant $c>0$ such that 
		\begin{align}
			\inf_{\x\in \Z_{\ge 0}^{\kappa}}	\P_{\x}( \textup{$\lVert W_{x} \rVert_{1}=0$ for some $x\le c \lVert \x \rVert_{1}^{2}$})  > 0.
		\end{align}
	\end{lemma}

	\begin{proof}
		We prove the assertion by induction on $\kappa$. If $\kappa=1$, then the assertion follows easily since $W_{x}$ then is a birth-deach chain on $\Z_{\ge 0}$ with negative bias $p_{1}-p_{0}<0$ (e.g., see Lemma \ref{lemma:drift_hitting_time}). For the induction step, note that $W^{\ge 2}_{x}$ behaves  as the subcritical carrier process with ball colors $\{0,2,3,\dots,\kappa\}$. That is, it evolves by the circular exclusion restricted on colors $\{0, 2,\dots,\kappa\}$ while ignoring balls of color 1. Thus $W^{\ge 2}_{x}$ is a lazy version of a carrier process with subcritical ball density as $\max(p_{2},\dots,p_{\kappa}) < p_{0}$.  Let $\tau_{i}$ for $i=1,2,\cdots$ denote the $i$th  time that $W^{\ge 2}_{x}$ returns to the origin. By the strong Markov property, $\tau_{i+1}-\tau_{i}$ for $i\ge 1$ are i.i.d. and they have finite moments of all orders by Lemma \ref{lemma:gen_return_time}. Also, by the induction hypothesis, there exists a constant $c_{1}>0$ such that 
		\begin{align}\label{eq:W_1st_hitting_pf1}
			\inf_{ \x \in \Z_{\ge 0}^{\kappa-1} }	\P_{\x}( \tau_{1}\le c_{1}\lVert \x \rVert^{2}_{1} )  > 0.
		\end{align}
		
		Denote $Q_{i}:=W_{\tau_{i}}(1)$ for $i\ge 1$. Then $(Q_{i})_{i\ge 1}$ is a Markov chain on $\Z_{\ge 0}$. Denote $\sigma:=\inf\{ i\ge 1\,:\, Q_{i} \le L \}$ where $L\ge 1$ is a  constant. Let $M:=\lVert W_{0} \rVert_{1}$ and let $c_{2}>0$ be a constant to be determined. Introduce the following events: 
		\begin{align}
			E_{1} &:= \{ \tau_{1} \le c_{1}M^{2} \}, \\
			E_{2} &:= \big\{  \max_{1\le k \le \lfloor 2c_{2}M \rfloor} |\tau_{i+1}-\tau_{i}| <M  \big\}, \\
			E_{3} &:=  \{ W_{\tau_{1}}(1) \le 2M \}, \\
			E_{4} &:= \{ \sigma \le c_{2} W_{\tau_{1}}(1) \}.
		\end{align}
		These events depend on constants $M,L,c_{2}>0$ that we will subsequently choose below. Note that 
		\begin{align}
			\sigma_{\tau} = \tau_{1} + \sum_{i=1}^{\sigma-1} (\tau_{i+1} -\tau_{i} ) \le \tau_{1} + \sigma \max_{1\le i \le \sigma} (\tau_{i+1}-\tau_{i}),
		\end{align}
		so $\sigma\le 2c_{2}M$ on  $E_{3}\cap E_{4}$. Hence  $\tau_{\sigma}\le (c_{1}+2c_{2}) M^{2}$ on $E:= \bigcap_{i=1}^{4} E_{i}$. Also note that  $\lVert W_{\tau_{\sigma}} \rVert_{1} = W_{\tau_{\sigma}}(1)=Q_{\sigma} \le L$. Hence  denoting $c:= (c_{1}+2c_{2})\lor 1$,  
		\begin{align}
			\left\{  \textup{$\lVert W_{x} \rVert_{1}\le L$ for some $x\le c M^{2}$}\right\} \supseteq E.
		\end{align}
		Moreover, 
		\begin{align}
			&	\P_{\x}( \textup{$\lVert W_{x} \rVert_{1}=0$ for some $x\le c M^{2} + L$})   \ge p_{0}^{L}	\, \P_{\x}( \textup{$\lVert W_{x} \rVert_{1}\le L$ for some $x\le c M^{2}$} ).
		\end{align}
		Furthermore,  since  $c\ge 1$, 
		\begin{align}
			\inf_{\lVert  \x \rVert_{1}< M}	\P_{\x}( \textup{$\lVert W_{x} \rVert_{1}=0$ for some $x\le c \lVert \x\rVert_{1}^{2} + L$}) \ge p_{0}^{M}.
		\end{align}
		Therefore,  it suffices to show that for some constant $M_{0}\ge 1$, 
		\begin{align}\label{eq:W_first_hitting_zero_pf1}
			\inf_{M\ge M_{0}}	\inf_{\lVert \x \rVert_{1}=M}  \P_{\x}(E) > 0.
		\end{align}
		Since $E_{1}$ has a uniformly positive probability by the induction hypothesis \eqref{eq:W_1st_hitting_pf1}, it is enough to show that $E_{2},E_{3},E_{4}$ have high probaiblity to occur.

		For $E_{2}$, since $\tau_{i+1}-\tau_{i}$ for $i\ge 1$ are i.i.d. and have finite moments of all orders, it follows that $E_{2}$ occurs with a high probability if $M$ is sufficiently large. To see this, note that 
		\begin{align}
			\P(E_{2}) = \left( 1- \P( \tau_{2}-\tau_{1} \ge M ) \right)^{ \lfloor 2c_{2}M \rfloor }  \ge  \left( 1-  \frac{\E[(\tau_{2}-\tau_{1})^{2}]}{M^{2}} \right)^{ \lfloor 2c_{2}M \rfloor }\rightarrow 1 \quad \textup{as $M\rightarrow\infty$.}
		\end{align}

		For $E_{3}$, by Lemma \ref{lemma:SRW_coupling_W}, on the event that $W_{\tau_{1}}(1)>2M $, a negatively biased birth-death chain $(S_{x})_{x\ge 0}$ on $\Z_{\ge 0}$ makes an up-crossing of height at least $M$ in $c_{1} M^{2}$ steps, so 
		\begin{align}
			1-	\P_{\x}(E_{3}) &\le  \P\left( \max_{ 0\le x \le \lfloor c_{1}M^{2} \rfloor } S_{x} > M \,\bigg|\, S_{0}=0 \right).
		\end{align}
		Since $S_{x}$ is a negatively biased simple random walk, the probability in the last expression is exponentially small in $M$. 
		
		For $E_{4}$, by Lemma \ref{lemma:gen_return_time} 
		there are constants $K,c_{3}>0$ such that 
		\begin{align} 
			\sup_{m\ge K}	\E[Q_{1}-Q_{0} \,|\, Q_{0}=m]  \le -c_{3}.
		\end{align}
		By Lemma \ref{lemma:drift_hitting_time}, $\sigma\le c_{4} Q_{1}=c_{4} W_{\tau_{1}}(1)$ occurs with probability at least $1-\frac{1}{c_{3}c_{4}}$  for some constant $c_{4}>0$. Hence by choosing $L\ge K$, $c_{2}\ge c_{4}$, and letting $c_{4}$ sufficiently large,  $E_{4}$ occurs with a high probability. This shows the assertion. 
	\end{proof}

	\begin{lemma}[Growth of (sub-)critical carrier]\label{lemma:growth_critical_carrier}
		Let $(W_{x})_{x\ge 0}$ be the carrier process in \eqref{eq:W_x_recursion} with arbitrary initial state $W_{0}$ and suppose $p_{0}\ge p^{*}:=\max(p_{1},\dots,p_{\kappa})$. Then for each $\eps>0$, almost surely, 
		\begin{align}\label{def:critical_decoupled_carrier_growth}
			\limsup_{n\rightarrow\infty} \, n^{-1}	\max_{0\le x \le n}	\lVert W_{x} \rVert_{1} \le \eps.
		\end{align}
	\end{lemma}
	
	\begin{proof}
		Suppose $W_{0}=(W_{0}(1),\dots, W_{0}(\kappa))$ is arbitrary and write $M:=\lVert W_{0} \rVert_{1}$. 	
		We may prepend to the ball configuration $\xi$ the following sequence: 
		\begin{align}
			(\underbrace{\kappa,\dots,\kappa}_{W_{0}(\kappa)},\, \underbrace{\kappa-1,\dots,\kappa-1}_{W_{0}(\kappa-1)},\dots, \underbrace{1,\dots,1}_{W_{0}(1)}) 
		\end{align}
		and denote the extended configuration $\tilde{\xi}=(\tilde{\xi}_{1},\dots,\tilde{\xi}_{M},\xi_{1},\xi_{2},\dots)$. 
		Let $\tilde{W}$ denote the carrier process with zero initial state run on  $\tilde{\xi}$. Then after scanning the first $M$ in the extended configuration, the new carrier $\tilde{W}$ attains exactly the same state $W_{0}$ (i.e., $\tilde{W}_{M}=W_{0}$) and thereafter it undergoes the same dynamics as $W$ (i.e., $\tilde{W}_{x+M}=W_{x}$ for all $x\ge 0$). Furthermore, $\max_{0\le x \le n} \lVert W_{x} \rVert_{1}\le \max_{0\le x \le n+M} \lVert \tilde{W}_{x} \rVert_{1}$, so it is enough to show the assertion for $\tilde{W}$. For simplicity, below we will denote $\tilde{W}$ and $\tilde{\xi}$ as $W$ and $\xi$, respectively, and assume that the first $M$ entries of $\xi$ may be deterministic.

		Fix $\eps>0$. By Lemmas \ref{lemma:queue_formula_soliton} and \ref{lemma:GK_invariants}, 
		\begin{align}
			\max_{0\le x \le n}	\lVert W_{x} \rVert_{1} = \lambda_{1}(n) = \max_{A_{1}\subseteq [0,n]} L(A_{1}, \xi),
		\end{align}
		where the right-hand side equals the penalized length of the longest non-increasing subsequence in $\xi(n):=(\xi_{0},\xi_{1},\dots,\xi_{n})$. Let $D_{i}(x_{1},x_{2})$ denote the number of $i$'s minus the number of $0$'s in $(\xi_{x_{1}},\xi_{x_{1}+1},\dots, \xi_{x_{2}} )$. If $\lambda_{1}(n)\ge \eps n + M$, then $D_{i}(x_{1},x_{2})\ge \eps n /\kappa$ for some $i$ and $M< x_{1}\le x_{2} \le n$. Note that $D_{i}(x_{1},x_{2})$ is the sum of $x_{2}-x_{1}$ i.i.d. Bernoulli variables with success probability $p_{i}-p_{0}\le 0$. Hence by union bound and Hoeffding's inequality, 
		\begin{align}
			\P(\lambda_{1}(n)\ge \eps n+M) &\le \sum_{i=1}^{\kappa}\sum_{M< x_{1} \le x_{2} \le n} \P(D_{i}(x_{1},x_{2}) \ge \eps n /\kappa ) \\
			&\le \kappa n^{2} \exp(-cn)
		\end{align}
		for some constant $c>0$. By Borel-Cantelli lemma, it follows that $\limsup_{n\rightarrow \infty} \lambda_{1}(n)/n \le  \eps $ almost surely. Then the assertion follows. 
	\end{proof}
	
	We remark Theorem \ref{thm:SRBM_weak_convergence}, which will be proved in Section \ref{sec:carrier_Skorokhod}, establishes the exact asymptotic $\max_{0\le x \le n}\lVert W_{x} \rVert_{1} \sim C\sqrt{n}$ for some constant $C>0$.

	\begin{lemma}[Drift and bound on hitting time]\label{lemma:drift_hitting_time}
		Let $(Y_{t})_{t\ge 0}$ be a Markov chain on $\Z_{\ge 0}$ with transition kernel $P$. Suppose $\E_{x}[|Y_{t}|]<\infty$ for all $x,t\ge 0$ and there exists constants $c,L>0$ such that 
		\begin{align}\label{eq:drift_lem_absorption_time1}
			\E_{x}[Y_{1}-x] \le -c \qquad \textup{for all $x\ge L$}.  
		\end{align}
		Let $\tau:=\inf\{ t\ge 0\,:\, Y_{t} \le L \}$. Then 
		\begin{align}
			\P_{x}(\tau \ge Cx)  \le \frac{1}{c C} \qquad \textup{for all $x\ge 0$ and $C>0$}. 
		\end{align}
	\end{lemma}
	
	\begin{proof}
		For any function $g:\Z_{\ge 0}\rightarrow \R$, denote $Pg(x):=\sum_{y} g(y)P(x,y)$ and $PY:=P\, \textup{id}(Y)$.
		Note that the condition \eqref{eq:drift_lem_absorption_time1} reads 
		\begin{align}\label{eq:drift_lem_absorption_time2}
			Px - x \le -c  \qquad \textup{for all $x\notin [0,L]$}. 
		\end{align}
		Define the compensator $(K_{t})_{t\ge 0}$ of $(Y_{t})_{t \ge 0}$  as $K_{0}=0$ and 
		\begin{align}
			K_{n}:=\sum_{k=0}^{n-1} (PY_{k} - Y_{k}). 
		\end{align}
		Then $Y_{n}-K_{n}$ is a martingale with respect to the natural filtration $(\mathcal{F}_{t})_{t\ge 0}$, $\mathcal{F}_{t}:=\sigma(Y_{0},\dots,Y_{t})$. Also note that by \eqref{eq:drift_lem_absorption_time2}, $K_{n\land \tau} \le -c(n\land \tau)$, for if $k<n\land \tau$, then $PY_{k} - Y_{k}\le -c$. Now using the martingale condition, 
		\begin{align}
			x = \E_{x}[Y_{0}-K_{0}] = \E_{x}[Y_{n\land \tau} - K_{n\land \tau}] \ge  c\, \E_{x}[n\land \tau]. 
		\end{align}
		Now if $n\ge Cx$, then $\{ \tau \ge Cx \} = \{ n\land  \tau \ge Cx \}$. Hence by Markov's inequality, by choosing $n\ge Cx$, we can conclude as  
		\begin{align}
			\P_{x}( \tau \ge Cx)= \P_{x}(n\land \tau \ge Cx) \le \frac{\E_{x}[n\land \tau]}{Cx} \le \frac{1}{cC}.
		\end{align}
	\end{proof}

	We now prove Proposition \ref{prop:drift_sub_total}.

	\begin{proof}[\textbf{Proof of Proposition  \ref{prop:drift_sub_total}}] 
		Part \textbf{(i)} follows immediately from  Theorem \ref{thm:excursion_length_MGF} with $\mathcal{C}_{e}=\emptyset$. Such choice of the set $\mathcal{C}_{e}$ of the exceptional colors satisfy the stability condition \eqref{eq:stability_color_assumption} in the subcritical regime $p_{0}>p^{*}$. 
		
		Next, we show \textbf{(ii)}.  Suppose the maxium ball density $p^{*}$ is achieved at positive colors $i_{1}\le i_{2} \le \dots \le i_{r}$.  That is, 
		\begin{align}\label{eq:carrier_return_pf0}
			p_{0}>p_{i_{1}}=\dots=p_{i_{r}} > \max\{ p_{j}\,:\ 1\le j \le \kappa,\, , j\notin \{i_{1},\dots,i_{r} \}  \}.
		\end{align} 
		Denote $\mathcal{C}^{*}:=\{ i_{1},\dots,i_{r} \}$. Fix $\lambda\in (0,1)$ and define a  set
		\begin{align}\label{eq:X_set_def}
			\mathcal{X}_{\lambda, M}&:= \left\{ \x=(x_{1},\dots,x_{\kappa})\in \Z_{\ge 0}^{\kappa} \,:\, \lVert \x \rVert_{1}=M,\, x_{i_{1}}\ge \lambda M \right\}.
		\end{align}
		We will omit $\lambda$ from the subscript of the above sets unless otherwise mentioned. 
		By Proposition \ref{prop:level_N_stationary_prob_partition_ft}, 
		\begin{align}\label{eq:partition_ft_pf}
			\P_{\pi} ( \lVert W_{0} \rVert_{1}=N) 
			%\sum_{x_{1}+\dots+x_{\kappa}= N} \prod_{i=1}^{\kappa} \left( \frac{p_{i}}{p_{0}} \right)^{x_{i}} 
			= \Theta\left( \binom{N+r-1}{r-1} \left( \frac{p^{*}}{p_{0}} \right)^{N} \right). 
		\end{align}
		Noting that
		\begin{align}
			\P_{\pi}( W_{0}\in \mathcal{X}_{N}  \,|\, \lVert W_{0} \rVert_{1}= N) &= \frac{	\P_{\pi}( W_{0}\in \mathcal{X}_{N}  )}{	\P_{\pi}(  \lVert W_{0} \rVert_{1}= N)} \\
			&\ge \frac{	\P_{\pi}(  \lVert W_{0} \rVert_{1}= N-\lceil \lambda N \rceil  )}{	\P_{\pi}(  \lVert W_{0} \rVert_{1}= N)} \left( \frac{p^{*}}{p_{0}} \right)^{\lceil \lambda N\rceil },
		\end{align}
		it follows that 
		\begin{align}\label{eq:Q_positive_prob0}
			\inf_{N\ge 1}	\P_{\pi}( W_{0}\in \mathcal{X}_{N}  \,|\, \lVert W_{0} \rVert_{1}= N) \ge c_{*}>0
		\end{align}
		for some constant $c_{*}=c_{*}(\lambda)>0$. 
		%It follows that for a fixed integer $L_{0}\ge 1$, 
		%\begin{align}\label{eq:Q_positive_prob}
		%	\inf_{N\ge L_{0}}	\P_{\pi}( W_{L_{0}}\in \mathcal{X}_{N-L_{0}}  \,|\, \lVert W_{0} \rVert_{1}= N)  \ge p_{0}^{L_{0}} c_{*}>0. 
		%\end{align}

		For each $\x\in \Z_{\ge 0}^{\kappa}$, let $\P_{\x}$ denote the law of $(W_{x})_{x\ge 0}$ with $W_{0}=\x$. We claim that there exists constants $L_{0},M_{0}\ge 1$  and $\lambda, c_{0}>0$ such that 
		\begin{align}\label{eq:claim0}
			\inf_{M\ge M_{0}}   \inf_{\x\in \mathcal{X}_{M}}	\P_{\x}\left(  \tau_{0}<\min(\tau_{M+L_{0}}, c_{0}M^{2} ) \right) > 0.
		\end{align}
		Due to \eqref{eq:Q_positive_prob0}, this is enough to conclude \eqref{eq:drift_sub_total2}. Indeed, since $W_{0}\in \mathcal{X}_{N}$ implies $\lVert W_{0} \rVert_{1}=N$, \eqref{eq:Q_positive_prob0} implies 
		\begin{align}
			\P_{\pi}(\cdot \,|\, \lVert W_{0} \rVert_{1} = N ) \ge c_{*} 	\P_{\pi}(\cdot \,|\,  W_{0} \in \mathcal{X}_{N} ) \ge c_{*} \inf_{\x\in \mathcal{X}_{N}} \P_{\pi}(\cdot ). 
		\end{align}
		Also note that, for any integer $L_{0}\ge 1$,  
		\begin{align}
			\P_{\x} \left( 	\tau_{0}<\min(\tau_{N}, c_{0}N^{2} + L_{0} ) \right) 
			&\ge  p_{0}^{L_{0}} \P_{\y} \left( 	\tau_{0}<\min(\tau_{N+L_{0}}, c_{0}N^{2}  ) \right),
		\end{align}
		where $\y\in \mathcal{X}_{N-L_{0}}$ is the carrier state obtained by inserting $L_{0}$ 0's into the carrier with state $\mathbf{x}$. 
		%\begin{align}
		%	\P_{\pi}(\cdot \,|\, \lVert W_{0} \rVert_{1} = N ) =  \frac{\P_{\pi}(\cdot \cap \lVert W_{0} \rVert_{1} = N )}{	\P_{\pi}( \lVert W_{0} \rVert_{1} = N ) } \ge c_{*} \frac{\P_{\pi}(\cdot \cap W_{0} \in \mathcal{X}_{N} )}{	\P_{\pi}( W_{0}\in \mathcal{X}_{N} ) } = c_{*} 	\P_{\pi}(\cdot \,|\,  W_{0} \in \mathcal{X}_{N} )  
		%\end{align}
		This yields 
		\begin{align*}
			\inf_{N\ge N_{0}}	\P_{\pi} \left( 	\tau_{0}<\min(\tau_{N}, c_{0}N^{2} + L_{0} )  \,|\, \lVert W_{0} \rVert_{1}=N  \right)  &\ge c_{*} \inf_{N \ge N_{0}} \inf_{\x\in \mathcal{X}_{N-L_{0}}} \P_{\x} \left( 	\tau_{0}<\min(\tau_{N+L_{0}}, c_{0}N^{2}  ) \right),
		\end{align*}
		where the right-hand side is positive due to \eqref{eq:claim0} by choosing $N_{0}= M_{0}+L_{0}$. 
		
		For the rest of the proof, we will show \eqref{eq:claim0}. 
		Let $a:=i_{1}$,  $\rho:=\inf\{ x\ge 0\,:\, W_{x}(i_{1})=0 \}$,  and $\tau_{0}:=\inf\{x\ge 0\,:\, \lVert W_{x} \rVert_{1}=0   \}$. According to Lemma \ref{lemma:1st_hitting_time_W_poly}, there exists a constant $c_{0}>0$  such that $\tau_{0}\le c_{0} \lVert W_{0} \rVert_{1}^{2}$ with a positive probability. Denote $M:=\lVert W_{0} \rVert_{1}$ and fix $\eps,L>0$.	Define the following events 
		\begin{align}
			A_{1}&:= \{ \tau_{0} \le  c_{0}M^{2}  \}, \\
			A_{2}&:= \left\{ \textup{$\lVert W_{x}^{\ge a} \rVert_{1}\le M+\frac{L}{2}-2\eps x$ for all $x\in [0,\rho]$}  \right\}, \\
			A_{3}&:=  \left\{ \textup{$\lVert W_{x}^{<a} \rVert_{1}\le \frac{L}{2}+\eps x$ for all $x\ge 0$}  \right\}, \\
			A_{4}&:= \left\{ \textup{$\lVert W_{x} \rVert_{1}\le M$ for all $x\in [\rho, \tau_{0}]$}  \right\}.
		\end{align}
		Note that  
		\begin{align}
			\left\{  \textup{$\lVert W_{x} \rVert_{1}$ hits $0$ before it hits $M+L$ for some $x\le c_{0} M^{2}$}\right\} \supseteq A:= \bigcap_{i=1}^{4} A_{i}. 
		\end{align}
		Thus it suffices to show that, for $M_{0}, L$ sufficiently large and $\eps>0$ sufficiently small, 
		\begin{align}\label{eq:claim_000}
			\inf_{M\ge M_{0}}  \inf_{\x\in \mathcal{X}_{M}}	\P_{\x}\left( A \right) > 0.
		\end{align}
		
		To this effect, first note that $A_{1}$ occurs with a uniformly positive probability by  Lemma \ref{lemma:1st_hitting_time_W_poly}. Next, we observe that $A_{2}$ and $A_{3}$ occur with high probability. For $A_{2}$,  
		according to Lemma \ref{lemma:SRW_coupling_W}, $\lVert W_{x}^{\ge a} \rVert_{1} \le S_{x}$  for all $x\in [0,\rho)$, where $(S_{x})_{s\ge 0}$ is a biased random walk on $\Z$ with a negative drift $p_{a}-p_{0}<0$. Let $\rho'$ denote the first time that $(S_{x})_{s\ge 0}$ hits the origin. Then $\rho\le \rho'$ by the coupling, so 
		\begin{align}
			\P_{\x}(A_{2}^{c})  \le  \P\left( S_{x} > S_{0} +\frac{L}{2} -2\eps x\,\, \textup{for some $x\ge 0$} \right).
		\end{align}
		The right-hand side above is the probability that a biased simple random walk on $\Z$ with mean increment $p_{a}-p_{0}+2\eps$ starts at zero and ever reaches height $L/2$. We choose $\eps>0$ small so that $p_{a}-p_{0}+2\eps<0$. Then by gambler's ruin for a negatively biased simple random walk on $\Z$, this probability is exponentially small in $L$. Thus by choosing $L$ large and $\eps>0$  small, we can make $\inf_{M\ge 1, \x\in \mathcal{X}_{M}} \P_{\x}(A_{2})$ 	arbitrarily close to one. 
		
		For $A_{3}$, let $X_{x}$ denote the decoupled carrier process with exceptional colors $\mathcal{C}_{e}=\{a\}$. Then by Proposition \ref{prop:carrier_comparison_localized}, $\lVert W_{x}^{<a}\rVert_{1}\le \lVert \hat{X}_{x}^{<a}\rVert_{1} = \lVert X_{x}^{<a}\rVert_{1}$ for all $x\ge 0$. Note that $\lVert W_{0}^{<a} \rVert_{1}\le (1-\lambda) M$ since $W_{0}\in \mathcal{X}_{M}$. Moreover, note that $X_{x}^{<a}$ behaves exactly as the subcritical carrier process with ball colors in $\{1,\dots,a\}$ and balls of color $a$ acting as the empty box. That is, $X_{x}^{<a}$ evolves by the circular exclusion restricted on colors $\{1,\dots,a\}$ while ignoring balls of colors in $\{a+1,\dots,\kappa,0\}$. Thus $X_{x}^{<a}$ is a lazy version of a carrier process with subcritical ball density as $\max(p_{1},\dots,p_{i_{1}-1}) < p_{i_{1}}$. Thus by Lemma \ref{lemma:growth_critical_carrier}, $\limsup_{n\rightarrow\infty} \max_{0\le x \le n} n^{-1}\lVert X_{x}^{<a} \rVert_{1} \le  \eps $ almost surely. Hence $A_{3}$ occurs with high probability for any fixed $\eps>0$ if $L$ is large enough.

		Next, we show that $\bigcap_{i=1}^{4}A_{i} $ occur with a uniformly positive probability. By definition, $\rho<\tau_{0}$. By the definition of the set $\mathcal{X}_{M}$ in \eqref{eq:X_set_def}, we get $	W_{0}(a) \ge \lambda M$. Since $W_{x}(a)$ can decrease at most by one, it follows that $\rho\ge \lambda M$ almost surely. 
		On $A_{2}\cap A_{3}$, $\lVert W_{\rho} \rVert_{1}\le M(1-\lambda\eps)+L$. Thus
		\begin{align}
			&A_{1} \cap A_{2}\cap A_{3}\cap A_{4}^{c} \\ &\quad \subseteq \{ \textup{$\lVert W_{x} \rVert_{1}$ makes an up-crossing from $M(1-\lambda\eps)+L$ to $M+L$ in $ c_{0} M^{2}$ steps} \} \\
			&\quad \subseteq \bigcup_{1\le i \le \kappa} \{ \textup{$W_{x}(i)$ makes an up-crossing of length $M\lambda\eps/\kappa$ in $ c_{0} M^{2}$ steps} \}.
		\end{align}
		By the coupling \eqref{eq:W_S_domination_pf} in Lemma \ref{lemma:SRW_coupling_W}, the last up-crossing probability is exponentially small in $M$. 
		This shows
		\begin{align}
			\P_{\x}\left(	\bigcap_{i=1}^{4}A_{i}  \right) \ge  	\P_{\x}\left(	\bigcap_{i=1}^{3}A_{i}   \right)  - e^{-O(M)}. 
		\end{align}
		Since $A_{1}$ has uniformly positive probability and $A_{2}\cap A_{3}$ has a high probability, by union bound the above is uniformly positive for $M$ sufficiently large. This finishes the proof.
	\end{proof}

	\subsection{Order statistics of the excursion heights and multi-dimensional Gambler's ruin}

	According to Theorem \ref{thm:carrier_subcritical} \textbf{(i)}, the carrier process $(W_{x})_{x\ge 0}$ in the subcritical regime $p_{0}>\max(p_{1},\cdots,p_{\kappa})$ will visit the origin $\mathbf{0}:=(0,0,\cdots,0)\in (\mathbb{Z}_{\ge 0})^{\kappa}$ infinitely often with finite mean excursion time $\pi(\mathbf{0})^{-1}$. Namely, the number $M_{n}$ of visits of $W_{x}$ to $\mathbf{0}$ during $[1,n]$ (defined in \eqref{eq:def_Mn}) satisfies 
	\begin{align}\label{eq:SLLN_Mn}
		\frac{M_{n}}{n} \rightarrow \pi(\mathbf{0}) = \prod_{i=1}^{\kappa} \left(1-\frac{p_{i}}{p_{0}} \right)  \quad  \text{ a.s.} \quad \text{as $n\rightarrow \infty$}
	\end{align}
	by Theorem \ref{thm:carrier_subcritical} \textbf{(i)} and the Markov chain ergodic theorem.
	
	According to Lemma \ref{lemma:queue_formula_soliton}, the first soliton length $\lambda_{1}(n)$ is essentially the same as the maximum of the first $M_{n}$ excursion heights of the carrier process. Roughly speaking, each excursion height is $O(1)$ with an exponential tail. Since there are $M_n \sim \pi(\mathbf{0}) n$ i.i.d. excursions, their maximum height behaves as $O(\log n)$.
	
	To make this estimate more precise, we analyze the order statistics of the excursion heights of the carrier process during $[1,n]$. For this, let $h_{1:m}\ge h_{2:m}\ge \cdots \ge h_{m:m}$ denote the order statistics of the first $m$ excursion heights $h_{1},\cdots,h_{m}$. The strong Markov property ensures that these excursion heights are i.i.d., so we have 
	\begin{align}\label{eq:CDF_order_statistics}
		\P\{h_{j:m}\leq N\} = \sum_{\ell=0}^{j-1} \binom{m}{\ell} \P(h_{1}\le N)^{m-\ell}\, \P(h_{1}>N)^{\ell}, \quad j=1,\cdots,m.
	\end{align}
	In the simplest case $\kappa=1$, the distribution function of the excursion height $h_{1}$ follows from the standard gambler's ruin probability and is given by
	\begin{equation}
		\label{eq:CDF_height_sub_excursion}
		\P(h_{1}\le N) = \left( 1-\frac{1-2p}{\theta^{N +1}-1}\right)\one(N\ge 0),
	\end{equation}
	where $\theta=p_{0}/p_{1}$ (see \cite[Sec. 4]{levine2020phase}). In order to obtain sharp asymptotics for top soliton lengths in the multicolor case, we need a similar result for a generalized gambler's ruin problem. That is, we need an asymptotic expression of the probability that the  subcritical carrier process reaches `height' $N$ (see \eqref{eq:def_carrier_height_def}) before coming back to the origin.

	However, solving the `carrier's ruin' problem asymptotically for $N\rightarrow \infty$ seems to be a nontrivial problem. 
	The essential issue is that the subcritical carrier process for $\kappa\ge 2$ may have a positive drift on a boundary of its state space. For instance, consider the $\kappa=2$ carrier process as in Figure \ref{fig:MC_diagram}. Assuming $p_{0}>\max(p_{1},p_{2})$, 
	the carrier process has a drift toward the origin in the interior and the right boundary of the state space $\Z_{\ge 0}^{2}$, but this is not necessarily true when there is no ball of color 1 (e.g., consider $\mathbf{p}=(0.4, 0.3, 0.3)$). A standard martingale argument for the gambler's ruin problem for $\kappa=1$ does not seem to readily apply for the general $\kappa\ge 2$ dimensional case for this reason. Another standard approach is the one-step analysis, which is computationally challenging since it involves inverting a large matrix (with blocks of expanding sizes) at every $N$, and one needs to obtain an asymptotic expression of the solution of a $N^{\kappa}  \times N^{\kappa}$ linear equation as $N\rightarrow \infty$. 
	
	Despite the technical difficulties we mentioned above, as stated in Theorem \ref{thm:carrier_subcritical} \textbf{(ii)}, we are able to obtain exact asymptotic expression on the probability that an excursion reaches height $N$ as $N\rightarrow \infty$. Our analysis uses a novel idea of `stationary balancing', which we believe to be useful for solving other multi-dimensional ruin problems.   A major technical component we will use in the proof is Proposition \ref{prop:drift_sub_total}\textbf{(ii)}.

	The following combinatorial observation will be used in the proof of Theorem \ref{thm:carrier_subcritical} \textbf{(ii)} below. It states that if we have $k$ independent geometric random variables of parameters $p_{1}/p_{0},\dots,p_{\kappa}/p_{0}$, and if we condition on their sum being $N$, then the total mass should be concentrated on the most probable colors. We note that in the statement, the $1-\frac{p_{i}}{p_{0}}$ terms are omitted from the product since they are all between $1-\frac{p^{*}}{p_{0}}$ and 1. The proof is given at the end of this section.

	\begin{prop}\label{prop:level_N_stationary_prob_partition_ft}
		Let $p_{0}>p^{*}=\max(p_{1},\dots,p_{\kappa})$. Let $r$ denote the number of $i$s in $\{1,\dots,\kappa\}$  such that $p_{i}=p^{*}$. If $p_{1}=\dots=p_{\kappa}$, then 
		\begin{align}\label{eq:partition_ft_Z_1}
			\sum_{x_{1}+\dots+x_{\kappa}= N} \prod_{i=1}^{\kappa} \left( \frac{p_{i}}{p_{0}} \right)^{x_{i}} = \left( \frac{p^{*}}{p_{0}} \right)^{N}  \binom{N+\kappa-1}{\kappa-1}.
		\end{align}
		Suppose $r<\kappa$ and let $p^{(2)}$ denote the second largest value among $p_{1},\dots,p_{\kappa}$. Then
		\begin{align}\label{eq:partition_ft_Z_2}
			\left( \frac{p^{*}}{p_{0}} \right)^{N}  \binom{N+r-1}{r-1} & \le 	\sum_{x_{1}+\dots+x_{\kappa}= N} \prod_{i=1}^{\kappa} \left( \frac{p_{i}}{p_{0}} \right)^{x_{i}} \le  \left( \frac{p^{*}}{p_{0}} \right)^{N} \binom{N+r-1}{r-1}  \left( \frac{p^{*}}{p^{*}-p^{(2)}}  \right)^{\kappa-r}.
		\end{align}

	\end{prop}
	
	We are now ready to prove Theorem \ref{thm:carrier_subcritical} \textup{\textbf{(ii)}}.
	
	\begin{proof}[\textbf{Proof of Theorem} \ref{thm:carrier_subcritical} \textup{\textbf{(ii)}}]
		Fix two disjoint subsets $A,B\subseteq \Z^{\kappa}_{\ge 0}$. Let $\tau_{i}$ for $i\ge 1$ denote the $i$th time that the Markov chain $(W_{x})_{x\ge 0}$ hits the union $A\cup B$. Then by strong Markov property, the subsequential process $\widetilde{W}_{i}:=W_{\tau_{i}}$ for  $i\ge 1$ is a Markov chain on the state space $A\cup B$. Since $(W_{x})_{x\ge 0}$ is irreducible and aperiodic, so is the restricted chain $(\widetilde{W}_{i})_{i\ge 1}$. So if the restricted chain has a stationary distribution, it has to be unique.  Note that that the following probability distribution $\pi_{A\cup B}$ on $A\cup B$ is a stationary distribution for $(\widetilde{W}_{i})_{i\ge 1}$: 
		\begin{align}\label{eq:carrier_restricted_chain_stationary}
			\pi_{A\cup B}(\x) = \pi(\x) / \pi(A\cup B) \qquad \textup{for $\x\in A\cup B$},
		\end{align}
		where for each subset $R\subseteq \Z_{\ge 0}^{\kappa}$,  we denote $\pi(R ):=\sum_{\y\in R} \pi(\y)$. Here $\pi$ is the stationary distribution for the subcritical carrier process defined in \eqref{eq:def_stationary_distribution_carrier_sub}. This can be justified by using the Markov chain ergodic theorem (see, e.g., \cite[Sec. 2.7.1]{aldous2002reversible}).

		\iffalse
		To see this, we first decompose the trajectory of $(W_{x})_{t\ge 0}$ into excursions from the set $A\cup B$, namely, on the time intervals $[0,\tau_{1}), [\tau_{1},\tau_{2}), \dots$. Let $N_{n}:=\sum_{t=0}^{n} \mathbf{1}(W_{t} \in A\cup B)$ denote the number of times that $W_{x}$ visits $A\cup B$ up to site $n$. Then by Theorem \ref{thm:carrier_subcritical} \textbf{(i)} and the Markov chain ergodic theorem, 
		\begin{align}\label{eq:SLLN_Mn_gen}
			\frac{N_{n}}{n} \rightarrow 	\pi(A\cup B)  \quad  \text{ a.s.} \quad \text{as $n\rightarrow \infty$}.
		\end{align}
		By Markov chain ergodic theorem and \eqref{eq:SLLN_Mn_gen}, 
		\begin{align}
			\frac{1}{N_{n}}	\sum_{s=1}^{n} \mathbf{1}(W_{s} = \x)  = 	\frac{n}{N_{n}} \frac{1}{n}	\sum_{s=1}^{n} \mathbf{1}(W_{s} = \x) \overset{n\rightarrow\infty}{\longrightarrow}  \frac{1}{\pi(A\cup B)} \pi(\x).
		\end{align}
		Also, since $\x\in A\cup B$, again by the Markov chain ergodic theorem, 
		\begin{align}
			\frac{1}{M_{n}}	\sum_{s=1}^{n} \mathbf{1}(W_{s} = \x) = 	\frac{1}{M_{n}}	\sum_{k=1}^{N_{n}} \mathbf{1}(W_{T_{k}} = \x) \overset{n\rightarrow\infty}{\longrightarrow} \pi_{A\cup B}(\x).
		\end{align}
		This verifies \eqref{eq:carrier_restricted_chain_stationary}. 
		\fi

		Let $(W'_{x})_{x\ge 0}$ be a carrier process on the ball configuration $\xi^{\p}$ but initialized as $W'_{0}\sim \pi_{A\cup B}$. If we restrict this chain at hitting times of $A\cup B$, then the restricted chain is stationary with distribution $\pi_{A\cup B}$. That is, if we denote the $i$th time that $W'_{t}$ visits $A\cup B$ as $\tau_{i}'$, then $W_{0}'$ and $W'_{\tau'_{1}}$ has the same distribution $\pi_{A\cup B}$. The key idea is to treat the restricted stationary process $(W'_{\tau_{i}'})_{i\ge 0}$ as if it is a two-state process on $\{ A, B \}$ and then derive a `balance equation' for the mass transport between $A$ and $B$. 
		
		By using \eqref{eq:carrier_restricted_chain_stationary}, 
		\begin{align}
			\P\left( \textup{$W'_{x}$ visits $B$ before $A$} \right) = 	\P\left( \textup{$W'_{\tau_{1}'}\in B$} \right) =\pi_{A\cup B} (B) = \frac{\pi(B)}{\pi(A\cup B)}. 
		\end{align}
		This gives 
		\begin{align}
			\pi_{A\cup B} (B) &=	\P\left( \textup{$W'_{x}$ visits $B$ before $A$} \right) \\
			&=	\P\left( \textup{$W'_{x}$ visits $B$ before $A$},\,  W'_{0}\in A \right)  + \P\left( \textup{$W'_{t}$ visits $B$ before $A$},\,  W'_{0}\in B \right) \\
			&= \P\left( \textup{$W'_{x}$ visits $B$ before $A$} \,\bigg|\,   W'_{0}\in A \right)  \, \pi_{A\cup B}(A) \\
			&\qquad + \P\left( \textup{$W'_{x}$ visits $B$ before $A$} \,\bigg| \,   W'_{0}\in B \right) \, \pi_{A\cup B}(B).
		\end{align} 
		Simplifying using \eqref{eq:carrier_restricted_chain_stationary}, we obtain the following `balance equation' 
		\begin{align}\label{eq:gambler_gen_stationary_magic}
			&	\P\left( \textup{$( W'_{x} )_{x\ge 1}$ visits $B$ before $A$} \,\bigg|\,   W'_{0}\in A \right)  \\
			&\qquad = \P\left( \textup{$( W'_{x} )_{x\ge 1}$ visits $A$ before $B$} \,\bigg| \,   W'_{0}\in B \right)  \frac{\pi(B)}{\pi(A)}.
		\end{align}
		
		Now we specialize in the above result. Take $A=\{ \mathbf{0} \}$ and $B=\{ \x\in \Z_{\ge 0}^{\kappa}\,:\, \lVert \x \rVert_{1} = N \}$.  Note that 
		\begin{align}
			\P\left( \textup{$( W'_{x} )_{x\ge 1}$ visits $B$ before $\mathbf{0}$} \,|\,  W'_{0}=\mathbf{0} \right)  = \P(h_{1}\ge N).
		\end{align}
		Recalling the the formula for $\pi$ in \eqref{eq:def_stationary_distribution_carrier_sub},  it follows that 
		\begin{align}\label{eq:carrier_process_height_formula}
			\P(h_{1}\ge N) 
			&=   \P\left( \textup{$W'_{x}$ visits $\mathbf{0}$ before $B$} \,|\,  W'_{0}\in B \right) \sum_{x_{1}+\dots+x_{\kappa}= N} \prod_{i=1}^{\kappa} \left( \frac{p_{i}}{p_{0}} \right)^{x_{i}},
		\end{align}
		where the sum is over all  integers $x_{1},\dots,x_{\kappa}\ge 0$ that sum to $N$. The above along with Proposition \ref{prop:level_N_stationary_prob_partition_ft} is enough to deduce the upper bound in \eqref{eq:carrier_gambler_formula}.

		To obtain a lower bound of matching order, we need to show that the probability in the right-hand side of \eqref{eq:carrier_process_height_formula} is uniformly positive for all sufficiently large $N$. This requires a substantial analysis, which we have done in proving Proposition \ref{prop:drift_sub_total}. By this result, there exists a constant $\delta>0$ such that 
		\begin{align}\label{eq:carrier_process_height_formula2}
			\liminf_{N\ge 1} \,  \P\left( \textup{$( W'_{x} )_{x\ge 1}$ visits $\mathbf{0}$ before $B$} \,|\,  W'_{0}\in B \right) > \delta >0.
		\end{align}
		Then the assertion follows from \eqref{eq:carrier_process_height_formula},  \eqref{eq:carrier_process_height_formula2}, and Proposition \ref{prop:level_N_stationary_prob_partition_ft}.
	\end{proof}

	\begin{proof}[\textbf{Proof of Proposition \ref{prop:level_N_stationary_prob_partition_ft}}]
		Suppose we have real numbers $a_{1}=a_{2}=\dots=a_{r}\ge a_{r+1} \ge \dots \ge a_{\kappa}>0$. Note that 
		\begin{align}
			\sum_{x_{1}+\dots+x_{\kappa}=N}  a_{1}^{x_{1}}\cdots a_{\kappa}^{x_{\kappa}} &= a_{1}^{N}	\sum_{x_{1}+\dots+x_{\kappa}=N}  \left( \frac{a_{r+1}}{a_{1}}\right)^{x_{r+1}} \cdots \left( \frac{a_{\kappa}}{a_{1}}\right)^{x_{\kappa}}  \nonumber \\
			&= a_{1}^{N} \sum_{q=0}^{N}  \binom{q+r-1}{r-1}   \sum_{x_{r+1}+\dots+x_{\kappa}=N-q}  \left( \frac{a_{r+1}}{a_{1}}\right)^{x_{r+1}} \cdots \left( \frac{a_{\kappa}}{a_{1}}\right)^{x_{\kappa}}. \label{eq:sum_product_estimate1}
		\end{align}
		If $a_{1}=\dots = a_{\kappa}$, then the above expression equals to $a_{1}^{N}  \binom{N+\kappa-1}{\kappa-1}$. Hence, if $p_{1}=\dots=p_{\kappa}$, we get \eqref{eq:partition_ft_Z_1}.
		
		We now assume $a_{1}=\dots=a_{r}>a_{r+1}\ge \dots \ge a_{\kappa}$ for some $r\in \{1,\dots,\kappa\}$. Then  the last expression in \eqref{eq:sum_product_estimate1} is at most 
		\begin{align}
			&a_{1}^{N} \sum_{q=0}^{N}	 \binom{q+r-1}{r-1}  \sum_{x_{r+1}+\dots+x_{\kappa}=N-q} \left(\frac{a_{r+1}}{a_{1}} \right)^{N-q}\\ 
			&\quad =	a_{1}^{N}  \sum_{q=0}^{N} \binom{q+r-1}{r-1}\binom{N-q+\kappa-r-1}{ \kappa-r-1}  \left( \frac{a_{r+1}}{a_{1}}\right)^{N-q}   \\
			%&\quad \le a_{1}^{N} \binom{N+r-1}{r-1}   \left[ 1 +  \binom{\kappa-r}{\kappa-r-1}\left( \frac{a_{r+1}}{a_{1}} \right)  +  \binom{\kappa-r+1}{\kappa-r-1}\left( \frac{a_{r+1}}{a_{1}} \right)^{2}+\cdots  \right]
			&\quad \le a_{1}^{N} \binom{N+r-1}{r-1}   \left[ \sum_{n\ge 0} \binom{n}{\kappa-r-1} \left( \frac{a_{r+1}}{a_{1}} \right)^{n-(\kappa-r-1)} \right].
		\end{align}
		Note that the sum in the bracket above equals 
		\begin{align}
			\left( \frac{a_{1}}{a_{r+1}}\right)^{(\kappa-r-1)} \sum_{n\ge 0} \binom{n}{\kappa-r-1} \left( \frac{a_{r+1}}{a_{1}} \right)^{n} 
			&= \left( \frac{a_{1}}{a_{r+1}} \right)^{(\kappa-r-1)} \frac{\left( \frac{a_{r+1}}{a_{1}} \right)^{(\kappa-r-1)}}{ \left( 1-\frac{a_{r+1}}{a_{1}} \right)^{\kappa-r} } = \left( \frac{a_{1}}{a_{1}-a_{r+1}} \right)^{\kappa-r},
		\end{align}
		where we used the generating function $\sum_{n\ge 0} \binom{n}{k} y^{n} = \frac{y^{k}}{(1-y)^{k+1}}$ (with $\binom{n}{k}=0$ for $n<k$). 
		% The last expression is at most 
		%	\begin{align}
			%		a_{1}^{N} \binom{N+r-1}{r-1}  \left( 1+ \frac{ (a_{r+1}/a_{1})^{\kappa-r-1} }{ (1-(a_{r+1}/a_{1}))^{\kappa-r}  }\right).
			%	\end{align}
		Hence it follows that 
		\begin{align}
			\sum_{x_{1}+\dots+x_{\kappa}= N} \prod_{i=1}^{\kappa} \left( \frac{p_{i}}{p_{0}} \right)^{x_{i}} \le \left( \frac{p^{*}}{p_{0}} \right)^{N} \binom{N+r-1}{r-1} \left( \frac{p^{*}}{p^{*}-p^{(2)}}  \right)^{\kappa-r}.
		\end{align}
		For the lower bound, note that  the last expression in \eqref{eq:sum_product_estimate1} is at least  
		\begin{align}
			&a_{1}^{N} \sum_{q=0}^{N}	 \binom{q+r-1}{r-1}  \sum_{x_{r+1}+\dots+x_{\kappa}=N-q} \left(\frac{a_{\kappa}}{a_{1}} \right)^{N-q}  \\
			&\qquad =	a_{1}^{N}  \sum_{q=0}^{N} \binom{q+r-1}{r-1}\binom{N-q+\kappa-r-1}{ \kappa-r-1}  \left( \frac{a_{\kappa}}{a_{1}}\right)^{N-q}   \ge a_{1}^{N} \binom{N+r-1}{r-1} .
		\end{align}
		Hence we get 
		\begin{align}
			\sum_{x_{1}+\dots+x_{\kappa}= N} \prod_{i=1}^{\kappa} \left( \frac{p_{i}}{p_{0}} \right)^{x_{i}} \ge  \left( \frac{p^{*}}{p_{0}} \right)^{N} \binom{N+r-1}{r-1}.
		\end{align}
		This shows the assertion.
	\end{proof}

	\subsection{Proof of Theorem \ref{thm:iid_subcritical}}

	Now that we have the asymptotic soliton to the `carrier's ruin' problem (Theorem \ref{thm:carrier_subcritical}\textbf{(ii)}), we are ready to obtain sharp scaling limit for the top soliton lengths in the subcritical regime, as stated in Theorem \ref{thm:iid_subcritical}. To do so, we first obtain the following scaling limit of $\mathbf{h}_{j}(n)$ using a similar argument to that developed in \cite{levine2020phase}. For instance, the maximum excursion height $\mathbf{h}_{1}(n)$ of the subcritical carrier process during $[0,n]$ scales like $(1+o(1))\log n$, where its tail follows the Gumbel distribution up to a constant shift. The tail cannot have a tight scaling limit due to a rounding error even in the $\kappa=1$ case, see \cite[Remark 5.5]{levine2020phase}.

	\begin{prop} 
		\label{prop:limit_excursion_heights}
		Suppose $p_{0}>p^{*}:=\max(p_{1},\cdots,p_{\kappa})$. Let $r$ denote the multiplicity of $p^{*}$, $\theta:= p_{0}/p^{*}$, and $\sigma:=\pi(\mathbf{0})>0$ (see \eqref{eq:SLLN_Mn}). Let $\nu_{n}:= (1+\delta_{n}) \log_{\theta}\left( \sigma n /(r-1)! \right)$, where we set  $\delta_{n} := \frac{(r-1) \log  \log_{\theta}\left( \sigma n /(r-1)! \right) + \log(r-1)! }{\log \sigma n / (r-1)!}$. 
		Fix $j\ge 1$ and  $x\in \R$. Then 
		\begin{align}
			& \limsup_{n\rightarrow \infty}  \left[ 	\exp\big(-\frac{C}{(r-1)!} \theta^{-(x-1)} \big)\left( \sum_{\ell=0}^{j-1} \frac{\theta^{-\ell x}}{\ell! (r-1)!} \right)\right]^{-1} \left( \P\left\{ \mathbf{h}_j(n) \leq x+\nu_{n} \right\} + o(1) \right)  \le 1, \\
			&\liminf_{n\rightarrow \infty}  \left[ 	\exp\big(-\frac{\delta}{(r-1)!} \theta^{-(x-1)} \big)	\left( \sum_{\ell=0}^{j-1} \frac{\theta^{-\ell (x-1)}}{\ell! (r-1)!} \right)\right]^{-1} \left( \P\left\{ \mathbf{h}_j(n) \leq x+\nu_{n} \right\} +o(1)\right) \ge 1, 
		\end{align}
		where constants  $\delta>0$ and $C\ge 1$ are as in the  Theorem \ref{thm:carrier_subcritical} \textbf{\textup{(ii)}}.
	\end{prop}
	
	\begin{proof}
		Fix $\eps\in (0,\sigma)$ and let $b_{n}=\lfloor (\sigma-\eps)n \rfloor$. As $M_{n}/n\rightarrow \sigma$ a.s. (see \eqref{eq:SLLN_Mn}), we have that $M_{n}\ge b_{n}$ for all sufficiently large $n$ almost surely. Hence 	for each fixed $x\in \R$, 
		\begin{align}\label{eq:excursion_height_order_stats0}
			\mathbb{P}\left( \mathbf{h}_{j}(n)\le x+\nu_{n} \right) \le \mathbb{P}\left( h_{j:b_{n} }\le x+\nu_{n} \right) + o(1).
		\end{align}
		Furthermore, according to Theorem \ref{thm:carrier_subcritical} \textbf{(ii)}  and \eqref{eq:CDF_order_statistics}, 	
		\begin{align}\label{eq:excursion_height_order_stats}
			& \mathbb{P}\left( h_{j:b_{n} }\le x+\nu_{n} \right)  \\
			&\,\, \le  \sum_{\ell=0}^{j-1} \binom{b_{n}}{\ell} \left( 1- \frac{ C \binom{\lfloor x + \nu_{n} \rfloor+r-1}{r-1}  }{\theta^{\lfloor x+\nu_{n} \rfloor}} \right)^{b_{n}-\ell} \left(  \frac{ \binom{\lfloor x + \nu_{n} \rfloor+r-1}{r-1} }{\theta^{\lfloor x+\nu_{n} \rfloor}} \right)^{\ell}  \\
			&\,\,  = \left( 1- \frac{C \binom{\lfloor x + \nu_{n} \rfloor+r-1}{r-1} }{\theta^{\lfloor  x+\nu_{n} \rfloor+1}} \right)^{b_{n}} \sum_{\ell=0}^{j-1} b_{n}^{-\ell} \binom{b_{n}}{\ell} \left(1- \frac{C \binom{\lfloor x + \nu_{n} \rfloor+r-1}{r-1} }{\theta^{\lfloor x+\nu_{n} \rfloor}}\right)^{-\ell}
			\left(\frac{\binom{\lfloor x+\nu_{n} \rfloor+r-1}{r-1} b_{n} }{\theta^{ \lfloor x+\nu_{n} \rfloor }}\right)^{\ell}.	
		\end{align}
		Since $\nu_{n}= (1+\delta_{n}) \log_{\theta}\left( \sigma n /(r-1)! \right)$, note that 
		%\begin{align}
		%	\log \nu_{n}^{r-1} &= (r-1) \log (1+\delta_{n}) + (r-1) \log  \log_{\theta}\left( \sigma n /(r-1)! \right) \\
		%		\log \theta^{\nu_{n}} &= (1+\delta_{n}) \log \left( \sigma n /(r-1)! \right) 
		%\end{align}
		\begin{align}
			\log \left( \frac{ \sigma \nu_{n}^{r-1} n }{\theta^{\nu_{n}}}  \right) %&=  (r-1) \log (1+\delta_{n}) + (r-1) \log  \log_{\theta}\left( \sigma n /(r-1)! \right) + \log \sigma n  \\
			%&\qquad  - (1+\delta_{n}) \log ( \sigma n)   + (1+\delta_{n})\log (r-1)!  \\
			%&= (r-1) \log (1+\delta_{n}) + (r-1) \log  \log_{\theta}\left( \sigma n /(r-1)! \right) -\delta_{n} \log \sigma n \\
			%&\qquad  +(1+\delta_{n})\log (r-1)!  \\
			&= (r-1) \log (1+\delta_{n}) + (r-1) \log  \log_{\theta}\left( \sigma n /(r-1)! \right) + \log(r-1)! \\
			&\qquad + \delta_{n}\left( \log (r-1)! - \log \sigma n \right) \\
			&= (r-1) \log (1+\delta_{n})  \rightarrow 0 \qquad \textup{as $n\rightarrow\infty$},
		\end{align}
		where the second equality uses the definition of $\delta_{n}$ and the limit follows since $\delta_{n}=o(1)$. 
		%\begin{align}
		%	\log \left( \frac{ \sigma \nu_{n}^{r-1} n }{\theta^{\nu_{n}}}  \right) &=  (r-1) \log (1+\delta_{n}) + (r-1) \log \log n - \delta_{n} \log n.\\
		%	&=  (r-1) \log (1+\delta_{n})  \rightarrow 0 \,\, \textup{as $n\rightarrow \infty$}.
		%\end{align}
		Using Stirling's approximation, $\binom{N+r-1}{r-1}=(1+o(1)) N^{r-1}/(r-1)!$ as $N\rightarrow \infty$, we get 
		\begin{align}
			\lim_{n\rightarrow\infty} \frac{ \theta^{\lfloor x+\nu_{n} \rfloor} }{\theta^{\nu_{n}}}	\log \left( 1- \frac{C \binom{  \lfloor x+\nu_{n} \rfloor  + r-1 }{r-1} }{\theta^{  \lfloor x+\nu_{n} \rfloor  }}  \right)^{b_{n}}  	&= - \lim_{n\rightarrow\infty}	 \frac{ b_{n} C \binom{  \lfloor x+\nu_{n} \rfloor + r-1 }{r-1} }{\theta^{ \nu_{n} }}  \\
			& =	 - \lim_{n\rightarrow\infty}  \frac{1}{(r-1)!} \left(\frac{\lfloor x + \nu_{n} \rfloor }{\nu_{n}} \right)^{r-1}	 \frac{ b_{n} C \nu_{n}^{r-1} }{  \theta^{\nu_{n} }  }   \\
			& =	 -\frac{C(1-\frac{\eps}{\sigma}) }{ (r-1)! }.
		\end{align}
		Similarly, 
		\begin{align}
			\lim_{n\rightarrow\infty} \frac{ \theta^{\lfloor x+\nu_{n} \rfloor} }{\theta^{\nu_{n}}}	\frac{\binom{\lfloor x+\nu_{n} \rfloor+r-1}{r-1} b_{n} }{\theta^{ \lfloor x+\nu_{n} \rfloor }} = \frac{(1-\frac{\eps}{\sigma}) }{ (r-1)! }.
		\end{align}
		Writing $\eta_{n}=(x+\nu_{n})-\lfloor x + \nu_{n} \rfloor \in [0,1)$, since $\theta>1$, 
		\begin{align}\label{eq:subcritical_tale_limits0}
			\theta^{x-1}	\le	\frac{ \theta^{\lfloor x+\nu_{n} \rfloor} }{\theta^{\nu_{n}}} = \theta^{x}	\theta^{-\eta_{n}}  \le \theta^{x}. 
		\end{align}
		Also note that $\lim_{n\rightarrow \infty}  b_{n}^{-\ell} \binom{b_{n}}{\ell} = \frac{1}{\ell !}$. From the above computations, we deduce 
		\begin{align}\label{eq:subcritical_tale_limits}
			&\limsup_{n\rightarrow \infty} \exp\left( \frac{C}{(r-1)!}\left( 1-\frac{\eps}{\sigma}\right)\theta^{-x+1 }\right)\left( 1- \frac{C \binom{  \lfloor x+\nu_{n} \rfloor  + r-1 }{r-1} }{\theta^{  \lfloor x+\nu_{n} \rfloor  }}  \right)^{b_{n}}   \le 1,\\
			%&\liminf_{n\rightarrow \infty} \exp\left( \frac{C}{(r-1)!}\left( 1-\frac{\eps}{\sigma}\right)\theta^{-(x) }\right)\left( 1- \frac{C \binom{  \lfloor x+\nu_{n} \rfloor  + r-1 }{r-1} }{\theta^{  \lfloor x+\nu_{n} \rfloor  }}  \right)^{b_{n}}   \ge 1,\\
			%& \limsup_{n\rightarrow\infty} \left( 1-\frac{\eps}{\sigma}\right)\theta^{x+1} \frac{ b_{n}}{\theta^{ \lfloor x+\nu_{n} \rfloor }} \le 1, \\
			%& \liminf_{n\rightarrow\infty} \left( 1-\frac{\eps}{\sigma}\right)\theta^{x} \frac{ b_{n}}{\theta^{ \lfloor x+\nu_{n} \rfloor }} \ge 1.
			&\limsup_{n\rightarrow\infty}  \frac{(r-1)!}{1-\frac{\eps}{\sigma}} \theta^{x-1}	\frac{\binom{\lfloor x+\nu_{n} \rfloor+r-1}{r-1} b_{n} }{\theta^{ \lfloor x+\nu_{n} \rfloor }} \le 1. 
		\end{align} 
		Then we obtain
		\begin{align}
			&\limsup_{n\rightarrow \infty} \left[ \exp\left(-\frac{C}{(r-1)!}\left( 1-\frac{\eps}{\sigma}\right)\theta^{-(x-1)}\right) \left( \sum_{\ell=0}^{j-1} (1+\eps)\left( 1-\frac{\eps}{\sigma}\right)^{\ell} \frac{\theta^{-\ell (x-1)}}{\ell ! (r-1)!} \right) \right]^{-1}	\\
			&\hspace{6cm} \times \mathbb{P}\left( h_{j:b_{n}}\le x+\nu_{n} \right)\\
			&\qquad \le 1.
		\end{align}
		Therefore letting $\eps\searrow 0$ and using \eqref{eq:excursion_height_order_stats0} give the limsup in the statement. A similar argument using $b_{n}=\lceil (\sigma + \eps )n\rceil$ shows the liminf in the statement. 
	\end{proof}

	Now we are ready to establish sharp scaling for the top soliton lengths in the subcritical regime.

	\begin{proof}[\textbf{Proof of Theorem} 
		\ref{thm:iid_subcritical}]
		
		Let $\nu_{n}$ be as in Theorem \ref{thm:iid_subcritical}. Note that 
		\begin{align}
			\nu_{n}=  \log_{\theta} n +  (r-1)\log_{\theta} \log n + c + o(1)
		\end{align}
		for some constant $c$. Hence the asymptotic \eqref{eq:iid_sub_thm0} for $\lambda_{j}(n)$ follows from \eqref{eq:iid_sub_thm1}.
		
		Now we derive \eqref{eq:iid_sub_thm1}. Fix $j\geq 1$ and $x\in\R$.  Then by Proposition  \ref{prop:limit_excursion_heights}, 
		\begin{align}
			&	\liminf_{n\rightarrow \infty}\P\left( \mathbf{h}_{1}(n) \leq x+\nu_{n} \right)
			\ge \exp\left(-\delta \theta^{-x} \right), \label{eq:iid_sub_pf0} \\
			&	\limsup_{n\rightarrow \infty}\P\left( \mathbf{h}_{j}(n) \leq x+\nu_{n} \right)
			\leq  \exp\left(-\frac{C}{(r-1)!} \theta^{-(x-1)} \right)\sum_{k=0}^{j-1} \frac{\theta^{-k(x-1)}}{k! (r-1)! }. \label{eq:iid_sub_pf1}
		\end{align}
		Moreover, recall the quantities $M_{n}$ and $r_{n}$ in \eqref{eq:def_Mn} and \eqref{eq:def_excursion_heights}, respectively. By Lemma \ref{lemma:queue_formula_soliton}, 
		\begin{align}
			\mathbf{h}_{1}(n) = 	\max\{ h_{1},\cdots, h_{M_{n}} \}	\le \lambda_{1}(n) \le \max\{ h_{1},\cdots, h_{M_{n}+1} \}. 
		\end{align}
		Also, note that 
		\begin{align}
			0\le 	\P( \mathbf{h}_{1}(n) \le x + \nu_{n})  - 			\P( \max\{ h_{1},\cdots, h_{M_{n}+1} \}  \le x + \nu_{n})  \le \P(h_{M_{n}+1} > \mathbf{h}_{1}(n)) = o(1). 
		\end{align}
		It follows that 
		\begin{align}
			\P(\lambda_{1}(n) \le x + \nu_{n}) = \P( \mathbf{h}_{1}(n) \le x + \nu_{n}) + o(1). 
		\end{align}
		Moreover, since $\lambda_{1}(n) \ge \lambda_{j}(n) \ge \mathbf{h}_{j}(n)$ by Lemma \ref{lemma:queue_formula_soliton}, 
		\begin{align}
			\P\left( \lambda_{1}(n) \leq x +\nu_{n} \right) \le \P\left( \lambda_{j}(n) \leq x +\nu_{n} \right) \leq \P\left( \mathbf{h}_{j}(n) \leq x +\nu_{n} \right). 
		\end{align}
		Then \eqref{eq:iid_sub_pf0}-\eqref{eq:iid_sub_pf1} show \eqref{eq:iid_sub_thm1}, as desired. 
	\end{proof}

	\section{The linear scaling limit of the carrier process}
	\label{sec:linear_scaling_carrier}

	In this section, we prove Theorem \ref{thm:SRBM_weak_convergence} \textbf{(i)}, concerning the linear scaling limit of the carrier process $W_{x}$ in \eqref{eq:W_x_recursion}.

	Throughout this section, we assume $p_{0}\le p^{*}=\max(p_{1},\dots,p_{\kappa})$. In this case the set $\mathcal{C}^{\p}_{s}$ of unstable colors (defined above the statement of Theorem \ref{thm:SRBM_weak_convergence}) is nonempty. Let $\alpha_{1}<\dots<\alpha_{r}$ denote the unstable colors. Let $(X_{x})_{x\ge 0}$ be the decoupled carrier process  in \eqref{eq:decoupled_carrier_on_stable_colors_00} with  $\mathcal{C}_{e}=\mathcal{C}_{u}^{\p}$. Recall the process $(\hat{X}_{x})_{x\ge 0}$ in  \eqref{eq:hat_X_def}.

	We first show that the coordinate $X_{x}(\ell)$ for $\ell$ an unstable color of supercritical density behaves like a random walk with a positive drift. 
	
	\begin{prop}\label{prop:supercritical_excursion_no_return}
		Fix $j\in \{1,\dots,r\}$ and denote  $\ell:=\alpha_{j}$, $\ell^{+}:=\alpha_{j+1}$. If $p_{\ell}>p_{\ell^{+}}$, then  $\bar{M}:=   -\inf_{k\in \N}  X_{k}(\ell)$ has a finite exponential moment. 
	\end{prop}
	
	\begin{proof}
		Recall that $(X_{x}(\ell))_{x\ge 0}$ is a Markov additive functional with increments $g^{\ell}(X_{x}^{s},\, \xi_{x+1})$  (see \eqref{eq:increment_as_functional0}). Under the hypothesis, it has a positive bias $\E_{\pi^{s}\otimes \p}[g^{\ell}(X_{x}^{s},\, \xi_{x+1})] = \alpha:=p_{\ell}-p_{\ell^{+}}>0$ (see Proposition \ref{prop:decoupled_carrier_bias}). Hence one can expect that $X_{x}(\ell)$ will essentially behave as a simple random walk on $\Z$ with a positive bias. Since $\bar{M}$ measures the height of the excursion of $X_{x}(\ell)$ below the $x$-axis, it should have a finite exponential moment. Below we give a rigorous justification. 
		
		Consider the Markov chain 	
		\begin{align}
			Y_{x}:=\big( X_{x}(\ell+1) ,\cdots, \,X_{x}(\ell^{+}-1) \big).
		\end{align}	
		Let $\tau_{j}:=j$ for $j\ge 0$ if $\ell+1=\ell^{+}$; Otherwise, let $\tau_{j}$ be the $j$th return time of $\big(X_{x}(\ell+1) ,\cdots, \,X_{x}(\ell^{+}-1) \big)$ to the origin. By strong Markov property, $\tau_{1}, \tau_{2}-\tau_{1},\tau_{3}-\tau_{2},\dots$ are i.i.d., and they have finite moments of all orders by Lemma \ref{lemma:gen_return_time}. 	Let $R_{j}:= X_{\tau_{j}}(\ell)$ for $j\ge 1$. Then $(R_{j})_{j\ge 1}$ is a random walk. Let $\eta_{i}:=R_{i}-R_{i-1}$ denote the increments.  It has a positive drift as 
		\begin{align}
			\E[\eta_{1}] =	\lim_{j\rightarrow \infty} \frac{R_{j}}{j} = 	\lim_{j\rightarrow \infty} \frac{X_{\tau_{j}}(\ell)}{\tau_{j}} \frac{\tau_{j}}{j} = \alpha\,  \E[\tau_{1}] > 0,
		\end{align}
		where the first two equalities use the strong law of large numbers and the Markov chain ergodic theorem. 
		
		Next, we claim that $X_{x}(\ell)$ returns to the origin only finitely many times almost surely. First note that by the strong law of large numbers $n^{-1}R_{n}\rightarrow \alpha>0$ almost surely. Hence $n^{-1} R_{n}>\alpha/2$ infinitely often almost surely. 	Note that for each $j\ge 1$, since $\tau_{j+1}-\tau_{j}$ is independent from $R_{j}$ and has the same distribution as $\tau_{1}$, by Chebyshev's inequality, 
		\begin{align}\label{eq:excursion_regurn_pf1}
			\P\left( \textup{$X_{x}(\ell)=0$ for some $x\in [\tau_{j}, \tau_{j+1})$} \right) &\le \P(R_{j} \le \tau_{j+1}-\tau_{j} ) \\
			&	\le \E[ \P( \tau_{j+1}-\tau_{j}  \ge R_{j} \,|\, R_{j}) ] \\ 
			&	\le \E[ R_{j}^{-2} \, \E[\tau_{1}^{2}] ] \\ 
			&	\le \E[\tau_{1}^{2}]\left(  (\alpha j/2)^{-2}  + c \P(R_{j}\le (\alpha/2)j) \right), 
		\end{align}
		where the last inequality follows by partitioning on two cases depending on $R_{j}\le (\alpha/2)j$ or $R_{j}> (\alpha/2)j$. If we denote $\overline{\eta}_{i}:=\E[\eta_{i}]-\eta_{i}$, then $\overline{\eta}_{i}$'s are mean zero i.i.d., so, noting that $\E[R_{n}]=\alpha n$, 
		\begin{align}
			\P(R_{n}\le (\alpha/2)n) &\le  \P( \E[R_{n}]-R_{n} \ge (\alpha/2)n) \\
			&\le   \P\left( \sum_{i=1}^{n} -\overline{\eta}_{i} \ge (\alpha/2)n \right) \\
			&\le   \P\left( \left( \sum_{i=1}^{n} -\overline{\eta}_{i} \right)^{4} \ge (\alpha/2)^{4}n^{4} \right) \\ 
			&\le \frac{ C (\E[\bar{\eta}_{1}^{2}]^{2} + \E[\bar{\eta}_{1}^{4}] ) }{ n^{2}}
		\end{align}
		for some constant $C>0$. Note that for the last inequality, we have used Chebyshev's inequality along with the fact that only the $O(n^{2})$ terms of the form $\bar{\eta}_{i}^{2}\bar{\eta}_{j}^{2}$ for $i\ne j$ and $\bar{\eta}_{i}^{4}$ have nonzero expectations. Sincet $|\eta_{1}|\le \tau_{1}$ has a finite moments of all orders, so does $\bar{\eta}_{1}$. Thus \eqref{eq:excursion_regurn_pf1} implies 
		\begin{align}
			\sum_{j\ge 1}  	\P\left( \textup{$Z_{x}(0)=0$ for some $x\in [\tau_{j}, \tau_{j+1})$} \right) <\infty. 
		\end{align}
		By the  Borel-Cantelli lemma, it follows that $X_{x}(\ell)$ visits the origin only finitely many times almost surely. This shows the claim. 
		
		Now we conclude that $\bar{M}$ has a finite exponential moment.  For this, we use the general result by Hansen  \cite{hansen2006maximum} about the running maximum of a random walk with negative drift, that if the running maximum is uniformly bounded almost surely, then the supremum of the running maximum has a finite exponential moment. We apply this result to the random walk $(-R_{j})_{j\ge 1}$. According to the claim, it follows that $\sup_{x\ge 0} -X_{x}(\ell)=-\inf_{x\ge 0} X_{x}(\ell)$ is almost surely finite. Hence $\sup_{j\ge 1} -R_{j}$ is almost surely finite, so by \cite[Thm. 2.1]{hansen2006maximum}, $\sup_{j\ge 1} -R_{j}=-\inf_{j\ge 1} R_{j}$ has a finite exponential moment. Since the increments of $R_{j}$ have finite exponential moments, we can conclude that $-\inf_{x\ge 0} X_{x}(\ell)$ also has a finite exponential moment. 
	\end{proof}

	\begin{prop}\label{prop:supercritical_excursion_bd}
		%Suppose $p_{0}\le p^{*}$. 
		Let $j\in \{1,\dots,r\}$ be arbitrary with $\ell:=\alpha_{j}$, $\ell^{+}:=\alpha_{j+1}$, and $p_{\ell}>p_{\ell^{+}}$.  Then for each integer $d\ge 1$, there exists a  constant $c>0$ such that for all $n\ge 1$ and $s>0$, 
		\begin{align}\label{eq:X_excursion_bd_1}
			&\P\left(  \max_{0\le t \le n}X_{t}(\ell) - X_{n}(\ell)  \geq   s \right)\le \exp(-c s ),\\
			& \P\left( \left| \max_{0\le t\le n} \hat{X}_{t}(\ell)- X_{n}(\ell)  \right| \ge  \eps \right) \le \exp(-c s ). \nonumber
		\end{align}
	\end{prop}
	
	\begin{proof}
		Consider the following Markov chain 
		\begin{align}\label{eq:decoupled_carrier_flipped}
			Y_{x}:=\big( \max_{1\le s \le x} X_{s}(\ell)- X_{x}(\ell), \, X_{x}(\ell+1) ,\cdots, \,X_{x}(\ell^{+}-1) \big)
		\end{align}	
		on $\Z_{\ge 0}^{\ell^{+}-\ell}$. Note that $Y_{0}=\mathbf{0}$.	Let $\tau$ denote the first return time of $Y_{x}$ to the origin. In  Theorem \ref{thm:excursion_length_MGF}, we have shown that $\tau$ has finite moments of all orders. 
		Let $L_{1},L_{2},\dots$ denote the lengths of excursions of $Y_{x}$ to the origin. Since $L_{k}\ge 1$ for all $k\ge 1$, $M_{n}\le n$. 
		%By the strong Markov property, these are i.i.d. random variables distributed as $\tau$. Hence if we denote $M_{n}:=\{ r\ge 1 \, ;\, L_{1}+\dots+L_{r} \ge n  \}$, the index of the first excursion of $Y_{x}$ to $\mathbf{0}$ that contains $n$, then there exists a constant $C>0$ such that $M_{n}\le C n$ with probability exponentially close to one. 
		Let $h_{1},h_{2},\cdots$ denote its subsequent excursion heights of $Y_{x}$. Since $h_{1}\le L_{1}=\tau$ and using the elementary inequality 
		\begin{align}\label{eq:elementary_ineq}
			1-(1-a)^{n}\le na \quad \textup{for $a\in (0,1)$}, 
		\end{align}
		for each $s>0$, 
		\begin{align}
			\P\left( \left| \max_{1\le t \le n} X_{t}(\ell) - X_{n}(\ell) \right| \geq s \right) &\le \P\left( \lVert Y_{n}\rVert_{1} \geq s \right) \\
			&\le \P\left( \max(h_{1},\dots,h_{M_{n}}) \ge s \right) \\
			&\le \P\left( \max(h_{1},\dots,h_{n}) \ge s \right)  \\
			& \le 1-(1-\P(h_{1} \ge s))^{n}  \\
			& \le n \, \P(h_{1} \ge s) \\
			& \le n \, \P(\tau \ge s)  
		\end{align}	
		Note that $\P(\tau\ge s)$ is exponentially small in $s$. %and $\P(M_{n}>Cn)$ is exponentially small in $n$. 
		Hence the first inequality in \eqref{eq:X_excursion_bd_1} follows.

		Next, we show the second inequality in \eqref{eq:X_excursion_bd_1}.  By definition of $\hat{X}_{x}$, we have 
		\begin{align}
			\P\left(  \left|  \max_{1\le s\le x} \hat{X}_{s}(\ell) - X_{x}(\ell) \right| \ge s \right) &= \P\left(  \left|  \max_{0\le s\le x} \left( X_{s}(\ell) - \min_{0\le t \le s} X_{t}(\ell) \right) - X_{x}(\ell) \right| \ge s \right)  \\
			&\le \P\left(   \max_{0 \le s\le x}  X_{s}(\ell) - X_{x}(\ell)  \ge s/2 \right)  \\
			&\qquad 	+ \P\left(  \max_{0\le s \le x} \left( -\min_{0\le t \le s} X_{t}(\ell) \right)   \ge s/2 \right) .
		\end{align}
		The second term in the last expression is exponentially small in $s$ due to Proposition \ref{prop:supercritical_excursion_no_return}. Hence the second inequality in \eqref{eq:X_excursion_bd_1} follows from the above and the first equality in  \eqref{eq:X_excursion_bd_1}. 
	\end{proof}

	The following lemma shows half of Theorem \ref{thm:SRBM_weak_convergence} \textbf{(i)}.

	\begin{lemma}\label{lemma:max_W_upper_bd}
		Let $\bmu=(\mu_{1},\dots,\mu_{\kappa}):=\sum_{j=1}^{r} \e_{\alpha_{j}} (p_{\alpha_{j}}- p_{\alpha_{j+1}})$. For $i=1,\dots,\kappa$, almost surely, 
		\begin{align}\label{eq:max_W_bd}
			\limsup_{n\rightarrow \infty} n^{-1} \left( \max_{0\le t \le n} W_{t}(i) \right) \le \mu_{i}. 
		\end{align}
	\end{lemma}

	\begin{proof}
		By Proposition \ref{prop:carrier_comparison_localized}, 
		\begin{align}\label{eq:W_bd_hat_X}
			\max_{0\le t \le x} W_{x}(i) \le  \max_{0\le t \le x}  \hat{X}_{x}(i) \quad \textup{$i=1,\dots,\kappa$},
		\end{align}
		where $\hat{X}_{x}(i)=X_{x}(i)-\min_{0\le s\le x} X_{s}(i)$. Let $\tau_{0}:=0$ and let $\tau_{j}$ for $j\ge 1$ denote the $j$th return time of $X_{x}^{s}$ to the origin, and let $h_{j}$ denote the maximum value of $\lVert X_{s}^{s} \rVert_{1}$ during the  interval $[\tau_{j-1},\tau_{j}]$. By the strong Markov property, $h_{j}$'s are i.i.d.. By Lemma \ref{lemma:localized_carrier_convergence}, $(X_{x}^{s})_{x\ge 0}$ is a Markov chain on $\Z^{\kappa}_{\ge 0}$ with a unique stationary distribution and its return time to the origin, say $\tau$, has finite  moments of all order by Theorem \ref{thm:excursion_length_MGF}. 
		
		Now note that, for each $s>0$, 
		\begin{align}
			\P\left( \max_{0\le t \le n} \lVert X_{t}^{s} \rVert_{1} \ge s  \right) \le 	\P\left( \max(h_{1},\dots,h_{n}) \ge s  \right) &= 1-(1-\P(h_{1} \ge s))^{n} \\
			& \le n \, \P(h_{1} \ge s) \le   n \, \P(\tau \ge s). 
		\end{align}
		Now choosing $s=n^{1/4}$, it follows that $	\P\left( n^{-1/2} \max_{0\le t \le n} \lVert X_{t}^{s} \rVert_{1} \ge n^{-1/4}  \right)$ is summable, so by the Borel-Cantelli lemma, 
		\begin{align}
			\lim_{n\rightarrow\infty}	n^{-1/2} \max_{0\le t \le n} \lVert X_{t}^{s} \rVert_{1}  = 0 \quad \textup{a.s..}
		\end{align}
		Combining with \eqref{eq:W_bd_hat_X} and recalling $\hat{X}_{x}(i)=X_{x}(i)$ for $i\in \mathcal{C}^{\p}_{s}$, we deduce \eqref{eq:max_W_bd} for all $i\in \mathcal{C}^{\p}_{s}$.

		By the argument in the previous paragraph, we may assume the set $\mathcal{C}^{\p}_{u}$ of unstable colors is nonempty and it remains to show the statement for unstable colors. Fix $j \in \{1,\dots,r \}$ and let $\ell=\alpha_{j}$, $\ell^{+}=\alpha_{j+1}$ (with $\alpha_{r+1}=0$). Since $\ell$ is an unstable color, $p_{\ell}\ge p_{\ell^{+}}$. First, suppose $p_{\ell}>p_{\ell^{+}}$. Then Propositions \ref{prop:supercritical_excursion_bd} and \ref{prop:limit_thms_unstable_colors} imply 
		\begin{align}
			\lim_{n\rightarrow\infty} n^{-1} \max_{1\le t\le n} \hat{X}_{t}(\ell)= 	\lim_{n\rightarrow\infty} n^{-1}  X_{n}(\ell) = p_{\ell} - p_{\ell^{+}}
		\end{align}
		almost surely. Then the assertion follows from \eqref{eq:W_bd_hat_X}.
		
		It remains to consider the case  $p_{\ell}=p_{\ell^{+}}$. In this case, we wish to show 
		\begin{align}
			\limsup_{n\rightarrow \infty} n^{-1} \left( \max_{0\le t \le n} W_{t}(\ell) \right) = 0.
		\end{align}
		Rewrite \eqref{eq:W_bd_hat_X} as 
		\begin{align}\label{eq:W_bd_hat_X1}
			\max_{0\le t\le x}	W_{x}(\ell)\le 		\max_{0\le t\le x}\left(  X_{t}(\ell)-\min_{0\le k \le t} X_{k}(\ell) \right)  \le 	\max_{0\le t\le x}X_{t}(\ell) + \max_{0\le t\le x} (-X_{t}(\ell)). 
		\end{align}
		Hence it suffices to show 
		\begin{align}
			\limsup_{n\rightarrow \infty} n^{-1} \left( \max_{0\le t \le n} X_{t}(\ell) \right)  = 			\limsup_{n\rightarrow \infty} n^{-1} \left( \max_{0\le t \le n} -X_{t}(\ell) \right)  = 0.
		\end{align}
		
		First assume $\ell+1=\ell^{+}$. In this case $X_{x}(\ell)$ is a lazy simple random walk on $\Z$.  Hence by the reflection principle, 
		\begin{align}
			\P\left( \max_{0\le t\le n} (-X_{t}(\ell)) \ge a  \right) \le 2 \P\left( -X_{n}(\ell) \ge a \right) \le \exp(-\frac{a^{2}}{n}). 
		\end{align}
		The right-hand side is exponentially small in $a$ by the bounded difference inequality. So taking $a=n^{2/3}$ and applying the Borel-Cantelli lemma show that $n^{-1} \max_{0\le t\le n} (-X_{t}(\ell))$ converges to zero almost surely. By a symmetric argument, the same conclusion holds for $n^{-1} \max_{0\le t\le n} X_{t}(\ell)$. Hence this verifies the assertion. 
		
		Lastly, suppose $\ell+1<\ell^{+}$. In this case, $X_{x}(\ell)$ is not a random walk. 
		Instead, from \eqref{eq:increment_as_functional0}, we can write it as a Markov additive functional:
		\begin{align}
			X_{x}(\ell) = \sum_{t=0}^{x} g^{\ell}(X_{t}^{s},\, \xi_{t+1}). 
		\end{align}
		Moreover, the increment $g^{\ell}(X_{t}^{s},\, \xi_{t+1})$ does not depend on the whole $X_{t}^{s}$, but only on 
		\begin{align}\label{eq:linear_scailing_W_pf2}
			Y_{t}:=\big(X_{t}(\ell+1) ,\cdots, \,X_{t}(\ell^{+}-1) \big).
		\end{align}	
		Let $\tau_{0}=0$ and $\tau_{i}$ for $i\ge 1$ denote the $i$th return time of $Y_{x}$ to the origin. According to Theorem \ref{thm:excursion_length_MGF}, $\tau_{1}$ (and hence all $\tau_{i}$'s) have a finite moments of all orders. 
		
		Consider the process $R_{i}:=-X_{\tau_{i}}(\ell)$. By the strong Markov property, the sequence $R_{i}$ for $i\ge 1$ is a random walk. Denote its increment $\eta_{i}:=R_{i}-R_{i-1}$. Then $\eta_{i}$  has finite moments of all orders since each $\tau_{i}-\tau_{i-1}$ does so and $X_{x}(\ell)$ changes at most by one in $x$. Moreover, by the strong law of large numbers and the Markov chain ergodic theorem, 
		\begin{align}
			\E[\eta_{1}] =\lim_{n\rightarrow\infty} \frac{R_{n}}{n}  = \lim_{n\rightarrow\infty} \frac{\tau_{n}}{n}  \frac{(-X_{\tau_{n}(\ell)})}{\tau_{n}} = \E[\tau_{1}]\,  \E_{\pi^{s} \otimes \p}[ g^{\ell}(X_{x}^{s} , \xi_{x+1}) ] = 0.
		\end{align}
		Hence $R_{i}$ is a mean-zero random walk. 
		
		Denote $M:=\max(\tau_{1}, \tau_{2}-\tau_{1},\dots,\tau_{n}-\tau_{n-1})$. Since this is the maximum of i.i.d. random variables of finite moments of all orders, union bound and Chebyshev's inequality and \eqref{eq:elementary_ineq} give 
		\begin{align}
			\P(M\ge a)  = 1 - (1-\P(\tau_{1}\ge a))^{n} \le n \, \P(\tau_{1}\ge a) = O(n \, a^{-d})
		\end{align}
		for any integer $d\ge 1$. 
		Also, since the increments $X_{s+1}(\ell)-X_{s}(\ell)$ are bounded by 1, 
		\begin{align}
			\max_{0\le s\le n} R_{s}  &\ge 	\max_{0\le s\le \tau_{n}} -X_{s}(\ell)  -M \ge  \max_{0\le s\le n} -X_{s}(\ell)  - M.
		\end{align}
		Hence combining the above inequalities and using Kolmogorov's maximal inequality, for any $b>1/\sqrt{n}$, 
		\begin{align}
			\P\left( n^{-1} \max_{0\le s\le n} -X_{s}(\ell) \ge b  \right) &\le  	\P\left( \max_{0\le s\le n} R_{s}(\ell) \ge nb-M  \right) \\
			&\le \P\left( \max_{0\le s\le n} R_{s}(\ell) \ge n b-\sqrt{n} \right) + \P\left( M > \sqrt{n}  \right)  \\
			&=  \frac{n\Var(\eta_{1})}{ (n b-\sqrt{n})^{2}}  + O(n^{-2}). 
		\end{align}
		Then taking $b=n^{-1/6}$ and denoting $T_{n}:=\max_{0\le s\le n} -X_{s}(\ell)$, we get 
		\begin{align}\label{eq:T_n_tail_prob}
			\P\left( n^{-1} T_{n} \ge n^{-1/3} \right)  \le \frac{c}{n^{2/3}}
		\end{align}
		for some constant $c>0$.   Notice that $T_{n}$ is non-decreasing in $n$. By Borel-Cantelli Lemma and \eqref{eq:T_n_tail_prob}, we have that $n^{-2}T_{n^{2}}\rightarrow 0$ almost surely. Fix $k\ge 1$ and let $n=n(k)$ denote the largest inetger such that $n^{2}\le k < (n+1)^{2}$. Then using monotonicity, 
		\begin{align}
			\frac{n^{2}}{(n+1)^{2}}	\frac{T_{n^{2}}}{n^{2}} \le \frac{T_{k}}{k} \le \frac{T_{(n+1)^{2}}}{(n+1)^{2}} \frac{(n+1)^{2}}{n^{2}}. 
		\end{align}
		Taking $k\rightarrow\infty$, we deduce that $k^{-1}T_{k}\rightarrow 0$ almost surely as $k\rightarrow\infty$. Therefore, it follows that $n^{-1} \max_{0\le t\le n} (-X_{t}(\ell))$ converges to zero almost surely. By a symmetric argument, the same conclusion holds for $n^{-1} \max_{0\le t\le n} X_{t}(\ell)$. This completes the proof. 
	\end{proof}

	Now we are ready to prove Theorem \ref{thm:SRBM_weak_convergence} \textbf{(i)}. 
	
	\begin{proof}[\textbf{Proof of Theorem \ref{thm:SRBM_weak_convergence} \textup{(i)}}]
		
		We wish to show that 
		\begin{align}\label{eq:W_lim_claim1}
			\lim_{n\rightarrow\infty} n^{-1}W_{n} = \bmu \quad \textup{a.s..}
		\end{align}
		Note that by Lemma \ref{lemma:max_W_upper_bd}, 
		\begin{align}\label{eq:W_lim_claim2}
			\limsup_{n\rightarrow\infty} n^{-1}W_{n} \le \bmu \quad \textup{a.s.,}
		\end{align}
		where we interpret the inequality componentwise. Recall the Skorokhod decomposition $W_{x} = X_{x} + RY_{n}$ in Lemma \ref{lemma:Skorokhod}. We first consider the case when $\kappa\ge 3$. Then writing $R=I-Q$ with $Q=\textup{tridiag}_{\kappa}(0,0,1)$ and using the identity $(I-Q)(I+Q+Q^{2}+\dots)=I$, we see that $R^{-1}$ is the following upper diagonal matrix whose nonzero entries equal to one:
		\begin{align}
			R^{-1} = I + Q + \dots + Q^{\kappa-1}.  
		\end{align}
		Write 
		\begin{align}
			n^{-1}Y_{n} = R^{-1}(n^{-1}W_{n} - n^{-1}X_{n}). 
		\end{align}
		Then by using \eqref{eq:W_lim_claim2} and the fact that $\lim_{n\rightarrow\infty} n^{-1}X_{n}=\bmu$ a.s. (see Prop. \ref{prop:limit_thms_unstable_colors}), 
		\begin{align}\label{eq:W_lim_claim3}
			\mathbf{a}:= \limsup_{n\rightarrow\infty} (n^{-1}W_{n} - n^{-1}X_{n}) \le \mathbf{0}, 
		\end{align}
		where we applied limsup as well as inequality componentwise. It is crucial to note that $R^{-1}$ has nonnegative entries. Hence 
		\begin{align}
			\limsup_{n\rightarrow\infty} n^{-1}Y_{n} = R^{-1} \mathbf{a} \le \mathbf{0}. 
		\end{align}
		But since each $Y_{n}$ is a nonnegative vector by definition, it follows that $\lim_{n\rightarrow\infty} n^{-1} Y_{n} =\mathbf{0}$ almost surely. Then using the Skorokhod decomposition once more, we get 
		\begin{align} 
			\lim_{n\rightarrow\infty} n^{-1} W_{n} = \bmu + R \lim_{n\rightarrow\infty} n^{-1}Y_{n} = \bmu
		\end{align}
		almost surely, as desired. 
		
		It remains to verify \eqref{eq:W_lim_claim1} for the case when $\kappa=1,2$. Denote $\mathbf{y}:=\limsup_{n\rightarrow\infty} n^{-1}Y_{n}$. Suppose $\kappa=2$. Then the Skorokhod decomposition and \eqref{eq:W_lim_claim3} yield
		\begin{align}
			\begin{bmatrix}
				1 & -1 \\ 
				0 & 1
			\end{bmatrix}
			\mathbf{y} = \mathbf{a} \le \mathbf{0}. 
		\end{align}
		Note that $\mathbf{y}\ge \mathbf{0}$ since $Y_{n}\ge \mathbf{0}$ for all $n\ge 1$. Then it is easy to see that $\mathbf{y}$ must equal $\mathbf{0}$. The case for $\kappa=1$ can be argued similarly. 
	\end{proof}

	\vspace{0.2cm}
	\section{The diffusive scaling limit of the  carrier process}
	\label{sec:carrier_Skorokhod} 
	
	In this section, we prove Theorem \ref{thm:SRBM_weak_convergence} \textbf{(ii)} on the diffusive scaling limit of the carrier process in the critical and the supercritical regime. 	The definition of SRBM below is adapted from \cite[Def. 3.1]{williams1998invariance}.

	\begin{definition}[Semimartingale reflecting Brownian motion]\label{def:SRBM}
		Fix an integer $\kappa\ge 1$ and a subset $J\subseteq \{1,\dots, \kappa\}$. Let $S:=\{(x_{1},\dots,x_{\kappa})\in \R^{\kappa}\,:\, x_{i}\ge 0  \,\, \textup{for all $i\in J$} \}$ and let $\mathcal{B}$ denotes the Borel $\sigma$-algebra on $S$, $\nu$ is a probability measure on $(S,\mathcal{B})$, $\theta$ is a constant vector in $\R^{\kappa}$, $\Sigma$ is a $\kappa\times \kappa$  covariance matrix (symmetric and positive semidefinite\footnote{We allow the covariance matrix to be degenerate.}), and $R$ is a $\kappa\times \kappa$ matrix. A \textit{semimartingale reflecting Brownian motion} (SRBM) associated with the data $(S,\theta,\Sigma, R,\nu)$ is an $\{ \mathcal{F}_{t} \}$-adapted, $\kappa$-dimensional process $W$ defined on some probability space $(\Omega, \mathcal{F}, \P)$ and filtration $\{\mathcal{F}_{t}; t\ge 0\}$ (an increasing family of sub-$\sigma$-algebras of $\mathcal{F}$) such that 
		\begin{description}
			\item[(i)] $W = X + RY$, $\P$-a.s.;
			\item[(ii)] $\P$-a.s., $W$ has continuous paths and $W(t)\in S$ for all $t\ge 0$;
			\item[(iii)] Under $\P$, 
			\begin{description}[labelindent=-0.7cm]
				\item[(a)] $X$ is a $\kappa$-dimensional Brownian motion with drift vector $\theta$, covariance matrix $\Sigma$  and $X(0)\sim \nu$;
				\item[(b)] $\{  X(t) - X(0) - \theta t , \mathcal{F}_{t}; t\ge 0 \}$ is a martingale; 
			\end{description}
			\item[(iv)] $Y$ is an $\{ \mathcal{F}_{t} \}$-adapted, $\kappa$-dimensional process such that $\P$-a.s. for $i=1,\dots, \kappa$, 
			\begin{description}[labelindent=-0.7cm]
				\item[(a)] $Y_{i}(0)=0$; 
				\item[(b)] $Y_{i}$ is continuous and non-decreasing; 
				\item[(a)] $Y_{i}$ can increase only when $W$ is on the face $F_{i}:=\{ x\in S\,:\, x_{i}=0 \}$, i.e., $\int_{0}^{\infty} \mathbf{1}(W_{i}(s)> 0)\, dY_{i}(s) = 0$.  
			\end{description}
		\end{description}
	\end{definition}
	
	Roughly speaking, an SRBM $W=X+RY$ behaves like the Brownian motion $X$ in the interior of the domain $S$ and it is confined to the domain by the instantaneous ``reflection'' (or ``pushing'') at the boundary, where the direction of such ``reflection'' on the $i$th face $F_{i}$ is given by the $i$th column of the \textit{reflection matrix} $R$. Note that in Def. \ref{def:SRBM}, the domain $S$ only requires coordinates in the set $J$ be nonnegative, while it is standard to take $S$ to be the nonnegative orthant $\R^{\kappa}_{\ge 0}$. We take this slightly more general domain to analyze the diffusive scaling limit of the centered carrier process $\overline{W}_{t}$ in Theorem \ref{thm:SRBM_weak_convergence}, which can take negative values in coordinates corresponding to unstable colors.

	A classical result of Reiman and Williams \cite{reiman1988boundary} (see also \cite[Thm. 3.1]{williams1998invariance}) shows that an SRBM associated with $(S,\theta, \Sigma, R, \nu)$ with $S=\R^{\kappa}_{\ge 0}$ and $\Sigma$ non-degenerate uniquely exists if and only if the reflection matrix $R$ is \textit{completely-$\mathcal{S}$} (see Def. \ref{def:completely_S}). Roughly speaking, this condition means that at any boundary point of $S$, there exists a nonnegative linear combination of the reflection directions (i.e., columns of $R$) that points to the interior of $S$. 
	When $\Sigma$ is degenerate, then SRBM still exists but may not be unique. 
	\begin{definition}[Completely-$\mathcal{S}$]\label{def:completely_S}
		A matrix $R\in \R^{d\times d}$ is \textit{completely-$\mathcal{S}$} if for every principal submatrix $R_{0}$ of $R$, there is a nonnegative vector $x_{0}$ such that $R_{0}x_{0}$ has strictly positive coordinates. Here a principal submatrix of $R$ is a matrix obtained by deleting all rows and columns of $R$ with indices in some proper subset set $\mathcal{I}\subsetneq \{1,\dots,d\}$ (possibly empty).  
	\end{definition}

	It is critical to notice that the reflection matrix $R$ in \eqref{eq:def_reflection_mx} that gives a Skorokhod decomposition of the carrier process $W_{x}$  as in Lemma \ref{lemma:Skorokhod} has the following property: For $\kappa\ge 3$, $R=I-Q$ where $Q$ has a spectral radius less than one. In this case, we can say a lot about SRBM with a more direct argument. The first step is to recall that the problem that defines SRBM in Def. \ref{def:SRBM} is a particular instance of the classical  Skorokhod problem stated below.

	\begin{definition}[Skorokod Problem]\label{def:SP}
		Fix a subset $J\subseteq \{1,\dots, \kappa\}$ and let $S:=\{(x_{1},\dots,x_{\kappa})\in \R^{\kappa}\,:\, x_{i}\ge 0  \,\, \textup{for all $i\in J$} \}$.  Let $C_{S}$ denote the subspace of $C^{\kappa}(0,\infty)$ consisting of paths $x$ with $x(0)\in S$.	Fix matrix $R\in \R^{\kappa\times \kappa}$ and $x\in C_{S}$. A pair $(z,y)\in C^{\kappa}(0,\infty)\times C^{\kappa}(0,\infty)$ is a solution of the Skorohod problem for $x$ w.r.t. $R$ if the following conditions hold: 
		\begin{description}[itemsep=0.1cm]
			\item[(i)] $z(t) = x(t) + R\, y(t)$ for all $t\ge 0$.
			
			\item[(ii)] $z(t)\in S$ for all $t\ge 0$.
			
			\item[(iii)] For $i=1,\dots, \kappa$, $y_{i}(0)=0$, $y_{i}(t)$ is non-decreasing, and $\int_{0}^{\infty} \mathbf{1}( i \in J) \mathbf{1}(z_{i}(t) \ge 0 ) \, dy_{i}(t) = 0$ .
		\end{description}
	\end{definition}
	
	When the reflection matrix $R$ can be written as $R=I-Q$ where $Q$ is nonnegative and has a spectral radius less than one, then there is a unique solution $(z,y)$ to the Skorokhod problem for each path $x$ and the map $x\mapsto (z,y)$ (the Skorohod map) is continuous. This result is stated and proved in Theorem \ref{thm:SK_strong_sol}.

	\begin{theorem}[Harrison and Reiman '81]\label{thm:SK_strong_sol}
		Let $S=\R^{d}\times \R^{\kappa-d}_{\ge 0}$ and $C_{S}$ be as in Def. \ref{def:SP}. Suppose the reflection matrix $R$ can be written as $R=I-Q$ where $Q$ is nonnegative and has a spectral radius of less than one. Then for each path $x\in C_{S}$, there exists a unique pair of functions $(z,y) \in C^{\kappa}(0,\infty)\times C^{\kappa}(0,\infty)$ that solves the Skorokhod problem in Def. \ref{def:SP}. Furthermore, denoting $z=\phi(x)$ and $y=\psi(x)$, both $\phi$ and $\psi$ are continuous mappings $C_{S} \rightarrow C^{\kappa}(0,\infty)$. 
	\end{theorem}

	\begin{proof}
		The original result \cite[Thm. 1]{harrison1981reflected} is stated for $S=\R^{\kappa}_{\ge 0}$, where in our setting we allow $S$ to be the intersection of axes-parallel half-spaces in $\R^{\kappa}$. A minor modification of the proof of \cite[Thm. 1]{harrison1981reflected} will show the minor extension as stated above. We sketch the argument for completeness. 
		
		Without loss of generality, assume $S=\R^{d}\times \R^{\kappa-d}_{\ge 0}$ for some $d\in \{0,\dots,\kappa\}$. Denote $C=C^{\kappa}(0,\infty)$ and fix $x\in C_{S}$. Let $C_{0}$ be the set of paths $y\in C$ such that $y(0)=\mathbf{0}$ and  and non-decreasing componentwise. Define a map $\pi=\pi_{x}:C_{0}\rightarrow C_{0}$ such that 
		\begin{align}
			\pi(y)_{i}(t) = \begin{cases} 
				0 & \textup{if $i=1,\dots,d$} \\
				\sup_{0\le s \le t} \left[ y(s) Q - x(s) \right]^{+} & \textup{if $i=d+1,\dots,\kappa$}.
			\end{cases}
		\end{align}
		Then one can check that $(z,y)$ is a solution to the Skorokhod problem if and only if 
		\begin{align}\label{eq:SK_continuity_pf1}
			y\in C_{0}, \qquad y=\pi(y),\qquad z = x + (I-Q)y.
		\end{align}
		One can then argue that there is a unique solution $y\in C_{0}$ such that $y=\pi(y)$.

		%To show this, we first argue that we can assume without loss of generality that 

		To this end, for each square matrix $Q$, we let $\lVert Q \rVert_{\infty}$  denote its maximum absolute row sum. Since $Q$ is nonnegative and has spectral radius $<1$, there exists a positive diagonal matrix $\Lambda$ such that $\tilde{Q}:=\Lambda^{-1}Q\Lambda$ satisfies $\lVert \tilde{Q} \rVert_{\infty}<1$ \cite[Lem. 3]{veinott1969discrete}. Observe that $(z,y)$ satisfies \eqref{eq:SK_continuity_pf1} if and only if $(z \Lambda, y \Lambda )$ satisfies \eqref{eq:SK_continuity_pf1} with $x$ and $Q$ replaced by $\Lambda x$ and $\tilde{Q}$. Thus, without loss of generality, we may assume $\lVert Q \rVert_{\infty}<1$.

		Now fix $T\ge 0$ and define $C_{0}[0,T]$ and $C_{S}[0,T]$ in the obvious way. These are complete metric spaces endowed with the norm 
		\begin{align}
			\lVert y \rVert:= \max_{1\le j \le \kappa} \sup_{0\le t \le T}  |y_{j}(t)|. 
		\end{align}
		Then one can show that the map $\pi$ is a contraction on $C_{0}[0,T]$: 
		\begin{align}
			\lVert \pi(y) - \pi(y') \rVert \le \lVert Q \rVert_{\infty} \,\, \lVert y-y' \rVert.
		\end{align}
		%where $\lVert Q \rVert$ denotes the spectral norm of $Q$  we are assuming that it is in $[0,1)$. 
		Since $\lVert Q \rVert_{\infty}<1$, it follows that $\pi$ is a contraction mapping, implying that there is a unique fixed point $y\in C_{0}$. 
		
		Now to show the continuity of the mapping $x\mapsto \phi(x)$, we observe that $\phi(x)$, being the unique fixed point of $y=\pi_{x}(y)$ of the contraction mapping $\pi$, can be explicitly constructed as the limit of $y^{n}(x):=\pi_{x}^{n}(y^{0})$ with $y^{0}\equiv 0$. %Since $I-Q$ is invertible, 
		Then note that for $x,x'\in C_{0}[0,T]$, 
		\begin{align}
			\lVert y^{n+1}(x) - y^{n+1}(x') \rVert \le 		\lVert  x - x'  \rVert + 	\lVert Q \rVert_{\infty} \, 	\lVert  y^{n}(x) - y^{n}(x')  \rVert. 
		\end{align}
		By an induction and taking $n\rightarrow \infty$, we get $	\lVert  \phi(x) - \phi(x')  \rVert \le 	\frac{1}{1-\lVert Q \rVert_{\infty}}	\lVert  x - x'  \rVert$. Thus $\phi$ is $	\frac{1}{1-\lVert Q \rVert_{\infty}}$-Lipschitz continuous on $C_{0}[0,T]$. Since $T$ was arbitrary, this implies continuity of $\phi$ on $C_{0}(0,\infty)$. Thus $\phi$ is continuous on $C^{\kappa}(0,\infty)$ in the topology of uniform convergence on compact intervals. The continuity of the mapping $x\mapsto \psi(x)$ is clear from the last identity in \eqref{eq:SK_continuity_pf1}. 
	\end{proof}

	In the proof of Theorem \ref{thm:SK_strong_sol}, we have used the fact that of $Q$ is a matrix of spectral radius less than one, then there exists a positive diagonal matrix $\Lambda$ such that $\Lambda^{-1}Q\Lambda$ has maximum absolute row sum strictly less than one, appealing to \cite[Lem. 3]{veinott1969discrete}. In our case, $Q=\textup{tridiag}_{\kappa}(0,0,1)$ and we can directly take $\Lambda$ to have diagonal entries $\Lambda(i,i)=\kappa-i+1$ for $i=1,\dots,\kappa$, in which case the maximum absolute row sum equals $\frac{\kappa-1}{\kappa}<1$.

	\begin{proof}[\textbf{Proof of Theorem} \ref{thm:SRBM_weak_convergence} \textup{\textbf{(ii)}}]
		For this proof, we will appeal to the continuity of the Skorokod map $x\mapsto (y,z)$ we established in Theorem \ref{thm:SK_strong_sol}. 
		Let $J=\{1,\dots,\kappa\}\setminus \{ \alpha_{j}\,;\, j=1,\dots, r,\, p_{\alpha_{j}}>p_{\alpha_{j+1}} \}$ and $S:=\{ (x_{1},\dots,x_{\kappa})\in \R^{\kappa} \,:\, x_{i}\ge 0 \,\, \textup{for all $i\in J$}  \}$. Let $\bmu=(\mu_{1},\dots,\mu_{\kappa}):=\sum_{j=1}^{r} \e_{\alpha_{j}} (p_{\alpha_{j}}- p_{\alpha_{j+1}})$. Then $\bmu$ is nonzero in its $j$th coordinate if and only if $j\in J$. Recall the Skorokhod decomposition of the carrier process $W_{x}$ in Lemma \ref{lemma:Skorokhod}:
		\begin{align}\label{eq:Skorokhod_decomp_centered_pf0}
			\overline{W}_{x}= \overline{X}_{x} + RY_{x} \quad \textup{for $x\in N$}, 
		\end{align}
		where we denoted $\overline{W}_{x}=W_{x}-x\bmu$ and $\overline{X}_{x}=X_{x}-x\bmu$. Since $W_{x}\in \R^{\kappa}_{\ge 0}$, we have $\overline{W}_{s}\in S$ for all $s\in \R_{\ge 0}$. Note that \eqref{eq:Skorokhod_decomp_centered_pf0} gives a Skorokhod decomposition of the centered carrier process $(\overline{W}_{x})_{x\in \N}$. Namely, for each $i\in J$, $Y_{x}(i)$ can increase only if $\overline{W}_{x}(i)=0$. This is because for $i\in J$, $\overline{W}_{x}(i)=W_{x}(i)$, and by Lemma \ref{lemma:Skorokhod}, we know that $Y_{x}(i)$ increases only if $W_{x}(i)=0$. From \eqref{eq:Skorokhod_decomp_centered_pf0}, we deduce 
		\begin{align}\label{eq:Skorokhod_decomp_centered_pf1}
			\widetilde{W}^{n}(t)=\widetilde{X}^{n}(t) + R\widetilde{Y}^{n}(t) \quad \textup{for $t\in \R_{\ge 0}$}, 
		\end{align}
		where $\widetilde{W}^{n}$, $\widetilde{X}^{n}$, and $\widetilde{Y}^{n}$ are the linear interpolations of $\frac{1}{\sqrt{n}}(W_{x}-x\bmu)$,  $\frac{1}{\sqrt{n}}(X_{x}-x\bmu)$, and  $\frac{1}{\sqrt{n}} Y_{x}$.

		Since $R=\textup{tridiag}_{\kappa}(0,1,-1)$, we can write $R=I-Q$ where $Q=\textup{tridiag}_{\kappa}(0,0,1)$, so $Q$ has spectral radius zero for all $\kappa\ge 1$ since $Q^{\kappa}$ is zero. Denoting $x\mapsto (\phi(x), \psi(x))$ by the Skorohod mapping as in Theorem \ref{thm:SK_strong_sol}, according to \eqref{eq:Skorokhod_decomp_centered_pf1}, for each $n\ge 1$, we have 
		\begin{align}
			&\phi( \widetilde{X}^{n} ) =  \widetilde{W}^{n} \quad \textup{and} \quad \psi( \overline{X}^{n}) = \widetilde{Y}^{n} .
		\end{align}
		That is, the pair $(\widetilde{W}^{n}, \widetilde{Y}^{n})$ is the unique solution of the Skorokhod problem for $\widetilde{X}^{n}$ with respect to the reflection matrix $R$. Recall that by Proposition \ref{prop:limit_thms_unstable_colors},  $\widetilde{X}^{n}$ converges weakly to the Brownian motion $B$ in $\R^{\kappa}$ with zero drift and covariance matrix $\Sigma$.  By continuity of the Skorohod mapping, it follows that 
		\begin{align}
			\widetilde{W}^{n} \Longrightarrow \lim_{n\rightarrow\infty} \phi(\widetilde{X}^{n}) = \phi(B),  \\
			\widetilde{Y}^{n} \Longrightarrow \lim_{n\rightarrow\infty} \psi(\widetilde{X}^{n}) = \psi(B).
		\end{align}
		In particular, $\widetilde{W}^{n}$ converges weakly to the SRBM associated with data $(S, \mathbf{0}, \Sigma, R, \delta_{\mathbf{0}})$, as desired. 
	\end{proof}

	\section{Proofs of Theorems \ref{thm:iid_supercritical} and \ref{thm:iid_supercritical}}
	\label{section:supercri}
	
	In this section, we establish scaling limits of the top soliton lengths for the i.i.d. model in the critical and the supercritical regimes.  
	
	By now, it is easy to deduce Theorem \ref{thm:iid_critical}.
	
	\begin{proof}[\textbf{Proof of Theorem \ref{thm:iid_critical}}]
		
		Suppose $p_{0}=\max(p_{1},\cdots, p_{\kappa})$. Then $\mathcal{C}_{u}^{\mathbf{p}}=\{0\le i \le \kappa\,\colon \, p_{i}=p_{0} \}$ and we may write $\mathcal{C}_{u}^{\mathbf{p}}=\{ \alpha_{0},\cdots,\alpha_{r} \}$ with $0=\alpha_{0}<\alpha_{1}<\cdots<\alpha_{r}$. Then the weak convergence of the diffusively scaled first soliton length in \eqref{eq:thm_iid_cri_lambda1} follows from  Lemma \ref{lemma:queue_formula_soliton}, Theorem \ref{thm:SRBM_weak_convergence}, and the continuous mapping theorem. 
		
		Next, we justify that $\lambda_{j}(n) = \Theta(n)$ with high probability for all $j\ge 1$. The upper bound follows since $\lambda_{j}(n)\le \lambda_{1}(n)=O(\sqrt{n})$ with high probability. For the lower bound, we use the fact that the carrier process in the critical regime converges weakly to an SRBM as in Theorem \ref{thm:SRBM_weak_convergence}. In particular, there are  excursions of the carrier process of height (i.e., the $L_{1}$-norm) at least $c\sqrt{n}$ with high probability if $c>0$ is small enough. Then the lower bound $\lambda_{1}(n)=\Omega(\sqrt{n})$ with high probability  follows from Lemma \ref{lemma:soliton_lengths_excursions}. 
	\end{proof}

	In the rest of this section, we prove  Theorem \ref{thm:iid_supercritical}. Throughout we will assume  $p^{*}=\max(p_{1},\cdots,p_{\kappa})>p_{0}$. Let $\alpha_{1}<\cdots<\alpha_{r}$ denote the unstable colors. Under the hypothesis it holds that $p_{\alpha_{1}}=p^{*}$.

	\begin{proof}[\textbf{Proof of Theorem \ref{thm:iid_supercritical} \textup{(i)}}]
		Let $\bmu=(\mu_{1},\dots,\mu_{\kappa})$ be as in Theorem \ref{thm:SRBM_weak_convergence}. By Lemma \ref{lemma:queue_formula_soliton} and Theorem \ref{thm:SRBM_weak_convergence}, almost surely, 
		\begin{align}
			\lim_{n\rightarrow\infty} 	n^{-1}\lambda_{1}(n)  &= \lim_{n\rightarrow\infty} n^{-1} \lVert W_{n} \rVert_{1} = \lVert \bmu \rVert_{1} \\
			&= (p_{\alpha_{1}} - p_{\alpha_{2}}) + (p_{\alpha_{2}} - p_{\alpha_{3}}) + \dots+ (p_{\alpha_{r}} - p_{0})  = p^{*}-p_{0}.
		\end{align}
		
		Next, recall the Skorokhod decomposition $W_{x}=X_{x} + RY_{x}$ in Lemma \ref{lemma:Skorokhod}. Define $t(n):=\argmax_{0\le t \le n} \lVert W_{t} \rVert_{1}$. Let $J$ denote the set of indices $i\in \{1,\dots,\kappa\}$ such that $\mu_{i}>0$. Then  $\mu_{i}=0$ if $i\notin J$, so 
		\begin{align}
			\lambda_{1}(n) - n \lVert \bmu \rVert_{1} &= \sum_{i=1}^{\kappa} W_{t(n)}(i)  - n \mu_{i} \\
			&= \sum_{i\in J} W_{t(n)}(i) - n\mu_{i} + \max_{0\le t \le n} \sum_{i\notin J}  W_{t}(i).
		\end{align}
		By Proposition \ref{prop:carrier_comparison_localized}, it follows that 
		\begin{align}\label{eq:lambda_1_fluctuation_pf1}
			\sum_{i\in J} W_{n}(i) - n\mu_{i} + \sum_{i\notin J}  W_{n}(i) \le 	\lambda_{1}(n) - n\bmu  \le \sum_{i\in J} \hat{X}_{t(n)}(i) - n\mu_{i} + \max_{0\le t \le n} \sum_{i\notin J}  W_{t}(i).
		\end{align}
		Recall that the linear interpolation of $n^{-1/2}(W_{n} - n\bmu)$ converges weakly to the SRBM with specified data as in Theorem \ref{thm:SRBM_weak_convergence}. Hence the lower bound in Theorem \ref{thm:iid_supercritical} \textbf{(i)} follows from above. For the upper bound, we use Proposition \ref{prop:supercritical_excursion_bd} to note that, almost surely, 
		\begin{align}
			\lim_{n\rightarrow\infty}		n^{-1/2} \left| 	\sum_{i\in J} \hat{X}_{t(n)}(i) - \sum_{i\in J} X_{t(n)}(i)   \right|   = 0.
		\end{align}
		Hence, almost surely, 
		\begin{align}
			& \limsup_{n\rightarrow\infty} n^{-1/2}	\sum_{i\in J}( \hat{X}_{t(n)}(i) - n\mu_{i}) +n^{-1/2} \max_{0\le t \le n} \sum_{i\notin J}  W_{t}(i) \\
			&\qquad =   \limsup_{n\rightarrow\infty} n^{-1/2}	\sum_{i\in J}( X_{n}(i) - n\mu_{i}) +n^{-1/2}\max_{0\le t \le n} \sum_{i\notin J}  W_{t}(i).
		\end{align}
		Recall that by Proposition \ref{prop:limit_thms_unstable_colors}, the linear interpolation of $n^{-1/2}(X_{n}-n\bmu)$ converges to a Brownian motion on $\R^{\kappa}$ with mean zero and an explicit covariance matrix $\Sigma$. 
		Also, by Theorem \ref{thm:SRBM_weak_convergence} and the continuous mapping theorem, 
		\begin{align}
			n^{-1/2}\max_{0\le t \le n} \sum_{i\notin J}  W_{t}(i) \Longrightarrow \sup_{0\le v \le 1} \sum_{i\notin J} W^{v}(v), 
		\end{align}
		where $W=(W^{1},\dots,W^{\kappa})$ is the SRBM in Theorem \ref{thm:SRBM_weak_convergence}. Thus the upper bound in \eqref{eq:lambda_1_fluctuation_thm} follows by the continuous mapping theorem. 
	\end{proof}

	Next, we complete the proof of Theorem \ref{thm:iid_supercritical} \textbf{(ii)}-\textbf{(iii)}. To this effect, it suffices to show the following statement. 
	
	\begin{theorem}\label{thm:subsequent_solitions_supercritical}
		Suppose $p^{*}>p_{0}$ and fix $j\ge 2$. Then the following hold. 
		\begin{description}
			\item[(i)] Suppose $p_{i}=p^{*}$ for a unique $1\le i \le \kappa$. Then $\lambda_{j}(n)=\Theta(\log n)$ with high probability. 
			
			\vspace{0.1cm}
			\item[(ii)] Suppose $p_{i}=p^{*}$ at least two distinct colors $1\le i \le \kappa$. Then $\lambda_{j}(n)=\Theta(\sqrt{n})$ with high probability. 
		\end{description} 
	\end{theorem}
	
	We begin with the following definition. For $0\le i,j\le \kappa$ and a finite subset $H\subseteq \mathbb{N}$, define a random variable $D_{i,j}(H)$ by  
	\begin{align}
		D_{i,j}(H) = \sum_{x\in H} \left[  \mathbf{1}(\xi^{\mathbf{p}}(x)=i) - \mathbf{1}(\xi^{\mathbf{p}}(x)=j) \right],
	\end{align}	 
	which equals the difference of the number of color $i$ and color $j$ balls in $H$ given by $\xi^{\mathbf{p}}$.
	
	\begin{prop}\label{prop:interval_hoeffding}
		Fix $1\le i,j\le \kappa$ and suppose $p_{i}>p_{j}$. Fix a finite subset $H\subseteq \mathbb{N}$. Then for any constant $C>0$, 
		\begin{align}
			\P\left(  D_{j,i}(H)  \ge 2C \log n \right) \le  \exp(  - C (p_{i}-p_{j}) \log n)
		\end{align}
		for all $n\ge 1$. 
	\end{prop}
	
	\begin{proof}
		Let $\eps=p_{i}-p_{j}>0$ and denote $|H|=m$. Note that $\E[D_{j,i}(H)] = - \eps m$. Since $D_{j,i}(H)$ is a sum of i.i.d.  increments with absolute value at most one, by Hoeffding's inequality, 
		\begin{align}
			\P( D_{j,i}(H)  - \E[D_{j,i}(H)] \ge t ) \le e^{ -t^2/ (2 m)}
		\end{align}
		for any $t > 0$. Let $t = \eps m + 2C \log n$. Then $t / m  \ge \eps$, so 
		\begin{align}
			\P( D_{j,i}(H)  \ge 2C\log n ) = \P( D_{j,i}(H)  - \E[D_{j,i}(H)] \ge t ) \le e^{ - (\eps /2) t }  \le e^{ - \eps C \log n }.
		\end{align}
		This shows the assertion. 
	\end{proof}

	\begin{proof}[\textbf{Proof of Theorem \ref{thm:subsequent_solitions_supercritical}}]
		Denote $\xi:=\xi^{n,\p}$. 	Our argument is based on  Lemma \ref{lemma:GK_invariants}. In this proof, for integers $a<b$,  an `interval' $[a,b]$ will refer to the set $\{a,a+1,\dots,b\}$. We say a subset $A\subseteq \mathbb{N}$ is a non-increasing subsequence if $\xi$ is non-increasing on $A$. The `support' of $A$ is the interval of integers  $[\min(A),\max(A)]$.  
		
		We first show the upper bounds in \textbf{(i)} and \textbf{(ii)}. It suffices to obtain bounds on $\lambda_{2}(n)$ in the corresponding regimes. Recall the formula for $\lambda_{1}(n)+\lambda_{2}(n)$ given by Lemma \ref{lemma:GK_invariants}:
		\begin{align}\label{eq:lambda1+lambda2_GK}
			\lambda_{1}(n) + \lambda_{2}(n) &=  \max_{A_{1} \prec A_{2} \subseteq [1,n] } \,\L(A_{1},\xi) + \L(A_{2},\xi). 
		\end{align}
		Let $A_{1}\prec A_{2}$ be an optimal choice of subsequences that achieves $\lambda_{1}(n)+\lambda_{2}(n)$ according to \eqref{eq:lambda1+lambda2_GK}. Let $I=[a,b]$ and $J=[c,d]$ denote the supporting intervals of $A_{1}$ and $A_{2}$, respectively.  We split $A_{1}$ into successive disjoint sub-subsequences $A_{\kappa}', A_{\kappa-1}', \cdots , A_1'$ where in each $A_\ell'$ we only pick the balls of color $\ell$ in $A_{1}$. Let $I_{j}:=[\min A_{j}', \max A_{j}']$. This gives a non-interlacing partition of $I=I_{\kappa}\sqcup\dots \sqcup I_{1}$.  We split $A_{2}$ similarly and obtain a non-interlacing partition $J=J_{\kappa}\sqcup \dots \sqcup J_{1}$ similarly. This gives us a partition of the whole interval $[1, n]$ into the following collection of disjoint sub-intervals 
		\begin{align}\label{eq:pf_subsequent_soliton_length_partition}
			\mathcal{H}=\{ 	[1, a-1], I_{\kappa} , I_{\kappa-1}, \cdots , I_1, [b+1, c-1], J_{\kappa} , J_{\kappa-1}, \cdots , J_1, [d+1, n] \},
		\end{align}
		ordered from left to right.

		For $\lambda_1(n)$, we choose a sub-optimal non-increasing subsequence $A^{(i)}$ by choosing all balls of color $i$ in $[1, n]$. Then $\lambda_{1}(n) \ge \L(A^{(i)}, \xi)$ by  Lemma \ref{lemma:GK_invariants}, so \eqref{eq:lambda1+lambda2_GK} yields
		\begin{align}\label{eq:lambda1+lambda2_GK1}
			\lambda_{2}(n) \le \L(A_{1},\xi) + \L(A_{2},\xi) - \L(A^{(i)},\xi ).
		\end{align}
		Then breaking the right-hand side of \eqref{eq:lambda1+lambda2_GK1} into sub-intervals given by the partition in \eqref{eq:pf_subsequent_soliton_length_partition}, we may write 
		\begin{align}\label{eq:lambda1+lambda2_GK2}
			\L(A_{1},\xi) + \L(A_{2},\xi) - \L(A^{(i)},\xi ) = \sum_{H\in \mathcal{H}} f(H),
		\end{align}
		where if $H = I_j$ or $J_j$  ($1 \le j \le k$), 
		\begin{align}
			f(H)  &:=  (\text{number of balls of color $j$  in $H$}  -  \text{number of balls of color 0  in $H$)} \\
			&\qquad -  (\text{number of balls of color $i$  in $H$}  - \text{ number of balls of color 0  in $H$}) \\
			&=   D_{j,i}(H),
		\end{align}
		else if  $H =  [1, a-1], \,[b+1, c-1]$ or $[d+1, n]$,
		\begin{align}
			f(H)  &:=  (\text{number of balls of color $0$  in $H$ } -  \text{number of balls of color $i$  in $H$})\\
			&=  D_{0,i}(H).
		\end{align}

		Now suppose that $p_{i}$ is the unique maximum among $p_{1},\cdots,p_{\kappa}$ and assume $p_{i}>p_{0}$. Note that $\mathcal{H}$ contains $2\kappa+3$ intervals. Noting that $D_{i,i}(H)=0$, a union bound and Proposition \ref{prop:interval_hoeffding} give 
		\begin{align}
			\P\left( \sum_{H\in \mathcal{H}} f(H) \ge 2 (2\kappa+3) C\log n  \right) 
			&\le \sum_{[s,t]\subseteq [1,n]} \sum_{\substack{0\le \ell \le \kappa \\ \ell\ne i }} \P\left( D_{\ell, i}([s,t]) \ge 2 C\log n \right) \\
			&\le 3n^{2} \sum_{\substack{0\le \ell \le \kappa \\ \ell\ne i }} \exp(  - C (p_{i}-p_{\ell}) \log n)
		\end{align}
		for any fixed constant $C>0$. For sufficiently large constant $C>0$, the last expression tends to zero as $n\rightarrow \infty$, so this shows $\lambda_2 = O(\log n)$ with high probability.
		
		Next, suppose $p_{i}=p^{*}$ at least two distinct colors $1\le i \le \kappa$. If we compare the number of balls of color $j$ in $H\in \mathcal{H}$  minus the number of balls of color $i$ in $H$. By using a similar argument, $D_{j,i}(H)$ is $O(\log n)$ with high probability as long as $p_j<p^{*}$. If $p_j = p^{*}$, then by the triangle inequality, 
		\begin{align}\label{eq:D_bound}
			D_{j,i}(H) \le \max_{1\le s \le t \le n} \left|D_{j,i}([s,t])\right| \le 2 \max_{1\le t \le n} \left| D_{j,i}([1,t]) \right|. 
		\end{align}
		In this case $D_{j,i}([1,t])$ is a  symmetric random walk with $t$ increments. Hence for some large enough constant $C>0$, the right-hand side of \eqref{eq:D_bound} is at most $C\sqrt{n}$ with probability at least $1- \eps$ by the functional central limit theorem. This shows that $\lambda_{2}(n)=O(\sqrt{n})$ with probability at least $1-\eps$.
		
		Now we prove the lower bounds in \textbf{(i)} and \textbf{(ii)}. Fix $j\ge 2$.  Let $A_{1},\dots,A_{j-1}$ denote an optimal choice of non-interlacing subsets of $[1,n]$ such that 
		\begin{align} 
			\lambda_1(n) + \cdots + \lambda_{j-1}(n) = \sum_{i=1}^{j-1} \L(A_{i}, \xi). 
		\end{align}
		Denote $I_{i}:=[\min A_{i}, \max A_{i}]$ for $i=1,\dots, r-1$, so that $I_{1},\dots,I_{j-1}$ are non-interlaving supporting intervals for $A_{1},\dots,A_{j-1}$.  For each interval $J=[s,t]$, let $N_{0}(J)$ denote the maximum number of \textit{consecutive} $0$'s in the sequence $\xi_{s}, \xi_{s+1},\dots,\xi_{t}$. For each integer $1\le \ell \le \kappa$, let $M_{\ell}(J)$ denote the maximum number of $\ell$'s (not necessarily consecutive) in the sequence $\xi_{s}, \xi_{s+1},\dots,\xi_{t}$. We will use these notations for the rest of the proof.

		Fix a constant $0<c_{1}<1/(3\log p_{0}^{-1})$. We first show that $\P(\lambda_{j}(n)/\log n  \ge c_{1}) = 1-o(1)$. To this end, we claim that 
		\begin{align}\label{eq:thm_iid_sup_pf0}
			&	\P\left( N_{0}(I_{i}) \ge c_{1}\log n \,\, \textup{for some $i=1,\dots,r-1$} \right) = 1-o(1). 
		\end{align}
		Note that if $N_{0}(I_{i}) \ge c_{1}\log n$, then we can split the non-increasing subsequence $A_{i}$ into two non-increasing subsequences $A_{i}'$ and $A_{i}''$ by removing the $c_{1}\log n$ consecutive zeros in the supporting interval $I_{i}$. Then $A_{1}\prec \dots A_{i-1} \prec A_{i}' \prec A_{i}'' \prec \dots \prec A_{j-1}$ is a non-interlacing collection of non-increasing subsequences, whose total penalized length has now increased by at least $c_{1} \log n$. Thus by Lemma \ref{lemma:GK_invariants}, $\lambda_{j}(n) \ge c_{1} \log n$ with high probability if the claim \eqref{eq:thm_iid_sup_pf0} holds.

		Now we show \eqref{eq:thm_iid_sup_pf0}. Fix a constant $0<c_{2}<p^{*}-p_{0}$. Since $\L(A_{i}, \xi)\le |I_{i}|$, 
		\begin{align}\label{eq:thm_iid_sup_pf1}
			\P\left(\sum_{i=1}^{j-1} |I_{i}| <c_{2}n \right)  \le 			\P\left( \lambda_{1}(n)+\dots+\lambda_{j-1}(n) <c_{2}n \right)  \le \P(\lambda_{1}(n)<c_{2}n). 
		\end{align}
		Since $\lambda_{1}(n)/n\rightarrow p^{*}-p_{0}>c_{2}$ a.s. by Theorem \ref{thm:iid_supercritical} \textbf{(i)}, the above probability is of order $o(1)$.

		Next, by using a union bound,
		\begin{align}
			&	\P\left( \sum_{i=1}^{j-1} |I_{i}| \ge c_{2}n,\, N_{0}(I_{i}) < c_{1}\log n \,\, \textup{for all $i=1,\dots, j-1$} \right)  \\
			& \quad \le \P\left( \bigcup_{J_{1} \prec \dots \prec J_{j-1} \subseteq [1,n] } \left\{ \sum_{i=1}^{j-1} |J_{i}| \ge c_{2}n,\, N_{0}(J_{i}) < c_{1}\log n \,\, \textup{for all $i=1,\dots,j-1$} \right\} \right)  \\
			& \quad \le \P\left( \bigcup_{J_{1} \prec \dots \prec J_{j-1} \subseteq [1,n] }  \bigcup_{i=1}^{j-1} \left\{  |J_{i}| \ge \frac{c_{2}n}{r-1},\, N_{0}(J_{i}) < c_{1}\log n \right\} \right)  \\
			& \quad \le (r-1) n^{2(r-2)} \sum_{J \subseteq [1,n],\, |J|\ge \frac{c_{2}n}{r-1}} \P\left( N_{0}(J) < c_{1} \log n \right) \\ 
			& \quad \le (r-1) n^{2(r-1)}  \P\left( N_{0}([1,n]) < c_{1} \log n \right), \label{eq:thm_super_pf_1}
		\end{align}
		where $J_{i}$s and $J$ above denote deterministic intervals. We can subdivide the interval $[1,n]$ into consecutive subintervals $K_{1},K_{2},\dots$ of length $\lceil c_{1} \log n \rceil$. There are at least $\lfloor \frac{n}{c_{1}\log n} \rfloor$ such subintervals, and they can be fully occupied with balls of color 0 independently with probability $p_{0}^{\lceil c_{1} \log n \rceil}$. Hence, recalling $0<c_{1}<1/(3\log p_{0}^{-1})$, 
		\begin{align}
			\P\left( N_{0}([1,n]) < c_{1} \log n \right) &\le \left( 1-p_{0}^{ c_{1} \log n} \right)^{ \lfloor \frac{n}{c_{1}\log n} \rfloor} \\
			& \le \exp \left( - p_{0}^{ c_{1} \log n}  \lfloor \frac{n}{c_{1}\log n} \rfloor \right) \le \exp(-n^{1/3}). 
		\end{align}
		Therefore, \eqref{eq:thm_super_pf_1} is of order $o(1)$. Now \eqref{eq:thm_iid_sup_pf0} follows by a union bound.  In particular, this completes the proof of \textbf{(i)}.

		Finally, suppose $p_{\alpha_{1}}= p_{\alpha_{2}} =p^{*}$ for some $1\le \alpha_{1}<  \alpha_{2} \le \kappa$.  Fix $\eps>0$. We will show that there exists a constant $c=c(\eps,j)>0$ such that 
		\begin{align}
			\liminf_{n\rightarrow\infty} 	\P\left( n^{-1/2} \lambda_{j}(n) \ge c \right) \ge 1-\eps.
		\end{align}	
		To this end, we split each $A_{i}$ into  successive disjoint sub-subsequences $A_{i,\kappa}, \cdots , A_{i,2}, A_{i,1}$ where in each $A_{i,\ell}$ we only pick the balls of color $\ell$ in $A_{i}$. Denote $I_{i,\ell}:=[\min A_{i,\ell}, \max A_{i,\ell}]$. By \eqref{eq:thm_iid_sup_pf1} and a union bound, 
		\begin{align}
			\P\left(  | I_{i, \ell}|  \ge \frac{c_{2} n}{\kappa(j-1)} \,\, \textup{for some $1\le i \le j-1$ and $1\le \ell \le \kappa$} \right)=1-o(1).
		\end{align}
		Fix $\delta>0$. Partition $[0,n]$ into intervals $J_{k}:=[ k\delta n, (k+1) \delta n]$ of equal length $\lfloor \delta n \rfloor$. We can choose $\delta$ small enough so that any fixed interval of length $\frac{c_{2} n}{\kappa(j-1)} $ in $[1,n]$  contains $J_{k}$ for some $1\le k\le \lfloor \delta^{-1} \rfloor$. 
		
		For each $1\le \ell \le \kappa$, choose $\ell_{*}\in \{ i_{1}, i_{2} \} \setminus \{\ell\}$. Fix a constant $\alpha>0$ and define the following event 
		\begin{align}
			E_{k,\ell}:= \left\{  \max_{ t \le \delta n}  n^{-1/2} \left| D_{\ell, \ell_{*}} \big([ (k-1) \lfloor \delta n \rfloor, k \lfloor \delta n \rfloor + t] \big) \right|  \ge \alpha \right\}. 
		\end{align}
		Since $D_{i,i_{*}}$ on these disjoint intervals are i.i.d., by the functional central limit theorem, we have 
		\begin{align}
			\liminf_{n\rightarrow \infty}\, \P\left(  \bigcap_{k=1}^{\lfloor 1/\delta \rfloor }  \bigcap_{\ell=1}^{\kappa}  E_{k,\ell}  \right) \ge 1-\frac{\eps}{2}
		\end{align}
		as long as the constant $\alpha>0$ is small enough.  By a union bound, for all $n\ge 1$ sufficiently large, 
		\begin{align}
			&	\P\left( \left\{  \textup{$J_{k}\subseteq I_{i,\ell}$}  \,\, \textup{for some $k,i,\ell$}\right\} \cap E_{k,\ell} \right)  \ge 1-\eps. 
		\end{align}
		We now claim that 
		\begin{align}
			\left\{  \textup{$J_{k}\subseteq I_{i,\ell}$}  \,\, \textup{for some $k,i,\ell$} \right\} \cap E_{k,\ell}  \subseteq \{ \lambda_{j}(n) \ge \alpha\sqrt{n}  \},
		\end{align}
		which is enough to conclude the desired lower bound $\lambda_{j}(n) = \Omega(\sqrt{n})$. To show this claim, suppose the event on the left-hand side above holds. Denote $I_{i,\ell}=[e,f]$. The maximum of $D_{\ell, \ell_{*}}$ in the event $E_{k,\ell}$  occurs at site $m$ in $J_{k}$, so we may split the interval $[e, f]$ into $[e, m]$ and $[m+1, f]$.  Supppose  $D_{\ell, \ell_{*}}([e,m]) \ge \alpha \sqrt{n}$. Let $A_{i,\ell}^{-}$ and $A_{i,\ell}^{+}$ denote the subsequences formed by picking up all $\ell$'s in $[e,m]$ and all $\ell_{*}$'s in $[m+1, f]$, respectively. Now define two non-increasing subsequences $A_{i}',A_{i}''$ by 
		\begin{align}
			\begin{cases}			
				A_{i}':=[A_{i,\kappa},\dots,A_{i,\ell+1}, A_{i,\ell}^{-}, A_{i,\ell}^{+}],\quad A_{i}'':=[A_{i,\ell-1},\dots,A_{i,1}] & \textup{if $\ell>\ell_{*}$} \\
				A_{i}':=[A_{i,\kappa},\dots,A_{i,\ell+1}, A_{i,\ell}^{-}],\quad A_{i}'':=[A_{i,\ell}^{+}, A_{i,\ell-1},\dots,A_{i,1}] & \textup{if $\ell<\ell_{*}$}.
			\end{cases}
		\end{align}
		Together with the other $j-2$ subsequences $A_{1},\dots,A_{i-1}, A_{i+1},\dots,A_{j-1}$, these $j$ non-interlacing  and non-increasing subsequences achieve total penalized lengths at least $\lambda_{1}(n)+\dots+\lambda_{j-1}(n)+\alpha\sqrt{n}$. By Lemma \ref{lemma:GK_invariants}, this implies $\lambda_{j}(n) \ge \alpha\sqrt{n}$. If  $D_{\ell, \ell_{*}}([e,m]) \le -\alpha \sqrt{n}$, then let $A_{i,\ell}^{-}$ and $A_{i,\ell}^{+}$ denote the subsequences formed by picking up all $\ell_{*}$'s in $[e,m]$ and all $\ell$'s in $[m+1, f]$, respectively, and define 
		\begin{align}
			\begin{cases}			
				A_{i}':=[A_{i,\kappa},\dots,A_{i,\ell+1}, A_{i,\ell}^{-}],\quad A_{i}'':=[A_{i,\ell}^{+}, A_{i,\ell-1},\dots,A_{i,1}] & \textup{if $\ell>\ell_{*}$}, \\
				A_{i}':=[A_{i,\kappa},\dots,A_{i,\ell+1}, A_{i,\ell}^{-}, A_{i,\ell}^{+}],\quad A_{i}'':=[A_{i,\ell-1},\dots,A_{i,1}] & \textup{if $\ell<\ell_{*}$}.
			\end{cases}
		\end{align}
		In this case, we can also conclude $\lambda_{j}(n) \ge \alpha\sqrt{n}$ similarly. This completes the proof. 
	\end{proof}

	\vspace{0.3cm}
	\section{Proofs of combinatorial lemmas}
	
	\label{section:proof_combinatorial_lemmas}
	
	In this section, we establish various combinatorial statements about the $\kappa$-color BBS dynamics and the associated carrier processes. Our main goal is to show Lemmas \ref{lemma:queue_formula_soliton},  \ref{lemma:soliton_lengths_excursions}, and \ref{lemma:GK_invariants}. We also provide an elementary and self-contained proof of Lemma \ref{lemma:carrier_row_lengths}, which has been proved in the more general form in \cite[Prop. 4.5]{kuniba2020large} using connections with combinatorial $R$.

	\subsection{Proof of Lemmas \ref{lemma:queue_formula_soliton} and \ref{lemma:soliton_lengths_excursions}}
	
	In this subsection, we prove Lemmas \ref{lemma:queue_formula_soliton} and \ref{lemma:soliton_lengths_excursions}. We rely on the finite-capacity carriers (see Section \ref{subsection:carrier_process_finite}) and Lemma \ref{lemma:carrier_row_lengths}. We need an additional combinatorial observation about the `coupling' between the carrier processes of capacity $c$ and $c+1$ over the same BBS configuration, which is stated below.

	\begin{prop}\label{prop:finite_capacity_carrier_coupling}
		Let $\xi:\mathbb{N}\rightarrow \mathbb{Z}_{\kappa+1}$ be any $\kappa$-color BBS configuration with finite support. Denote by  $(\Gamma_{x;c})_{x\ge 0}$ and $(\Gamma_{x;c+1})_{t\ge 0}$ the carrier processes over $\xi$ with finite capacities $c$ and $c+1$, respectively. Then for any $t\ge 0$, $\Gamma_{x;c}$ viewed as a $c$-dimensional vector is obtained by omitting a single coordinate in $\Gamma_{x;c+1}$ viewed as a $c+1$-dimensional vector.
	\end{prop}
	
	\begin{proof}
		Fix a $\kappa$-color BBS configuration $\xi:\mathbb{N}\rightarrow \mathbb{Z}_{\kappa+1}$. Let $(\Gamma_{x;c})_{x\ge 0}$ and $(\Gamma_{x;c+1})_{x\ge 1}$ denote the carrier processes over $\xi$ with finite capacities $c$ and $c+1$, respectively. We will show the assertion by induction on $x\ge 0$. For $x=0$, both carriers are filled with zeros so omitting any entry of $\Gamma_{0;c+1}$ gives $\Gamma_{0;c}$. For the induction step, suppose the assertion holds for some $x\ge 0$. Denote $S=\Gamma_{x;c}, T=\Gamma_{x+1;c}\in \mathcal{B}_{c}$ and $S'=\Gamma_{x;c+1}, T'=\Gamma_{x+1;c+1}\in \mathcal{B}_{c+1}$. Recall that the entries in carrier states are non-increasing from left, which is the opposite of the convention for semistandard Young tableaux (as used in \cite{kuniba2020large} and \cite{kuniba2018randomized}).  
		
		By the induction hypothesis, we may assume that $S$ can be obtained from $T$ by omitting its $j_{*}$th entry $T(j_{*})=r$. Let $B$ and $A$ be the blocks to the left and right of the entry $T(j_{*})$ of $T$. Hence $S$ is the concatenation of the blocks $B$ and $A$ (see Figure \ref{fig:carrier_coupling_pf} left). Let $q:=\xi_{x+1}$.    
		
		\begin{figure*}[h]
			\centering
			\includegraphics[width=0.7 \linewidth]{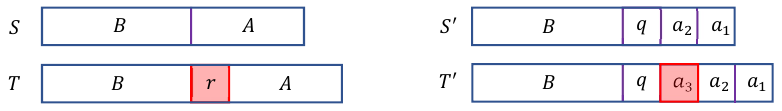}
			\caption{ (Left) $S\in \mathcal{B}_{c}$ is obtained from $T\in \mathcal{B}_{c+1}$ by omitting an entry $r$. (Right) After inserting $q$ into $T$ and $S$ according to the circular exclusion rule, one can still omit a single entry from the larger tableau to get the smaller one.  
			}
			\label{fig:carrier_coupling_pf}
		\end{figure*}
		
		First, suppose that $q$ does not exceed the smallest entry of $T$. In this case inserting $q$ into $T$ replaces the largest entry of $T$, so  $T'$ is given by $T'(j)=T(j+1)$ for $1\le j \le c$ and $T'(c+1) = q$. We also have $S'(j) = S(j+1)$ for $1\le j < c$ and $S'(c)=q$. It follows that $S'$ is obtained by omitting the same entry $r=T'(j_{*}-1)$ from $T'$. 
		
		Second, suppose that $q$ exceeds the smallest entry of $T$. so that $T'$ is computed from the pair $(T,q)$ using the reverse bumping. If $q$ replaces some entry of $A$ or $B$ in $T$ to get $T'$, then the same replacement occurs to compute $S'$ from the pair $(S,q)$. Hence in this case $S'$ is obtained by omitting $r=T'(j_{*})$ from $T'$. Otherwise, $q$ replaces $r$ in $T$ to get $T'$ (see in Figure \ref{fig:carrier_coupling_pf} right). Then $q$ must replace the largest entry of $A$ in $S$ to get $S'$. Then $S'$ is obtained from $T'$ by deleting the largest entry in $A$. This shows the assertion. 
	\end{proof}

	\begin{proof}[\textbf{Proof of Lemma} \ref{lemma:queue_formula_soliton}]
		Fix a $\kappa$-color BBS configuration $\xi:\mathbb{N}\rightarrow \mathbb{Z}_{\kappa+1}$. For each integer $c\ge 1$, let   $(\Gamma_{x;c})_{x\ge 0}$ denote the capacity-$c$ carrier process over $\xi$. Let $(\Gamma_{x})_{x\ge 0}$ denote the infinite capacity carrier process over $\xi$. We also write 
		\begin{align}
			M=\max_{s\ge 0} \left( \text{$\#$ of nonzero entries in $\Gamma_{s}$} \right)
		\end{align}
		Note that from Lemma \ref{lemma:carrier_row_lengths}, we can deduce that  for any $1\le j\le \rho_{1}(\xi)$, 
		\begin{align}\label{eq:lambda_energt_formula}
			\lambda_{j}(\xi) = | \{ k\ge 1\,:\, \rho_{k}(\xi)\ge j \} | = \max\left\{k\ge 1\,\bigg|\, E_{k}(\xi)\ge E_{k-1}(\xi)+j\right\},
		\end{align}
		where $E_{k}(\xi)$ is defined in \eqref{def:energy_X}.
		
		Let $\tau_{c}$ be the first time $t$ that the carrier $\Gamma_{x;c}$ is completely full with nonzero entries and $X_{0}(x+1)>0$ does not exceed the smallest entry of $\Gamma_{x;c}$. More precisely, let 
		\begin{align}
			\tau_{c}:= \inf \big\{x\ge 0\mid \text{$\Gamma_{x;c}$ contains all positive entries and  $0<\xi_{x+1}\le \min \Gamma_{x;c}(x)$}\big\}.
		\end{align}
		We let $\tau_{c}=\infty$ if the set on the right-hand side is empty. Note that if we consider two carrier processes $\Gamma_{x;c}$  and $\Gamma_{x;c+1}$, then $\tau_{c}+1$ is the first time that they contain distinct sets of nonzero entries. Moreover, $\Gamma_{\tau_{c}+1;c+1}$ has $c+1$ nonzero entries. Hence if $c\ge M$, then $\tau_{c}=\infty$ and the two carrier processes have the same set of nonzero entries for all times. It follows that 
		\begin{align}
			E_{c} = Const. \qquad \forall c\ge M.
		\end{align}
		Hence $\lambda_{1}(\xi)\le M$ by \eqref{eq:lambda_energt_formula}.

		On the other hand, note that $x^{*}:=\tau_{M-1}<\infty$ and $\xi_{x^{*}+1}$ does not exceed the smallest entry in $\Gamma_{x^{*};M-1}$ by definition of $\tau_{M-1}$. So $\mathbf{1}(\xi_{x^{*}+1} >\min \Gamma_{x^{*};M-1})=0$. Also, since $\Gamma_{x^{*};M-1}$ and $\Gamma_{x^{*};M}$ share the same positive entries, $\Gamma_{x^{*};M}$ is obtained from $\Gamma_{x^{*};M-1}$ by augmenting $0$ to its right. Since $\xi_{x^{*}+1}>0$ by definition of $x^{*}$, we have $\mathbf{1}(\xi_{x^{*}+1}>\min \Gamma_{x^{*};M})=1$. Moreover, by Proposition \ref{prop:finite_capacity_carrier_coupling}, 
		\begin{align}
			\mathbf{1}(\xi_{x+1}>\min \Gamma_{x;c}) \ge \mathbf{1}(\xi_{x+1}>\min \Gamma_{x;c-1})
		\end{align}
		for all $c\ge 1$ and $x\ge 0$. It follows that $E_{M}\ge E_{M-1}+1$. Hence by \eqref{eq:lambda_energt_formula}, we deduce $\lambda_{1}(\xi) \ge M$. This shows $\lambda_{1}(\xi)=M$, as desired.
	\end{proof}

	\begin{proof}[\textbf{Proof of Lemma \ref{lemma:soliton_lengths_excursions}}]
		Fix a $\kappa$-color BBS configuration $\xi$ with finitely many balls of positive colors. Let $W:=(W_{x})_{x\ge 0}$ be the carrier process over $\xi$. Let $T_{0}:=0$ and let $T_{k}$ for $k\ge 1$ denote the $k$th site that the carrier returns to the origin. Define sub-configurations $\xi^{(1)}:=(\xi_{0},\xi_{1},\dots,\xi_{T_{1}-1})$, $\xi^{(2)}:=(\xi_{T_{1}}, \xi_{T_{1}+1}, \dots, \xi_{T_{2}-1} )$, and so on. Let $N$ denote the number of nontrivial excursions of the carrier process $W$. Then $\xi$ is the concatenation of $\xi^{(1)},\dots,\xi^{(N)}$. We wish to show that the soliton decomposition of $\xi$ is the union of the soliton decomposition of $\xi^{(i)}$'s. Equivalently, we wish to show that 
		\begin{align}\label{eq:soliton_decomp_claim1}
			\rho_{c}(\xi) = \sum_{k=1}^{N} \rho_{c}(\xi^{(k)})\quad \textup{for all $c\ge 1$}.
		\end{align}

		To show the claim \eqref{eq:soliton_decomp_claim1} above, let $(\Gamma_{x;c})_{x\ge 0}$ denote the capacity-$c$ carrier process over $\xi$. By Proposition \ref{prop:finite_capacity_carrier_coupling}, we have $\Gamma_{T_{k};c}=\mathbf{0}$ for all $k\ge 0$. In words, the capacity-$c$ carrier resets to empty at each site $T_{k}$. Hence, if we let $(\Gamma_{x;c}^{(k)})_{T_{k-1}\le x < T_{k}}$ denote the capacity-$c$ carrier process over $\xi^{(k)}$, then 
		\begin{align}
			(\Gamma_{x;c}^{(k)})_{T_{k-1}\le x < T_{k}} = (\Gamma_{x;c})_{T_{k-1}\le x < T_{k}}. 
		\end{align}
		It follows that 
		\begin{align}
			\sum_{x=1}^{N} \mathbf{1}(\xi_{s}>\min \Gamma_{x-1;c}) = \sum_{k=1}^{N}  \sum_{T_{k-1}< x \le T_{k}} \mathbf{1}(\xi_{s}^{(k)}>\min \Gamma^{(k)}_{x-1;c}).
		\end{align}
		By Lemma \ref{lemma:carrier_row_lengths}, the above yields 
		\begin{align}
			\rho_{1}(\xi)+\dots+\rho_{c}(\xi) = \sum_{k=1}^{N}  	\rho_{1}(\xi^{(k)})+\dots+\rho_{c}(\xi^{(k)}). 
		\end{align}
		The above holds for all $c\ge 1$. By using induction in $c$, one can then deduce  \eqref{eq:soliton_decomp_claim1}. 
		
		The second part of the assertion that $\lambda_{j}(n)\ge \mathbf{h}_{j}(n)$ is immediate from the first part we have just shown above and Lemma \ref{lemma:queue_formula_soliton}. 
	\end{proof}

	\subsection{Proof of Lemmas  \ref{lemma:carrier_row_lengths} and \ref{lemma:GK_invariants}}
	\label{section:pf_lem_row_GK}
	
	Recall the notations introduced in Section \ref{subsection:GK_invariants}. 
	For any $\kappa$-color BBS configuration $X:\mathbb{N}\rightarrow \mathbb{Z}_{\kappa+1}$ with finite support and integer $k\ge 1$, we denote 
	\begin{align}
		R_{k}(\xi):=\max_{A_{1}\sqcup \cdots \sqcup A_{k}}\sum_{i=1}^{k} \NA(A_{i},\xi) ,\qquad L_{k}(\xi):=\max_{A_{1}\prec\cdots \prec A_{k} \subseteq \mathbb{N} } \sum_{i=1}^{k} \L(A_{i},\xi).
	\end{align}
	Lastly, we also denote 
	\begin{align}\label{def:energy_X}
		E_{k}(\xi) := \sum_{s=1}^{\infty} \mathbf{1}(\xi_{s}>\min \Gamma_{s-1;k})
	\end{align}
	where $(\Gamma_{x;i})_{t\ge 0}$ is the capacity-$i$ carrier process over $\xi$. We set $R_{0}(\xi)=L_{0}(\xi)=E_{0}(\xi)=0$ for convenience. In this subsection, we will show with an elementary argument that the above quantities associated with a $\kappa$-color BBS configuration are invariant under time evolution. This will lead to the proof of Lemmas \ref{lemma:GK_invariants} and \ref{lemma:carrier_row_lengths}.
	
	We remark that the invariants $E_{k}(\xi)$ are called the \textit{energy}.  They were first introduced in \cite{fukuda2000energy} for the $\kappa=1$ BBS and were recently used to define an \textit{energy matrix} for the general $\kappa$-color BBS that characterizes the full set of invariants. Time invariance of the energy (and also the energy matrix) in the literature is usually shown by using the alternative characterization of the BBS dynamics in terms of combinatorial $R$ and connections to the Yang-Baxter equation \cite{fukuda2000energy, inoue2012integrable, kuniba2020large, kuniba2018randomized}.

	Recall the BBS evolution rule defined in the introduction: For $i = \kappa, \kappa - 1, \cdots, 1$, the balls of color $i$ each make one jump to the right, into the first available empty box (site with color 0), with balls that start to the left jumping before balls that start to their right. (This is the map $K_{i}$ defined in the introduction.) A single step of $\kappa$-color BBS evolution $X\mapsto X'$ is defined by 
	\begin{align}
		\xi' := K_{1}\circ K_{2}\circ \cdots \circ K_{\kappa} (\xi). 
	\end{align}
	
	We propose two ways to simplify the $\kappa$-color BBS dynamics. First, using the cyclic symmetry of the system, we can reformulate the update of a $\kappa$-color BBS configuration in terms of $\kappa$ applications of a single rule. Namely, let $\mathcal{T}_{\kappa}$ denote the following update rule for BBS configurations with finite support: all the balls of color $\kappa$ jump according to the rule $K_{\kappa}$, and we relabel each of them with color $1$ and increase the positive colors of all other balls by $1$. Then we have 
	\begin{align}\label{eq:cyclic_transition_rule_equiv}
		K_{1}\circ K_{2}\circ \cdots \circ K_{\kappa} (\xi) = (\mathcal{T}_{\kappa})^{\kappa}(\xi).
	\end{align}
	
	Second, we introduce ``standardization'' of BBS dynamics, which allows us to only consider BBS configurations with no repeated use of any positive color. Namely, given a $\kappa$-color BBS configuration $\xi:\mathbb{N}\rightarrow \mathbb{Z}_{\kappa+1}$ of finite support, we define its \textit{standardization} to be the following map $\hat{\xi}:\mathbb{N}\rightarrow \mathbb{Z}_{\ge 0}$: For each $1\le i \le \kappa$, let $m_{i}$ denote the number of balls in $X$ of color $i$. Then to produce $\hat{\xi}$, we relabel first the color 1 balls from 1 to $m_{1}$ from right to left (so that the leftmost ball that was previously colored $1$ is now colored $m_1$), and then the original color 2 balls are relabeled with colors $m_{1}+1$ to $m_{1}+m_{2}$ from right to left, and so on. Thus, if $N=\sum_{i=1}^{\kappa}m_{i}$ is the total number of balls of positive color then $\hat{\xi}$ is an $N$-color BBS configuration with each color in $\{1,\cdots, N\}$ used for exactly one ball.

	\begin{prop}\label{prop:BBS_standardization}
		Let $\xi$ and $\hat{\xi}$ denote a $\kappa$-color BBS configuration with finite support and its standardization, respectively. Then the following hold. 
		\begin{description}	
			\item[(i)] Standardization preserves the number of ascents, non-interlacing non-increasing sequences, and their penalized lengths. In particular, for each $k\ge 1$, 
			\begin{align}
				R_{k}(\xi) = R_{k}(\hat{\xi}),\qquad L_{k}(\xi) = L_{k}(\hat{\xi}). 
			\end{align}
			
			\item[(ii)] $\xi$ and $\hat{\xi}$ give the same soliton partition, i.e., $\Lambda(\xi)=\Lambda(\hat{\xi})$.
		\end{description}
	\end{prop}
	
	\begin{proof}
		By construction, standardization preserves ordering in the following sense: for $y < z$, one has $\xi_{y} < \xi_{z}$ if and only if $\hat{\xi}(y) < \hat{\xi}(z)$.  Thus, a given sequence of balls has an ascent in $X$ if and only if it has an ascent in $\hat{\xi}$, and likewise, a given sequence of balls is non-increasing in $\xi$ if and only if it is non-increasing in $\hat{\xi}$.  Part \textbf{(i)} follows immediately.
		
		To show \textbf{(ii)}, denote by $\xi'$ and $(\hat{\xi})'$ the BBS configurations obtained by applying one step of the BBS evolution rule to $\xi$ and $\hat{\xi}$, respectively. Since standardization does not change the location of balls, it suffices to show that standardization commutes with BBS time evolution rules, i.e., 
		\begin{align}\label{eq:BBS_standardization_pf}
			\hat{\xi'} = (\hat{\xi})'.
		\end{align}
		To see this, observe that for the evolution $\xi\mapsto \xi'$, after all, balls of color $\kappa$ have jumped, they return to the same left-right order as before: if some ball of color $\kappa$, say in position $x$, jumped over some other ball of color $\kappa$, say in position $y$, to land in position $z$ (so $x < y < z$), it must be the case that sites between $y$ and $z$ were occupied.  Therefore, when it is time for the ball in position $y$ to jump, it jumps over all sites in $(y, z]$. Hence in the first step, the balls of color $\kappa$ in the previous step are triggered one by one from left, and since they restore the same left-right order, they will continue to be triggered in this order in all future steps. This exactly agrees with the time evolution $\hat{\xi}\mapsto \hat{\xi}'$. This shows \eqref{eq:BBS_standardization_pf}, as desired. 
	\end{proof}
	
	In the following proposition, we show the time-invariance of the three quantities associated with a given BBS configuration. This will show most of Lemma \ref{lemma:GK_invariants}.
	
	\begin{prop}\label{prop:invariants_proof}
		Let $\xi$ be an arbitrary $\kappa$-color BBS configuration of finite support. Fix $j\ge 1$. The following hold. 
		\begin{description}
			\item[(i)] $E_{j}(\xi)=E_{j}(\mathcal{T}_{\kappa}(\xi))$.  
			
			\vspace{0.1cm}
			\item[(ii)] $R_{j}(\xi)=E_{j}(\xi)$.
			
			\vspace{0.1cm}
			\item[(iii)] $L_{j}(\xi)=L_{j}(\mathcal{T}_{\kappa}(\xi))$.
			
			\vspace{0.1cm}
			\item[(iv)] If $(\xi^{(t)})_{t\ge 0}$ denotes the $\kappa$-color BBS trajectory with $\xi=\xi_{0}$, then for all $t\ge 1$, 
			\begin{align}
				E_{j}(\xi^{(t)})=R_{j}(\xi^{(t)})\equiv E_{j}(\xi), \qquad L_{j}(\xi^{(t)})\equiv L_{j}(\xi). 
			\end{align} 
		\end{description}
	\end{prop}
	
	We first derive Lemmas \ref{lemma:GK_invariants} and \ref{lemma:carrier_row_lengths} assuming Proposition \ref{prop:invariants_proof}.

	\begin{proof}[\textbf{Proof of Lemma  \ref{lemma:carrier_row_lengths} and \ref{lemma:GK_invariants}}]
		
		Let $(\xi^{(t)})_{t\ge 0}$ be a $\kappa$-color BBS trajectory such that $\xi_{0}$ has finite support. We take $T\ge 1$ large enough so that at time $T$ the system decomposes into non-interacting solitons whose lengths are non-decreasing from left. We can reformulate the condition that a $\kappa$-color BBS configuration has reached its soliton decomposition as follows: Suppose two consecutive solitons are separated by $g$ 0's, where the left and right solitons have length $l$ and $r$, where `length' of a soliton is its number of balls of positive colors. Suppose the gap is small, i.e., $g<l$. In order for the left soliton to be preserved during the update $\xi^{(T)}\mapsto \xi^{(T+1)}$, all balls in the left soliton must be dropped by the carrier before any balls in the right soliton are dropped. It follows that for each $i\ge 1$, the following `separation condition' must hold at time $T$:
		\begin{align}\label{eq:separation_condition}
			\begin{matrix}
				\text{The $i$th largest entry of the right soliton is  strictly larger} \\
				\text{than the $i + g$th largest entry of the left soliton.}
			\end{matrix}
		\end{align}
		When $\kappa = 1$, this simply asserts that each soliton of length $l$ must be followed by at least $l$ empty sites. This is not the case for $\kappa > 1$, as illustrated in the example
		\begin{align}
			\cdots 00433200431100\cdots.
		\end{align}
		For each $k\ge 1$, let $\lambda_{k}$ denote the length of the $k$th-longest soliton and let $\rho_{k}$ denote the number of solitons of length $\ge k$. They both form the same Young diagram, whose $k$th column and row lengths are given by $\lambda_{k}$ and $\rho_{k}$, respectively. 
		
		For each $j\ge 1$, let $(\Gamma_{s;j})_{s\ge 0}$ denote the capacity-$j$ carrier process on $\xi^{(t)}$. As the carrier process over $\xi^{(t)}$ runs over a soliton of length $k$, the carrier obtains $\min(k,j)$ contribution to the energy. When the carrier was empty at the beginning of the soliton, this is clear, and otherwise, it is still true due to the separation condition \eqref{eq:separation_condition}. Hence we have 
		\begin{align}
			E_{j}(\xi^{(T)}) = \sum_{k=1}^{\infty} \min(\lambda_{k}, j) = \sum_{k=1}^{j}\rho_{k}.
		\end{align}
		Then by Proposition \ref{prop:invariants_proof}, we deduce 
		\begin{align}
			R_{j}(\xi^{(t)})=E_{j}(\xi^{(t)}) = E_{j}(\xi^{(T)}) =  \sum_{k=1}^{j}\rho_{k}
		\end{align}
		for all $t\ge 0$, as desired. In the general case, the above equations hold due to the separation condition \eqref{eq:separation_condition}. This shows Lemma \ref{lemma:carrier_row_lengths} as well as the first equation in Lemma \ref{lemma:GK_invariants}. 
		
		Similarly, for the second equation in Lemma \ref{lemma:GK_invariants}, it suffices to show $L_{j}(\xi^{(T)})=\lambda_{1}+\cdots+\lambda_{j}$.  
		It is easy to see $L_{j}(\xi^{(T)})\ge \lambda_{1}+\cdots+\lambda_{j}$ by choosing the $j$ longest non-increasing sequences given by the top $j$ solitons. It remains to show the converse inequality, choose a collection of non-interlacing non-increasing subsequences on supports $A_{1},A_{2},\cdots,A_{j}$ that achieves $L_{j}(\xi^{(T)})$. We may assume that $|A_{1}|+\cdots+|A_{j}|$ is as small as possible, where $|\cdot|$ means (non-penalized) cardinality. We claim that every $A_{i}$ is contained in the support of a single soliton (where it has positive colors). Then clearly the maximum sum of penalized lengths is achieved when $A_{i}$'s are the support of the $j$ longest non-increasing sequences given by the solitons, which shows the assertion. 
		
		To show the claim, for each $i\ge 1$, let $\mathbf{u}_{i}$ denote the maximal non-increasing subsequence of positive colors in the $i$th longest soliton in $\xi^{(T)}$. Schematically, we can write $\xi^{(T)}$ as 
		\begin{align}
			\xi^{(T)} : \quad \cdots \mathbf{u}_{3}  0\cdots 0\mathbf{u}_{2} 0\cdots 0\mathbf{u}_{1} 00 \cdots.
		\end{align} 
		Let $l_{i}$ denote the number of 0's between $\mathbf{u}_{i+1}$ and $\mathbf{u}_{i}$. 
		
		Suppose for contradiction that some $A_{k}$ intersects with two $\mathbf{u}_{i}$'s. Let $i$ be as small as possible so that $A_{k}$ intersects with $\mathbf{u}_{i+1}$ and $\mathbf{u}_{i}$. We first suppose the case when the two solitons have a sufficient gap, i.e., $l_{i+1}\ge \lambda_{i+1}$. Let $A_{k}' = A_{k} \setminus \mathbf{u}_{i+1}$. Then $A_{1},\cdots, A_{k-1},A_{k}',A_{k+1},\cdots,A_{j}$ is a sequence of non-interlacing non-increasing subsequences in $\xi^{(t)}$ with a strictly smaller total number of elements than the original sequence. Moreover, this new sequence achieves the optimum $L_{j}(\xi^{(T)})$ since   
		\begin{align}
			\L(A_{k}',\xi^{(T)}) \ge \L(A_{k}, \xi^{(T)}) - \mathbf{u}_{i+1} + l_{i} \ge \L(A_{k}, \xi^{(T)}).
		\end{align}
		Namely, omitting all elements of $\mathbf{u}_{i+1}$ from $A_{k}$ deletes at most $|\mathbf{u}_{i+1}|$ positive numbers but at least $l_{i}\ge |\mathbf{u}_{i+1}|$ zeros. This contradicts the minimality of the original sequence $A_{1},\cdots,A_{j}$. This shows the claim. Lastly, when the gap between the solitons is small, i.e., $l_{i+1}<\lambda_{i+1}$, one can argue similarly by using the separation condition \eqref{eq:separation_condition}. This shows the claim, as desired.  
	\end{proof} 
	
	Lastly in this subsection, we prove Proposition \ref{prop:invariants_proof}.

	\begin{proof}[\textbf{Proof of Proposition \ref{prop:invariants_proof}}]
		\textbf{(iv)} immediately follows from \textbf{(i)}-\textbf{(iii)}. According to Proposition \ref{prop:BBS_standardization}, the assertion is valid for arbitrary BBS if and only if it is true for the standardized system with initial configuration $\hat{\xi}$, where each positive color is used exactly once. Hence, without loss of generality, we may assume that each positive color in $\xi$ is used exactly once. Furthermore, in proving (i)-(iii), we may assume that there is a ball of color $\kappa$ in $\xi$, since otherwise the cyclic update rule $\mathcal{T}_{\kappa}$ simply increases all positive colors by $1$. Since all the invariants depend only on the relative ordering between ball colors, the assertion holds trivially. We will also denote $\xi'=\mathcal{T}_{\kappa}(\xi)$.  For any string $\mathbf{u}$ of integers in $\{0,1,\dots,\kappa-1\}$, we let $\mathbf{u}'$ denote the string obtained by incrementing the positive integers in $\mathbf{u}$ by one. 
		
		\begin{description}
			\item[(i)] Suppose $\xi_{x}=\kappa$ and the ball of color $\kappa$ is in a contiguous block of balls whose labels are $\mathbf{u}\kappa \mathbf{v}0\mathbf{w}$ for some words $\mathbf{u}, \mathbf{v}$. Note that $\mathbf{u}$ and $\mathbf{w}$ consist of integers in $\{0,\dots,\kappa-1\}$, while $\mathbf{v}$ is either empty or only has positive integers $<\kappa$. After the update $\xi\mapsto \xi':=\mathcal{T}_{\kappa}(\xi)$, we reach an arrangement in which $\mathbf{u}, \mathbf{v}$, and $\mathbf{w}$ have had their labels incremented, the space between them is empty ($\xi'_{x}=0$), and $1$ follows $\mathbf{v}$. Let $y$ be the site such that $\xi'_{y}=1$. Here is a schematic:
			\[
			\begin{array}{c|ccccccc}
				\textrm{configuration } & &&& \textrm{arrangement} &&& \\
				\hline
				\xi &  &  & [\,\; \cdots \;\; \mathbf{u} \;\; \cdots \,\;] & \kappa & [\,\; \cdots \;\; \mathbf{v} \;\; \cdots \,\;] & 0 &\mathbf{w} \\[24pt]
				\xi'=\mathcal{T}_{\kappa}(\xi) &  &  &  [\,\; \cdots \;\; \mathbf{u}' \;\; \cdots \,\;]  & 0 &[\,\; \cdots \;\; \mathbf{v}' \;\; \cdots \,\;] & 1 & \mathbf{w}' 
			\end{array}
			\]
			
			Consider running the capacity-$j$ carrier over $\xi$ and $\mathcal{T}_{\kappa}(\xi)$ and computing their energies $E_{j}(\xi)$ and $E_{j}(\xi')$. Let the corresponding carrier processes be denoted by $\Gamma:=(\Gamma_{x})_{x\ge 0}$ and $\Gamma':=(\Gamma_{x}')_{t\ge 0}$, respectively. Observe that up to `time' $x-1$, the two carriers go through the equivalent environments $\mathbf{u}$ and $\mathbf{u}'$, so $\Gamma_{x-1}'$ can be obtained from $\Gamma_{x-1}$ by adding 1 to all positive colors in the latter carrier. It follows that the contributions to the energies of both carry up to this point are the same. 
			
			Next, after inserting $\xi_{x}=\kappa$ and $\xi'_{x}=0$ into these carriers, we get carrier states $\Gamma_{x}=[\kappa, A, 0\cdots 0]$ and  $\Gamma_{x}'=[A', 0\cdots 0]$ for some (possibly empty) positive decreasing sequence $A$ (see Figure \ref{fig:energy_invairance} left). This only adds 1 to the energy for the carrier $\Gamma$. Also note that, since $\kappa$ is the unique largest color in the system, it sits in the carrier $\Gamma$ and does not interact with any other incoming balls thereafter. We can think of this as the capacity of the carrier $\Gamma$ being decreased to $j-1$ after time $x$. Then over the interval $(x,\infty)$, the carriers go through the input $[\mathbf{v}0\mathbf{w}]$ and $[\mathbf{v}'1\mathbf{w}']$, respectively.
			
			\begin{figure*}[h]
				\centering
				\includegraphics[width=0.8 \linewidth]{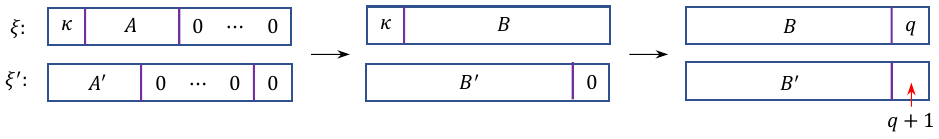}
				\caption{ Two capacity-$j$ carriers over $\xi$ and $\xi'=\mathcal{T}_{\kappa}(\xi)$. They end up with the same energy. 
				}
				\label{fig:energy_invairance}
			\end{figure*}
			
			Ignoring $\kappa$ in the carrier $\Gamma$ and shift by 1, they both have the same dynamics (and hence the same contribution to the energy) until the first time $x^{*}$ that $\Gamma_{x^{*}}$ is full and a new ball of color $\xi_{x^{*}+1}=q>\min \Gamma_{x^{*}}$. In this case, $q+1$ replaces 0 in $\Gamma'_{x^{*}}$ but it replaces $\kappa$ in $\Gamma_{x^{*}}$. If such $x^{*}$ is not encountered up to the location $y$ of $1$ in $\xi'$, then at site $y$, 0 replaces the maximum entry in $\Gamma_{y}$ but $1$ replaces 0 in $\Gamma_{y}'$, so this makes up the energy gap of 1 between the two carriers. Otherwise, suppose there exists such $x^{*}$ between $x$ and $y$. Then we can write the carrier states as $\Gamma_{x^{*}}=[\kappa, B]$ and $\Gamma_{x^{*}}'=[B+1, 0]$ for some positive decreasing sequence $B$ of length $j-1$. Then since $\xi_{x^{*}+1}=q>\min \Gamma_{x^{*}}$, inserting $q$ (resp., $q+1$) into $\Gamma_{x^{*}}$ (resp., $\Gamma_{x^{*}}'$) replaces $\kappa$ (resp., 0), only adding 1 to the energy for $\Gamma'$. Then $\Gamma_{x^{*}+1}=[B,q]$ and $\Gamma_{x^{*}+1}'=[B+1,q+1]$ and all colors in $\Gamma'$ are at least 2, so inserting 0 and 1 at site $y$ does not increment energies of both carriers. Hence they end up with the same energy. This shows the assertion. 
			
			\vspace{0.1cm}
			\item[(ii)] Let $(\Gamma_{x})_{x\ge 0}$ denote the capacity-$j$ carrier process over $\xi$. We will partition the sites that contain balls of positive colors into $j$ disjoint sets $A_{1},\dots,A_{j}$ such that if $x\in A_{i}$ and the energy $E_{j}$ increases when inserting the ball $\xi_{x}$ into the carrier $\Gamma_{x-1}$, then either $x$ is the rightmost (smallest) element of $A_{i}$ or there exists a unique $y\in A_{i}$ such that $(y,x)$ counts as an ascent in $A_{i}$. The existence of such subsets $A_{1},\dots,A_{j}$ implies that 
			\begin{align}
				R_{j}(\xi)\ge \sum_{i=1}^{j} \NA(A_{i},\xi) \ge E_{j}(\xi).
			\end{align}

			For this proof, we will consider sites with color zero as having a ball of color zero. We will recursively construct sets $A_{1}(x),\dots,A_{j}(x)$ for $x\ge 0$ as follows. Initially, make all $j$ sets to be empty. Consider the ball at site $x$ with color $\xi_{x}$ (we may simply call it the `ball $\xi_{x}$') is inserted into the carrier $\Gamma_{x-1}$. There are $j$ positions in $\Gamma_{x-1}$ at which $\xi_{x}$ can be placed after the insersion, and let $r(x)\in \{1,\dots,j\}$ denote that position. Note that $r(x)<j$ if and only if $\xi_{x}>\min \Gamma_{x}$ if and only if $E_{j}$ increase by one. Now define $A_{1}(x),\dots,A_{j}(x)$ as follows: For $i=1,\dots,j$,
			\begin{align}
				&	\textup{If $r(x)<j$:} 	
				\hspace{1cm} A_{i}(x) = 
				\begin{cases}
					A_{i}(x-1) \cup \{ x \} & \textup{if $r(x)=i$} \\ 
					A_{i}(x-1) & \textup{if $r(x)\ne i$},
				\end{cases}
				\\
				&	\textup{If $r(x)=j$:} 	
				\hspace{1cm} A_{i}(x) = 
				\begin{cases}
					A_{i-1 \, (\textup{mod $j$})}(x-1) \cup \{ x \} & \textup{if $i=j$} \\ 
					A_{i-1 \, (\textup{mod $j$})}(x-1) & \textup{if $i\ne j$},
				\end{cases}
			\end{align}
			That is, if the energy $E_{j}$ increases by inserting the ball $\xi_{x}$ into the carrier $\Gamma_{x-1}$, which occurs exaclty when $r(x)<j$, we append $x$ to the set $A_{i}(x-1)$ where the new ball $\xi_{x}$ is placed at in $\Gamma_{x-1}$. Otherwise, the new ball $\xi_{x}$ is inserted in position $j$, and all the other balls are shifted to the left by one, while the ball at position $1$ is dropped out. In this case, we first shift the indices of all sets $A_{1}(x-1),\dots, A_{j}(x-1)$ by $-1$ modulo $j$, and then append $x$ to the set with index $j$ (previously of index $1$).

			Then clearly $A_{i}$'s are disjoint and partitions $\mathbb{N}$. Moreover, we claim that it has the required properties. Indeed, suppose that the energy $E_{j}$ increases when inserting the ball $\xi_{x}$ into the carrier $\Gamma_{x-1}$, i.e., $\xi_{x}>\min \Gamma_{x-1}$. Then $\xi_{x}$ replaces some ball $\xi_{y}$ (possibly 0) in $\Gamma_{x-1}$. Then necessarily $\xi_{y}<\xi_{x}$. Moreover, if $\xi_{x}$  is inserted in the $i$th position in $\Gamma_{x-1}$, then the ball $\xi_{y}$ it is replacing should also be in the $i$th position in $\Gamma_{x-1}$. By construction, we have $y,x\in A_{i}$. So $(y,x)$ is an ascent in $A_{i}$,  as desired.

			For the other direction, suppose that $R_{j}(\xi)$ is achieved by a collection of disjoint sets $A'_1,\cdots, A'_{j}$ that is different from the sets $A_{1},\cdots,A_{j}$ computed by the carrier process. Find the first place that they differ, say that $x$ belongs to $A_i$ but to $A'_{i^{*}}$ for $i^{*} \neq i$. Then perform the following surgery: let
			\begin{align}
				A''_\ell = 
				\begin{cases}
					( [1, x] \cap A_i ) \cup ( (x, \infty) \cap A'_{i^{*}} )  & \text{if $\ell=i$}\\
					( [1, x] \cap A_{i^{*}}) \cup ( (x, \infty) \cap A'_i ) & \text{if $\ell=i^{*}$} \\
					A'_{\ell} &\text{otherwise}.
				\end{cases}
			\end{align}
			Then by construction, this new collection of sets $A''_1, \cdots, A''_{j}$ has at least as many ascents as the $A'$-sequences do, and the point of disagreement with the $A$'s is moved later. Therefore repeating this process eventually produces the sets $A_1, \cdots, A_k$, and does not decrease the number of ascents. This shows $R_{j}(\xi)\le E_{j}(\xi)$, as desired.

			\vspace{0.1cm}
			\item[(iii)] Let $\Lnew:=L_{j}(\xi')$. We wish to show $L_{j}=\Lnew$. We begin by showing that $L_{j} \leq \Lnew$. In the original system $\xi$, fix a set of $k$ non-interlacing decreasing subsequences whose sum of penalized lengths is the maximum value $L_{j}$. We will produce a set of non-interlacing decreasing subsequences in $\xi'$ that have the same sum of penalized lengths. We call the unique ball of color $\kappa$ in $\xi'$ by simply $\kappa$. Suppose $\kappa$ is in position $a$, and that positions $a + 1, a + 2, \ldots, b - 1$ have balls in them, but that position $b$ is empty; let $I = \{a, \cdots, b - 1\}$.  There are cases, depending on two different questions: whether $\kappa$ is part of a decreasing subsequence, or is in the interval spanned by a decreasing subsequence, or neither; and whether there is a decreasing subsequence whose interval spans $b$, or one that ends in $I$ with no other sequence that spans $b$, or neither.
			
			If $\kappa$ belongs to a decreasing subsequence, it is the largest entry.  Therefore removing it decreases the length by $1$ and does not add a penalty (because the gap created is not in the interior of any remaining sequence).  If $\kappa$ is in the interval spanned by a decreasing subsequence but doesn't belong to it, removing $\kappa$ introduces a gap and so penalizes the length of that sequence by $1$.  If neither holds, removing $\kappa$ does not change the penalized lengths of any subsequences.  Adding $1$ to every ball label does not change the penalized lengths of any subsequences.  If a sequence spans $b$ then inserting the new ball $1$ removes a gap from that sequence, so increases its penalized length by $1$.  If a sequence ends in $I$ and no subsequence spans $b$, then the $1$ inserted in position $b$ can be appended to this sequence; there are no gaps in $I$, so this increases the penalized length by $1$.  And if neither holds, then inserting $1$ does not change the penalized lengths of any of the subsequences.  Then, it is enough to observe that in either of the cases that result in a decrease of $1$, it is necessarily the case that some sequence ends in $I$ or spans $b$.  Thus, $\Lnew \geq L_{j}$, as claimed.
			
			Finally, to show that actually $\Lnew = L_{j}$, we apply the ``reverse-complement'' operation, reversing the order of $\ZZ$ and the order of the labels.  This preserves decreasing subsequences, the non-interlacing relation between them, and their penalized lengths; moreover, one time-step in the reverse-complement is exactly the reverse-complement of one inverse time-step in the original. Thus also $\Lnew \leq L_{j}$.  This shows $L_{j}=\Lnew$, as desired. 
		\end{description} 
	\end{proof}

	%\vspace{-0.7cm}
	
	\section{Open questions and final remarks}
	In this section, we discuss some open problems and future directions.

	\vspace{0.1cm}
	\noindent \textbf{Two-sided limiting shape of the Young diagrams.} Many of the known results in scaling limits of invariant Young diagrams of randomized BBS (\cite{levine2020phase, kuniba2020large, kuniba2018randomized} and the present paper) concern rescaling of the first finite rows or columns. Is it possible to jointly scale the rows and columns and obtain the proper two-sided limiting shape of the Young diagram as in the case of the Plancherel measure \cite{kerov1988combinatorics} \cite{ivanov2002kerov}? This question is not entirely obvious since the top rows (soliton numbers) obey the laws of large numbers, whereas the top columns (soliton lengths) obey extreme value statistics.

	\vspace{0.2cm}
	\noindent \textbf{Column length scaling of higher order invariant Young diagrams.} The $\kappa$-color BBS is known to have $\kappa$-tuple of invariant Young diagrams, where the `higher order' Young diagrams describe the internal degrees of the freedom of the solitons \cite{kuniba2020large}. It is our future work to extend the methods and results in the present paper for the first-order Young diagram of the $\kappa$-color BBS into higher-order Young diagrams.

	\vspace{0.2cm}
	\noindent \textbf{Generalization to discrete KdV.} One of the most well-known integrable nonlinear partial differential equations is the Korteweg-de Vries (KdV) equation: 
	\begin{align}
		u_{t} + 6uu_{t} + u_{xxx} = 0,
	\end{align}
	where $u = u(x, t)$ is a function of two continuous parameters $x$ and $t$, and the lower indexes denote derivatives with respect to the specified variables. In 1981, Hirota \cite{hirota1981discrete} introduced the following discrete KdV (dKdV) equation that arises from KdV by discretizing space and time:
	\begin{align}\label{eq:dKdV}
		y_{k}^{t} + \frac{\delta}{y^{t}_{i+1}} = \frac{\delta}{y_{k}^{t+1}} + y^{t+1}_{k+1}.
	\end{align} 
	A further discretization of the continuous box state in dKdV leads to the ultradiscrete KdV (udKdV) equation, which corresponds to the $\kappa=1$ BBS by Takahashi-Satsuma \cite{takahashi1990soliton}:
	\begin{align}
		U_{n}^{t+1} = \min\left( 1-U^{t}_{n}, \sum_{k=-\infty}^{n-1} (U_{k}^{t}-U_{k}^{t+1}) \right),
	\end{align}
	where $U_{k}^{t}$ denotes the number of balls at time $t$ in box $k$.
	
	The scaling limit of soliton numbers and lengths of various BBS with random initial configuration has been studied extensively \cite{levine2020phase, kuniba2020large, kuniba2018randomized}, including the present paper. Hence a natural open question is to generalize the similar program to the case of discrete KdV (as opposed to ultradiscrete). For instance, suppose we initialize dKdV \eqref{eq:dKdV} so that the first $n$ box states are independent $\textup{Exp}(1)$ random variables and evolve the system until solitons come out. %Evolve the system forward until solitons come out.
	What is the scaling limit of the soliton lengths and numbers as $n\rightarrow \infty$? Can we at least obtain estimates on their expectation? These are much harder questions for dKdV because not everything decomposes into solitons: just like in the usual KdV, there is chaotic ``radiation'' left behind.

	\vspace{0.3cm}
	\section*{Acknowledgments}
	JBL was supported in part by an ORAU Powe award and a grant from the Simons Foundation (634530). HL was partially supported by NSF grants DMS-2206296 and DMS-2010035. PP was supported by the NSF Career grant DMS 1351590 and NSF RTG grant DMS 1745638. AS was partially Simons Foundation MP-TSM-00002716. We are grateful to Emily Gunawan, Olivia Fugikawa, and David Zeng for spotting an error in our proof of Prop. \ref{prop:invariants_proof} \textbf{(ii)} in an earlier version and suggesting a possible fix. We are also grateful to Russ Williams for helpful discussions on SRBM.

	\vspace{0.3cm}
	
	\small{
		\bibliographystyle{amsalpha}
		\bibliography{mybib}
	}
	
\end{document}